\documentclass[noinfoline]{imsart}

\usepackage{amsthm,amsmath}

\usepackage{amssymb}
\usepackage{amsfonts}
\usepackage{graphicx}
\usepackage[scanall]{psfrag}

\startlocaldefs

\newcommand{\NN}{\mathbb N}
\newcommand{\RR}{\mathbb R}

\newcommand{\1}{1}

\newcommand{\SA}{\mathcal A}
\newcommand{\SE}{\mathcal E}

\newcommand{\SB}{\mathcal B}
\newcommand{\SC}{\mathcal C}
\newcommand{\SF}{\mathcal F}
\newcommand{\SL}{\mathcal L}

\newcommand{\SSS}{\mathcal S}

\newcommand{\dx}[1]{\frac{\partial}{\partial x_{#1}}}
\newcommand{\ddx}[1]{\frac{\partial^2}{\partial x_{#1}^2}}

\newcommand{\dfdx}[2]{\frac{\partial^2 #1}{\partial x_{#2}^2}}

\newcommand{\norm}[1]{\parallel\! #1 \!\parallel_{\infty}}
\newcommand{\bignorm}[1]{\left| \left| #1 \right| \right|_{\infty}}
\newcommand{\normx}[1]{\parallel\! #1 \!\parallel}

\newcommand{\ainorm}[2]{| #1 |_{\alpha,#2}}
\newcommand{\wabsv}[1]{| #1 |_{\SC_w^{\alpha}}}
\newcommand{\normw}[1]{\parallel\! #1 \!\parallel_{\SC_w^{\alpha}}}
\newcommand{\bignormw}[1]{\left| \left| #1 \right| \right|_{\SC_w^{\alpha}}}

\newcommand{\picturefig}[4]{
\begin{figure}
\begin{center}
\includegraphics[scale=#1]{#2}
\end{center}
\caption{#3}
\label{equ:#4}
\end{figure}}

\DeclareMathOperator{\prodm}{\otimes}

\newtheorem{thm}{Theorem}[section]
\newtheorem{pro}[thm]{Proposition}
\newtheorem{lem}[thm]{Lemma}
\newtheorem{hyp}[thm]{Hypothesis}
\newtheorem{conv}[thm]{Convention}
\newtheorem{rmk}[thm]{Remark}

\theoremstyle{definition}
\newtheorem{defi}[thm]{Definition}
\newtheorem{nota}[thm]{Notation}
\newtheorem*{note}{Note} 

\endlocaldefs
\begin{document}
%
% ============================================================
%
\begin{frontmatter}
\title{Degenerate Stochastic Differential Equations for Catalytic Branching Networks}
\title{\'Equations Diff\'erentielles Stochastiques D\'eg\'ener\'ees pour des R\'eseaux avec Branchement Catalytique}
\runtitle{Branching Networks}
\author{\fnms{Sandra} \snm{Kliem} \ead[label=e1]{kliem@math.ubc.ca} \thanksref{t1}} \\
\thankstext{t1}{The work of the author was supported by a ``St John's College Reginald and Annie Van Fellowship'', a ``University of BC Graduate Fellowship'' and an NSERC Discovery Grant.}
\runauthor{S. Kliem}
\address{Department of Mathematics, UBC, \\
1984 Mathematics Road, Vancouver, BC V6T1Z2, Canada \\
\printead{e1}}
\affiliation{University of British Columbia}
\bf{

This is a more detailed version. In the original version, proofs mimicking proofs of \cite{r6} were often omitted. Here those calculations are added. The wording is typically as in \cite{r6}.

}
\begin{abstract}
Uniqueness of the martingale problem corresponding to a degenerate SDE which models catalytic branching networks is proven. This work is an extension of the paper by Dawson and Perkins \cite{r6} to arbitrary catalytic branching networks. As part of the proof estimates on the corresponding semigroup are found in terms of weighted H\"older norms for arbitrary networks, which are proven to be equivalent to the semigroup norm for this generalized setting. 
\newline

On prouve l'unicit\'e d'un probl\`eme de martingale correspondant \`a une EDS d\'eg\'ener\'ee, qui appara\^it comme un mod\`ele de r\'eseaux avec branchement catalytique. Ce travail est une extension des r\'esultats de Dawson et Perkins \cite{r6} au cas de r\'eseaux g\'en\'eraux. On obtient en particulier des estim\'ees pour le semi-groupe des r\'eseaux g\'en\'eraux, sous forme de normes de H\"older pond\'er\'ees; et on \'etablit l'\'equivalence de ces normes avec des normes de semi-groupe dans ce contexte g\'en\'eral.
\end{abstract}
\end{frontmatter}
\pagebreak
%
% ============================================================
% ============================================================
%
\section{INTRODUCTION}
%
% ============================================================
%
\subsection{Catalytic branching networks}
%
% ============================================================

In this paper we investigate weak uniqueness of solutions to the following system of stochastic differential equations (SDEs): 
For $j \in R \subset \{1,\ldots,d\}$ and $C_j \subset \{1,\ldots,d\} \backslash \{j\}$:
\begin{equation} \label{equ:SDE_1}
  dx_t^{(j)} = b_j(x_t) dt + \sqrt{ 2 \gamma_j(x_t) \left( \sum_{i \in C_j} x_t^{(i)} \right) x_t^{(j)} } dB_t^j 
\end{equation}
and for $j \notin R$
\begin{equation} \label{equ:SDE_2}
  dx_t^{(j)} = b_j(x_t) dt + \sqrt{ 2 \gamma_j(x_t) x_t^{(j)} } dB_t^j. 
\end{equation}
Here $x_t \in \RR_+^d$ and $b_j, \gamma_j, j=1,\ldots,d$ are H\"older-continuous functions on $\RR_+^d$ with $\gamma_j(x) > 0$, and $b_j(x) \geq 0$ if $x_j=0$.

The degeneracies in the covariance coefficients of this system make the investigation of uniqueness a challenging question. Similar results have been proven in \cite{r1} and \cite{r4} but without the additional singularity $\sum_{i \in C_j} x_t^{(i)}$ in the covariance coefficients of the diffusion. Other types of singularities, for instance replacing the additive form by a multiplicative form $\prod_{i \in C_j} x_t^{(i)}$, are possible as well, under additional assumptions on the structure of the network (cf. Remark \ref{equ:rmk_on_multiplicative_setup} at the end of Subsection \ref{equ:subsection_1_3}).

The given system of SDEs can be understood as a stochastic analogue to a system of ODEs for the concentrations $y_j, j = 1,\ldots,d$ of a type $T_j$. Then $y_j / \dot{y}_j$ corresponds to the rate of growth of type $T_j$ and one obtains the following ODEs (see \cite{r8}): for independent replication $\dot{y}_j = b_j y_j$, autocatalytic replication $\dot{y}_j = \gamma_j y_j^2$ and catalytic replication $\dot{y}_j = \gamma_j \bigl( \sum_{i \in C_j} y_i \bigr) y_j$. In the catalytic case the types $T_i, i \in C_j$ catalyze the replication of type $j$, i.e. the growth of type $j$ is proportional to the sum of masses of types $i, i \in C_j$ present at time $t$. 

An important case of the above system of ODEs is the so-called hypercycle, firstly introduced by Eigen and Schuster (see \cite{r7}). It models hypercyclic replication, i.e. $\dot{y}_j = \gamma_j y_{j-1} y_j$ and represents the simplest form of mutual help between different types.

The system of SDEs can be obtained as a limit of branching particle systems. The growth rate of types in the ODE setting now corresponds to the branching rate in the stochastic setting, i.e. type $j$ branches at a rate proportional to the sum of masses of types $i, i \in C_j$ at time $t$. 

The question of uniqueness of equations with non-constant coefficients arises already in the case $d=2$ in the renormalization analysis of hierarchically interacting two-type branching models treated in \cite{r5}. The consideration of successive block averages leads to a renormalization transformation on the diffusion functions of the SDE
\[ dx_t^{(i)} = c \left( \theta_i - x_t^{(i)} \right) dt + \sqrt{ 2 g_i(x_t) } dB_t^i, i=1,2 \]
with $\theta_i \geq 0, i=1,2$ fixed. Here $g=(g_1,g_2)$ with $g_i(x) = x_i \gamma_i(x)$ or $g_i(x) = x_1 x_2 \gamma_i(x)$, $i=1,2$ for some positive continuous function $\gamma_i$ on $\RR_+^2$. The renormalization transformation acts on the diffusion coefficients $g$ and produces a new set of diffusion coefficients for the next order block averages. To be able to iterate the renormalization transformation indefinitely a subclass of diffusion functions has to be found that is closed under the renormalization transformation. To even define the renormalization transformation one needs to show that the above SDE has a unique weak solution and to iterate it we need to establish uniqueness under minimal conditions on the coefficients.

This paper is an extension of the work done in Dawson and Perkins \cite{r6}. The latter, motivated by the stochastic analogue to the hypercycle and by \cite{r5}, proved weak uniqueness in the above mentioned system of SDEs (\ref{equ:SDE_1}) and (\ref{equ:SDE_2}), where (\ref{equ:SDE_1}) is restricted to
\[ dx_t^{(j)} = b_j(x_t) dt + \sqrt{ 2 \gamma_j(x_t) x_t^{(c_j)} x_t^{(j)} } dB_t^j, \]
i.e. $C_j = \{ c_j \}$ and (\ref{equ:SDE_2}) remains unchanged. This restriction to at most one catalyst per reactant is sufficient for the renormalization analysis for $d=2$ types, but for more than $2$ types one will encounter models where one type may have two catalysts. The present work overcomes this restriction and allows consideration of general multi-type branching networks as envisioned in \cite{r5}, including further natural settings such as competing hypercycles (cf. \cite{r7} page 55 resp. \cite{r8}, p. 106). In particular, the techniques of \cite{r6} will be extended to the setting of general catalytic networks.

Intuitively it is reasonable to conjecture uniqueness in the general setting as there is less degeneracy in the diffusion coefficients; $x_t^{(c_j)}$ changes to $\sum_{i \in C_j} x_t^{(i)}$, all coordinates $i \in C_j$ have to become zero at the same time to result in a singularity. 

For $d=2$ weak uniqueness was proven for a special case of a mutually catalytic model ($\gamma_1=\gamma_2=\mbox{const.}$) via a duality argument in \cite{r9}. Unfortunately this argument does not extend to the case $d>2$.
%
% ============================================================
%
\subsection{Comparison with Dawson and Perkins \cite{r6}} 
%
% ============================================================

The generalization to arbitrary networks results in more involved calculations. The most significant change is the additional dependency among catalysts. In \cite{r6} the semigroup of the process under consideration could be decomposed into groups of single vertices and groups of catalysts with their corresponding reactants (see Figure \ref{equ:compare}). Hence the main part of the calculations in \cite{r6}, where bounds on the semigroup are derived, i.e. Section 2 of \cite{r6} (``Properties of the basic semigroups''), could be reduced to the setting of a single vertex or a single catalyst with a finite number of reactants. In the general setting this strategy is no longer available as one reactant is now allowed to have multiple catalysts (see again Figure \ref{equ:compare}). As a consequence we shall treat all vertices in one step only. This results in more work in Section 2, where bounds on the given semigroup are now derived directly. 
%
% ------------------------------------------------------------
%
\picturefig{0.6}{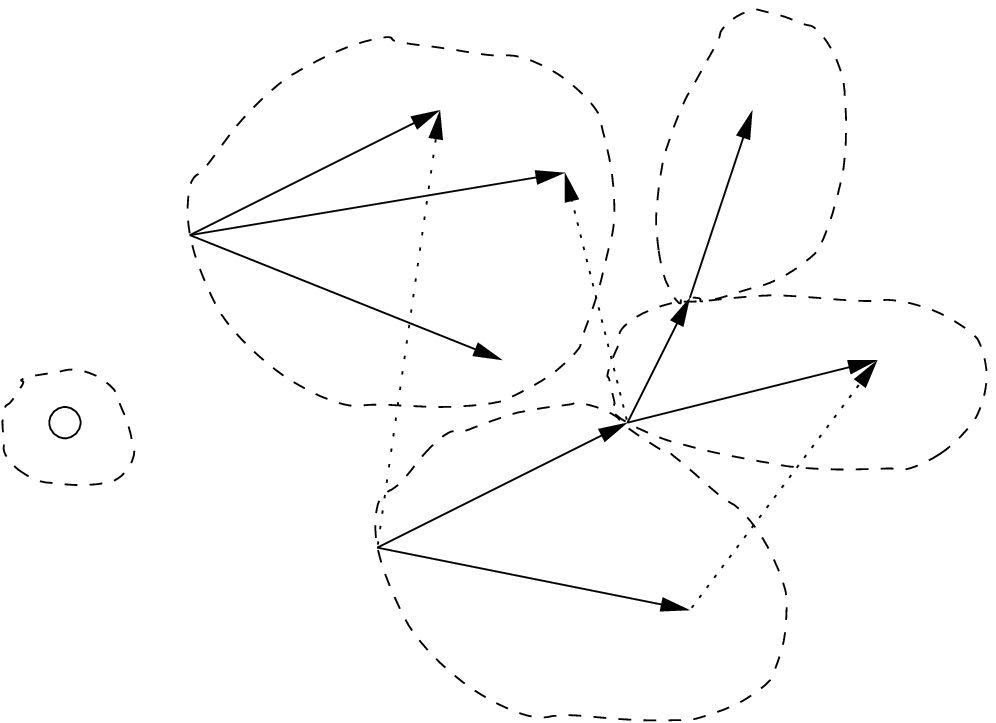}{Decomposition from the catalyst's point of view: Arrows point from vertices $i \in N_C$ to vertices $j \in R_i$. Separate points signify vertices $j \in N_2$. The dotted arrows signify arrows which are only allowed in the generalized setting and thus make a decomposition of the kind used in \cite{r6} inaccessible.}{compare} 
%
% ------------------------------------------------------------
%

We also employ a change of perspective from reactants to catalysts. In \cite{r6} every reactant $j$ had one catalyst $c_j$ only (and every catalyst $i$ a set of reactants $R_i$). For the general setting it turns out to be more efficient to consider every catalyst $i$ with the set $R_i$ of its reactants. In particular, the restriction from $R_i$ to $\bar{R}_i$, including only reactants whose catalysts are all zero, turns out to be crucial for later definitions and calculations. It plays a key role in the extension of the definition of the weighted H\"older norms to general networks (see Subsection \ref{equ:section_1_4}). 

Changes in one catalyst indirectly impact other catalysts now via common reactants, resulting for instance in new mixed partial derivatives. As a first step a representation for the semigroup of the generalized process had to be found (see (\ref{equ:eighteen})). In \cite{r6}, (12) the semigroup could be rewritten in a product form of semigroups of each catalyst with its reactants. Now a change in one catalyst resp. coordinate of the semigroup impacts in particular the local covariance of all its reactants. As the other catalysts of this reactant also appear in this coefficient, a decomposition becomes impossible. Instead the triangle inequality has to be often used to express resulting multi-dimensional coordinate changes of the function $G$, which is closely related with the semigroup representation (see (\ref{equ:nineteen})), via one-dimensional ones. As another important tool Lemma \ref{equ:lemma_13} was developed in this context.

The ideas of the proofs in \cite{r6} often had to be extended. Major changes can be found in the critical Proposition \ref{equ:proposition_23} and its associated Lemmas (especially Lemma \ref{equ:lemma_26}). The careful extension of the weighted H\"older norms to arbitrary networks had direct impact on the proofs of Lemma \ref{equ:lemma_18} and Theorem \ref{equ:theorem_19}.
%
% ============================================================
%
\subsection{The model} 
%
% ============================================================

Let a branching network be given by a directed graph $(V,\SE)$ with vertices $V = \{1,\ldots,d \}$ and a set of directed edges $\SE = \{ e_1,\ldots,e_k \}$. The vertices represent the different types, whose growth is under investigation, and $(i,j) \in \SE$ means that type $i$ ``catalyzes'' the branching of type $j$. As in \cite{r6} we continue to assume:
%
% ------------------------------------------------------------
%
\begin{hyp} \label{equ:hypothesis_1}
$(i,i) \notin \SE$ for all $i \in V$. 
\end{hyp}
%
% ------------------------------------------------------------

Let $C$ denote the set of catalysts, i.e. the set of vertices which appear as the $1$st element of an edge and $R$ denote the set of reactants, i.e. the set of vertices that appear as the $2$nd element of an edge. 

For $j \in R$, let 
\[ C_j = \{ i: (i,j) \in \SE \} \]
be the set of catalysts of $j$ and for $i \in C$, let
\[ R_i = \{ j: (i,j) \in \SE \} \]
be the set of reactants, catalyzed by $i$. If $j \notin R$ let $C_j = \emptyset$ and if $i \notin C$, let $R_i = \emptyset$.

We shall consider the following system of SDEs: \\
For $j \in R$:
\[ dx_t^{(j)} = b_j(x_t) dt + \sqrt{ 2 \gamma_j(x_t) \left( \sum_{i \in C_j} x_t^{(i)} \right) x_t^{(j)} } dB_t^j \]
and for $j \notin R$
\[ dx_t^{(j)} = b_j(x_t) dt + \sqrt{ 2 \gamma_j(x_t) x_t^{(j)} } dB_t^j. \]

Our goal will be to show the weak uniqueness of the given system of SDEs.
%
% ============================================================
%
\subsection{Statement of the main result} 
%
% ============================================================

In what follows we shall impose additional regularity conditions on the coefficients of our diffusions, similar to the ones in Hypothesis 2 of \cite{r6}, which will remain valid unless indicated to the contrary. $|x|$ is the Euclidean length of $x \in \RR^d$ and for $i \in V$ let $e_i$ denote the unit vector in the $i$th direction. 
%
% ------------------------------------------------------------
%
\begin{hyp} \label{equ:hypothesis_2}
For $i \in V$,
\begin{align*}
  & \gamma_i: \RR_+^d \rightarrow (0,\infty), \\
  & b_i: \RR_+^d \rightarrow \RR
\end{align*}
are taken to be H\"older continuous on compact subsets of $\RR_+^d$ such that $|b_i(x)| \leq c(1+|x|)$ on $\RR_+^d$, and
\[ \begin{cases}
  & b_i(x) \geq 0 \mbox{ if } x_i=0. \mbox{ In addition, } \cr
  & b_i(x) > 0 \mbox{ if } i \in C \cup R \mbox{ and } x_i=0.
\end{cases} \]
\end{hyp}
%
% ------------------------------------------------------------
%
\begin{defi} \label{equ:definition_3}
If $\nu$ is a probability on $\RR_+^d$, a probability $P$ on $\SC(\RR_+,\RR_+^d)$ is said to solve the martingale problem {\it MP}($\SA$,$\nu$) if under $P$, the law of $x_0(\omega) = \omega_0$ ($x_t(\omega)=\omega(t)$) is $\nu$ and for all $f \in \SC_b^2(\RR_+^d)$,
\[ M_f(t) = f(x_t) - f(x_0) - \int_0^t \SA f(x_s) ds \]
is a local martingale under $P$ with respect to the canonical right-continuous filtration $(\SF_t)$.
\end{defi}
%
% ------------------------------------------------------------
%
\begin{rmk} \label{equ:rmk_unique}
The weak uniqueness of a system of SDEs is equivalent to the uniqueness of the corresponding martingale problem (see for instance, \cite{r10}, \\ V.(19.7)).
\end{rmk}
%
% ------------------------------------------------------------
%
For $f \in \SC_b^2(\RR_+^d)$, the generator corresponding to our system of SDEs is
\begin{align*}
  \SA f(x) & = \SA^{(b,\gamma)} f(x) \\
  & = \sum_{j \in R} \gamma_j(x) \left( \sum_{i \in C_j} x_i \right) x_j f_{jj}(x) + \sum_{j \notin R} \gamma_j(x) x_j f_{jj}(x) + \sum_{j \in V} b_j(x) f_j(x).
\end{align*}
Here $f_{ij}$ is the second partial derivative of $f$ w.r.t. $x_i$ and $x_j$. 
%
% ------------------------------------------------------------

As a state space for the generator $\SA$ we shall use
\begin{equation} \label{equ:state_space}
  \SSS = \left\{ x \in \RR_+^d: \prod_{j \in R} \left( \sum_{i \in C_j} x_i + x_j \right) > 0 \right\}. 
\end{equation}
We first note that $\SSS$ is a natural state space for $\SA$. 
%
% ------------------------------------------------------------
%
\begin{lem} \label{equ:lemma_5}
If $P$ is a solution to {\it MP}($\SA$,$\nu$), where $\nu$ is a probability on $\RR_+^d$, then $x_t \in \SSS$ for all $t>0$ $P$-a.s. 
\end{lem}
%
% ------------------------------------------------------------
%
\begin{proof}
The proof follows as for Lemma 5, \cite{r6} on p. 377 via a comparison argument with a Bessel process, using Hypothesis \ref{equ:hypothesis_2}.
\end{proof} 
%
% ------------------------------------------------------------

We shall now state the main theorem which, together with Remark \ref{equ:rmk_unique} provides weak uniqueness of the given system of SDEs for a branching network.
%
% ------------------------------------------------------------
%
\begin{thm} \label{equ:theorem_4}
Assume Hypothesis \ref{equ:hypothesis_1} and \ref{equ:hypothesis_2} hold. Then for any probability $\nu$, on $\SSS$, there is exactly one solution to {\it MP}($\SA$,$\nu$).
\end{thm}
%
% ============================================================
%
\subsection{Outline of the proof} \label{equ:subsection_1_3}
%
% ============================================================

Our main task in proving Theorem \ref{equ:theorem_4} consists in establishing uniqueness of solutions to the martingale problem {\it MP}($\SA$,$\nu$). Existence can be proven as in Theorem 1.1 of \cite{r1}. The main idea in proving uniqueness consists in understanding our diffusion as a perturbation of a well-behaved diffusion and applying the Stroock-Varadhan perturbation method (refer to \cite{r11}) to it. This approach can be devided into three steps. \\ 
%
% ............................................................
%

{\it Step 1: Reduction of the problem.} We can assume w.l.o.g. that $\nu = \delta_{x^0}$. Furthermore it is enough to consider uniqueness for families of strong Markov solutions. Indeed, the first reduction follows by a standard conditioning argument (see p. 136 of \cite{r3}) and the second reduction follows by using Krylov's Markov selection theorem (Theorem 12.2.4 of \cite{r11}) together with the proof of Proposition 2.1 in \cite{r1}.

Next we shall use a localization argument of \cite{r11} (see e.g. the argument in the proof of Theorem 1.2 of \cite{r4}), which basically states that it is enough if for each $x^0 \in \SSS$ the martingale problem $MP(\tilde{\SA},\delta_{x^0})$ has a unique solution, where $b_i = \tilde{b}_i$ and $\gamma_i = \tilde{\gamma}_i$ agree on some $B(x^0,r_0) \cap \RR_+^d$. Here we used in particular that a solution never exits $\SSS$ as shown in Lemma \ref{equ:lemma_5}. 

Finally, if the covariance matrix of the diffusion is non-degenerate, uniqueness follows by a perturbation argument as in \cite{r11} (use e.g. Theorem 6.6.1 and Theorem 7.2.1). Hence consider only singular initial points, i.e. where either
\[ \Bigl\{ x_0^{(j)} = 0 \mbox{ or } \sum \limits_{i \in C_j} x_0^{(i)} = 0 \mbox{ for some } j \in R \Bigr\} \mbox{ or } \Bigl\{ x_0^{(j)} = 0 \mbox{ for some } j \notin R. \Bigr\} \]
%
% ............................................................
%

{\it Step 2: Perturbation of the generator.} Fix a singular initial point $x^0 \in \SSS$ and set (for an example see Figure \ref{equ:perturb_defs})
\begin{align*}
  N_R &= \left\{ j \in R: \sum \limits_{i \in C_j} x_i^0 = 0 \right\}; \\
  N_C &= \cup_{j \in N_R} C_j; \\
  N_2 &= V \backslash \left( N_R \cup N_C \right); \\
  \bar{R}_i &= R_i \cap N_R,
\end{align*}
i.e. in contrast to the setting in \cite{r6}, p. 327, $N_2$ can also include zero catalysts, but only those whose reactants have at least one more catalyst being non-zero.
%
% ------------------------------------------------------------
%
\picturefig{0.6}{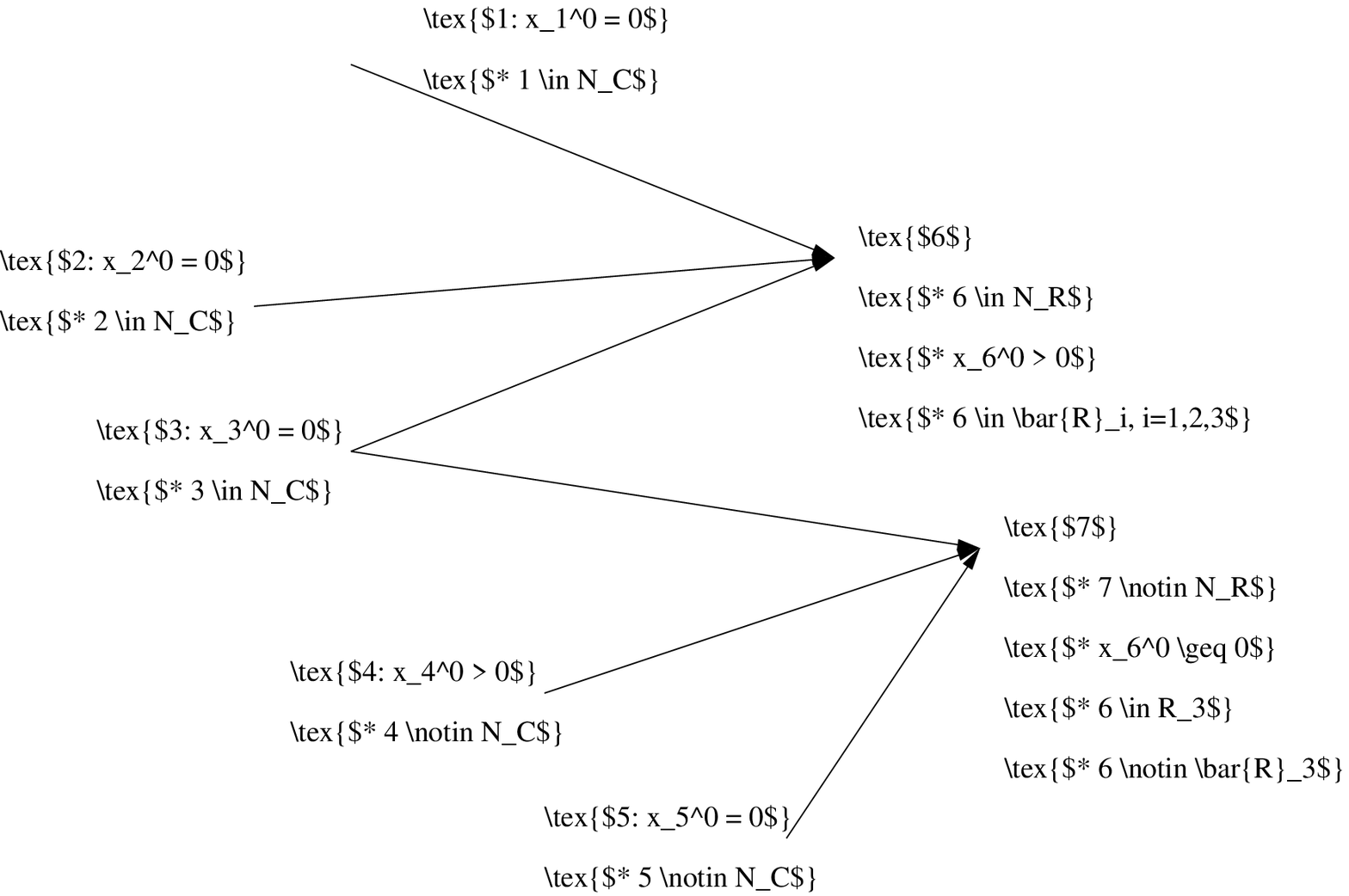}{Definition of $N_R, N_C$ and $\bar{R}_i$. The $*$'s are the implications deduced from the given setting.}{perturb_defs} 
%
% ------------------------------------------------------------
%

%
Let $Z = Z(x^0) = \{ i \in V: x_i^0 = 0 \}$ (if $i \notin Z$, then $x_i^0 > 0$ and so $x_s^{(i)} > 0$ for small $s$ a.s. by continuity). Moreover, if $x^0 \in \SSS$, then $N_R \cap Z = \emptyset$ and
\[ N_R \cup N_C \cup N_2 = V \]
is a disjoint union. 
%
% ------------------------------------------------------------
%
\begin{nota}
In what follows let 
\[ \RR^A \equiv \{ f, f: A \rightarrow \RR \} \mbox{ resp. } \RR_+^A \equiv \{ f, f: A \rightarrow \RR_+ \}. \]
for arbitrary $A \subset V$. 
\end{nota} 
%
% ------------------------------------------------------------
 
Next we shall rewrite our system of SDEs with corresponding generator $\SA$ as a perturbation of a well-understood system of SDEs with corresponding generator $\SA^0$, which has a unique solution. The state space of $\SA^0$ will be $\SSS\!\left( x^0 \right) = \SSS_0 = \{ x \in \RR^d: x_i \geq 0 \mbox{ for all } i \notin N_R \}$. 

First, we view $\left\{ x^{(j)} \right\}_{j \in N_R} \cup \left\{ x^{(i)} \right\}_{i \in N_C}$, i.e. the set of vertices with zero catalysts together with these catalysts, near its initial point $\left\{ x_j^0 \right\}_{j \in N_R} \cup \left\{ x_i^0 \right\}_{i \in N_C}$ as a perturbation of the diffusion on $\RR^{N_R} \times \RR_+^{N_C}$, which is given by the unique solution to the following system of SDEs:
\begin{align}
  & dx_t^{(j)} = b_j^0 dt + \sqrt{2 \gamma_j^0 \left( \sum \limits_{i \in C_j} x_t^{(i)} \right) } dB_t^{j}, \ x_0^{(j)} = x_j^0, \mbox{ for } j \in N_R \nonumber \\
  & \mbox{and} \label{equ:five} \\ 
  & dx_t^{(i)} = b_i^0 dt + \sqrt{2 \gamma_i^0 x_t^{(i)}} dB_t^{i}, \ x_0^{(i)} = x_i^0, \mbox{ for } i \in N_C, \nonumber
\end{align}
where for $j \in N_R$, $b_j^0 = b_j(x^0) \in \RR$ and $\gamma_j^0 = \gamma_j(x^0) x_j^0 > 0$ as $x_j^0 > 0$ if its catalysts are all zero. Also, $b_i^0 = b_i(x^0) > 0$ as $x_i^0 = 0$ for $i \in N_C$ and $\gamma_i^0 = \gamma_i(x^0) \sum_{k \in C_i} x_k^0 > 0$ if $i \in N_C \cap R$ as $i$ is a zero catalyst thus having at least one non-zero catalyst itself, or $\gamma_i^0 = \gamma_i(x^0) > 0$ if $i \in N_C \backslash R$. Note that the non-negativity of $b_i^0, i \in N_C$ ensures that solutions starting in $\{ x_i^0 \geq 0 \}$ remain there (also see definition of $\SSS_0$).

Secondly, for $j \in N_2$ we view this coordinate as a perturbation of the Feller branching process (with immigration)
\begin{equation} \label{equ:six}
  dx_t^{(j)} = b_j^0 dt + \sqrt{2 \gamma_j^0 x_t^{(j)}} dB_t^{j}, \ x_0^{(j)} = x_j^0, \mbox{ for } j \in N_2,
\end{equation}
where $b_j^0 = (b_j(x^0) \vee 0)$ (at the end of Section \ref{equ:section_3} the general case $b_j(x^0) \in \RR$ is reduced to $b_j(x^0) \geq 0$ by a Girsanov transformation), $\gamma_j^0 = \gamma_j(x^0) \sum_{i \in C_j} x_i^0 > 0$ if $j \in R$ by definition of $N_2$, i.e. at least one of the catalysts being positive, or $\gamma_j^0 = \gamma_j(x^0) > 0$ if $j \notin R$. As for $i \in N_C$, the non-negativity of $b_j^0, j \in N_2$ ensures that solutions starting in $\{ x_j^0 \geq 0 \}$ remain there (see again definition of $\SSS_0$).

Therefore we can view $\SA$ as a perturbation of the generator
\begin{equation} \label{equ:seven}
  \SA^0 = \sum_{j \in V} b_j^0 \dx{j} + \sum_{j \in N_R} \gamma_j^0 \left( \sum_{i \in C_j} x_i \right) \ddx{j} + \sum_{i \in N_C \cup N_2} \gamma_i^0 x_i \ddx{i}.
\end{equation}

The coefficients $b_i^0, \gamma_i^0$ found above for $x^0 \in \SSS$ now satisfy
\begin{equation} \label{equ:eight}
  \begin{cases}
    & \gamma_j^0 > 0 \mbox{ for all } j, \cr
    & b_j^0 \geq 0 \mbox{ if } j \notin N_R, \cr
    & b_j^0 > 0 \mbox{ if } j \in (R \cup C) \cap Z,  
  \end{cases}
\end{equation}
where
\begin{equation} \label{equ:nine} 
  N_R \cap Z = \emptyset.
\end{equation}
In the remainder of the paper we shall always assume the conditions (\ref{equ:eight}) hold when dealing with $\SA^0$ whether or not it arises from a particular $x^0 \in \SSS$ as above. As we shall see in Subsection \ref{equ:section_2_1} the $\SA^0$ martingale problem is then well-posed and the solution is a diffusion on
\begin{equation} \label{equ:ten}
  \SSS_0 \equiv \SSS\!\left( x^0 \right) = \{ x \in \RR^d: x_i \geq 0 \mbox{ for all } i \in V \backslash N_R = N_C \cup N_2 \}.
\end{equation}
%
% ............................................................
%
\begin{nota}
In the following we shall use the notation
\[ N_{C2} \equiv N_C \cup N_2. \]
\end{nota}
%
% ............................................................
%

{\it Step 3: A key estimate.} Set 
\begin{align*}
  \SB f :=& \; (\SA - \SA^0) f \\
  =& \; \sum_{j \in V} \left( \tilde{b}_j(x) - b_j^0 \right) \frac{\partial f}{\partial x_j} + \sum_{j \in N_R} \left( \tilde{\gamma}_j(x) - \gamma_j^0 \right) \left( \sum_{i \in C_j} x_i \right) \dfdx{f}{j} \\
  & \; + \, \sum_{i \in N_{C2}} \left( \tilde{\gamma}_i(x) - \gamma_i^0 \right) x_i \dfdx{f}{i}, 
\end{align*}
where 
\begin{align*}
  & \mbox{ for } j \in V,&& \tilde{b}_j(x) = b_j(x), \\
  & \mbox{ for } j \in N_R,&& \tilde{\gamma}_j(x) = \gamma_j(x) x_j, \mbox{ and } \\
  & \mbox{ for } i \in N_{C2},&& \tilde{\gamma}_i(x) = \1_{\{i \in R\}} \gamma_i(x) \sum_{k \in C_i } x_k + \1_{\{i \notin R\}} \gamma_i(x). 
\end{align*}

By using the continuity of the diffusion coefficients of $\SA$ and the localization argument mentioned in Step 1 we may assume that the coefficients of the operator $\SB$ are arbitrarily small, say less than $\eta$ in absolute value. The key step (see Theorem \ref{equ:thm_36_37}) will be to find a Banach space of continuous functions with norm $\normx{\cdot}$, depending on $x^0$, so that for $\eta$ small enough and $\lambda_0 > 0$ large enough,
\begin{equation} \label{equ:fourteen}
  \normx{ \SB R_{\lambda} f } \leq \frac{1}{2} \normx{f}, \ \forall \ \lambda > \lambda_0. 
\end{equation}
Here 
\begin{equation} \label{equ:eleven}
  R_{\lambda} f = \int_0^{\infty} e^{-\lambda s} P_s f ds
\end{equation}
is the resolvent of the diffusion with generator $\SA^0$ and $P_t$ is its semigroup.

The uniqueness of the resolvent of our strong Markov solution will then follow as in \cite{r11} and \cite{r4}. A sketch of the proof is given in Section \ref{equ:section_3}.
%
% ------------------------------------------------------------
%
\begin{rmk} \label{equ:rmk_on_multiplicative_setup}
Under additional restrictions on the structure of the branching network our results carry over to the system of SDEs, where the additive form for the catalysts is replaced by a multiplicative form as follows. For $j \in R$ we now consider
\[ dx_t^{(j)} = b_j(x_t) dt + \sqrt{ 2 \gamma_j(x_t) \left( \prod_{i \in C_j} x_t^{(i)} \right) x_t^{(j)} } dB_t^j \]
instead and for $j \notin R$
\[ dx_t^{(j)} = b_j(x_t) dt + \sqrt{ 2 \gamma_j(x_t) x_t^{(j)} } dB_t^j \]
as before. Indeed, if we impose that for all $j \in R$ we have either
\begin{align*}
  & |C_j| = 1 \mbox{ or } \\
  & |C_j| \geq 2 \mbox{ and for all } i_1 \neq i_2, i_1, i_2 \in C_j: i_1 \in C_{i_2} \mbox{ or } i_2 \in C_{i_1}, 
\end{align*}
and if we assume that Hypothesis \ref{equ:hypothesis_2} holds, then we can show a result similar to Theorem \ref{equ:theorem_4}. 

For instance, the following system of SDEs would be included.
\begin{align*}
  dx_t^{(1)} &= b_1(x_t) dt + \sqrt{ 2 \gamma_1(x_t) x_t^{(2)} x_t^{(3)} x_t^{(1)} } dB_t^1, \\
  dx_t^{(2)} &= b_2(x_t) dt + \sqrt{ 2 \gamma_2(x_t) x_t^{(3)} x_t^{(4)} x_t^{(2)} } dB_t^2, \\
  dx_t^{(3)} &= b_3(x_t) dt + \sqrt{ 2 \gamma_3(x_t) x_t^{(4)} x_t^{(1)} x_t^{(3)} } dB_t^3, \\
  dx_t^{(4)} &= b_4(x_t) dt + \sqrt{ 2 \gamma_4(x_t) x_t^{(1)} x_t^{(2)} x_t^{(4)} } dB_t^4.
\end{align*}
Note in particular, that the additional assumptions on the network ensure that at most one of either the catalysts in $C_j$ or $j$ itself can become zero, so that we obtain the same generator $\SA^0$ as in the setting of additive catalysts if we set $\gamma_j^0 \equiv \gamma_j(x^0) \prod_{i \in \{j\} \cup C_j: x_i^0 > 0} x_i^0$ (cf. the derivation of (\ref{equ:five})). 
\end{rmk}

\begin{rmk}
In \cite{r13} the H\"older condition on the coefficients was successfully removed but the restrictions on the network as stated in \cite{r6} were kept. As both \cite{r6} and \cite{r13} are based upon realizing the SDE in question as a perturbation of a well-understood SDE, one could start extending \cite{r13} to arbitrary networks by using the same generator and semigroup decomposition for the well-understood SDE as considered in this paper. 
\end{rmk}
%
% ============================================================
%
\subsection{Weighted H\"older norms and semigroup norms} \label{equ:section_1_4}%
% ============================================================

In this section we describe the Banach space of functions which will be used in (\ref{equ:fourteen}). In (\ref{equ:fourteen}) we use the resolvent of the generator $\SA^0$ with state space $\SSS_0 = \SSS\!\left( x^0 \right) = \{ x \in \RR^d: x_i \geq 0 \mbox{ for all } i \in N_{C2} \}$. Note in particular that the state space and the realizations of the sets $N_R, \bar{R}_i$ etc. depend on $x^0$.

Next we shall define the {\it Banach space of weighted $\alpha$-H\"older continuous functions on $\SSS_0$}, $\SC_w^{\alpha}(\SSS_0) \subset \SC_b(\SSS_0)$, in two steps. It will be the Banach space we look for and is a modification of the space of weighted H\"older norms used in \cite{r4}.

Let $f: \SSS_0 \rightarrow \RR$ be bounded and measurable and $\alpha \in (0,1)$. As a first step define the following seminorms for $i \in N_C$: 
\begin{align*}
  \ainorm{f}{i} &= \sup\Bigl\{ | f(x+h) - f(x) | \left( |h|^{-\alpha} x_i^{\alpha/2} \vee |h|^{-\alpha/2} \right): \\
  & \qquad \qquad \qquad \qquad |h|>0, h_k = 0 \mbox{ if } k \notin \{i\} \cup \bar{R}_i, x, h \in \SSS_0 \Bigr\}.
\end{align*}
For $j \in N_2$ this corresponds to setting
\begin{align*}
  \ainorm{f}{j} &= \sup\Bigl\{ | f(x+h) - f(x) | \left( |h|^{-\alpha} x_j^{\alpha/2} \vee |h|^{-\alpha/2} \right): \\
  & \qquad \qquad \qquad \qquad \quad h_j > 0, h_k = 0 \mbox{ if } k \neq j, x \in \SSS_0 \Bigr\}. 
\end{align*}
This definition is an extension of the definition in \cite{r6}, p. 329. In our context the definition of $\ainorm{f}{i}, i \in N_C$ had to be extended carefully by replacing the set $R_i$ (in \cite{r6} equal to the set $\bar{R}_i$) by the set $\bar{R}_i \subset R_i$. Observe that the seminorms for $i \in N_C$ and $j \in N_2$ taken together still allow changes in all coordinates (see Figure \ref{equ:resulting_sde}). The definition of $\ainorm{f}{j}, j \in N_2$ furthermore varies slightly from the one in \cite{r6}. We use our definition instead as it enables us to handle the coordinates $i \in N_C, j \in N_2$ without distinction. 
%
% ------------------------------------------------------------
%
\picturefig{0.4}{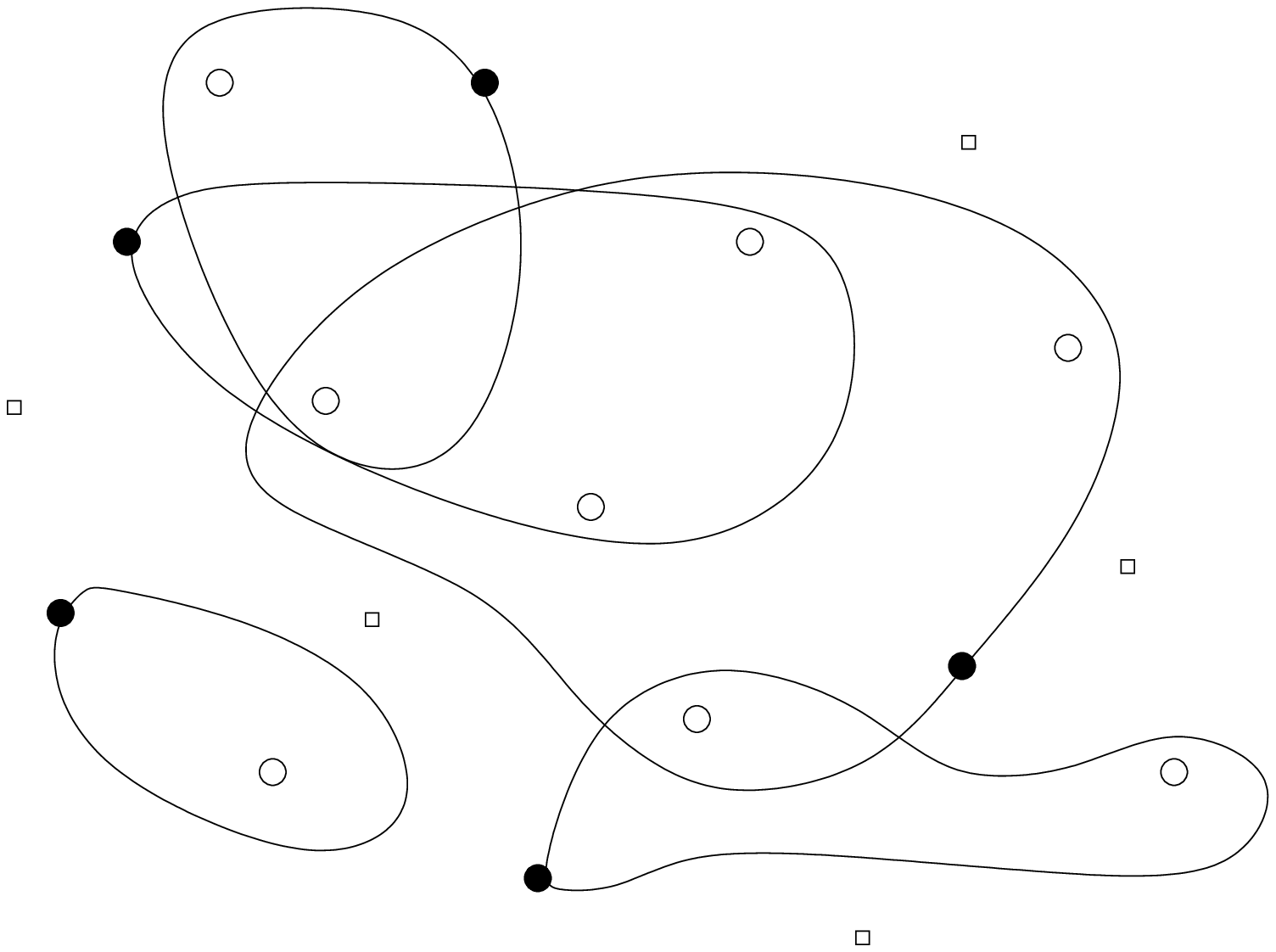}{Decomposition of the system of SDEs: unfilled circles, resp. filled circles, resp. squares are elements of $N_R$, resp. $N_C$, resp. $N_2$. The definition of $\ainorm{f}{i}, i \in N_C$ allows changes in $i$ (filled circles) and the associated $j \in \bar{R}_i$ (unfilled circles), the definition of $\ainorm{f}{j}, j \in N_2$ allows changes in $j \in N_2$ (squares). Hence changes in all vertices are possible.}{resulting_sde} 
%
% ------------------------------------------------------------
%

Secondly, set $I = N_{C2}$. Then let
\[ \wabsv{f} = \max \limits_{j \in I} \ainorm{f}{j}, \qquad \normw{f} = \wabsv{f} + \norm{f}, \]
where $\norm{f}$ is the supremum norm of $f$. $\normw{f}$ is the norm we looked for and its corresponding Banach subspace of $\SC_b(\SSS_0)$ is
\[ \SC_w^{\alpha}(\SSS_0) = \{ f \in \SC_b(\SSS_0): \normw{f} < \infty \}, \]
the Banach space of weighted $\alpha$-H\"older continuous functions on $\SSS_0$. Note that the definition of the seminorms $\ainorm{f}{j}, j \in I$ depends on $N_C, \bar{R}_i$ etc. and hence on $x^0$. Thus $\normw{f}$ depends on $x^0$ as well.

The seminorms $\ainorm{f}{i}$ are weaker norms near the spatial degeneracy at $x_i=0$ where we expect to have less smoothing by the resolvent.

Some more background on the choice of the above norms can be found in \cite{r4}, Section 2. Bass and Perkins (\cite{r4}) consider
\begin{align*}
  &\ainorm{f}{i}^* \equiv \sup\Bigl\{ | f(x+h e_i) - f(x) | |h|^{-\alpha} x_i^{\alpha/2} : h>0, x \in \RR_+^d \Bigr\}, \\
  &|f|_\alpha^* \equiv \sup_{i \leq d} \ainorm{f}{i}^* \mbox{ and } \normx{f}_\alpha^* \equiv |f|_\alpha^* + \norm{f} 
\end{align*}
instead, where $e_i$ denotes the unit vector in the $i$-th direction in $\RR^d$. They show that if $f \in \SC_b(\RR_+^d)$ is uniformly H\"older of index $\alpha \in (0,1]$, and constant outside of a bounded set, then $f \in \SC_w^{\alpha,*} \equiv \{ f \in \SC_b(\RR_+^d): \normx{f}_\alpha^* < \infty \}$. On the other hand, $f \in \SC_w^{\alpha,*}$ implies $f$ is uniformly H\"older of order $\alpha/2$.

As it will turn out later (see Theorem \ref{equ:theorem_19}) our norm $\normw{f}$ is equivalent to another norm, the so-called {\it semigroup norm}, defined via the semigroup $P_t$ corresponding to the generator $\SA^0$ of our process. As we shall mainly investigate properties of the semigroup $P_t$ on $\SC_b(\SSS_0)$ in what follows, it is not surprising that this equivalence turns out to be useful in later calculations. 

In general one defines the semigroup norm (cf. \cite{r2}) for a Markov semigroup $\{P_t\}$ on the bounded Borel functions on $D$ where $D \subset \RR^d$ and $\alpha \in (0,1)$ via
\begin{equation} \label{equ:f_alpha_definition}
  |f|_{\alpha} = \sup \limits_{t>0} \frac{\norm{P_tf - f}}{t^{\alpha/2}}, \qquad \normx{f}_{\alpha} = |f|_{\alpha} + \norm{f}. 
\end{equation}

The associated Banach space of functions is then
\begin{equation} \label{equ:S_alpha_definition}
  \SSS^{\alpha} = \{ f: D \rightarrow \RR: f \mbox{ Borel }, \normx{f}_{\alpha} < \infty \}. 
\end{equation}
%
% ------------------------------------------------------------
%
\begin{conv} \label{equ:convention_1}
Throughout this paper all constants appearing in statements of results and their proofs may depend on a fixed parameter $\alpha \in (0,1)$ and $\{b_j^0, \gamma_j^0: j \in V\}$ as well as on $|V|=d$. By (\ref{equ:eight})
\begin{equation} \label{equ:fifteen}
  M^0 = M^0(\gamma^0,b^0) \equiv \max\limits_{i \in V}\!\left\{ \gamma_i^0 \vee \left( \gamma_i^0 \right)^{-1} \vee \left| b_i^0 \right| \right\} \vee \max \limits_{i \in (R \cup C) \cap Z}\!\left( b_i^0 \right)^{-1} < \infty.
\end{equation}
Given $\alpha \in (0,1)$, $d$ and $0<M<\infty$, we can, and shall, choose the constants to hold uniformly for all coefficients satisfying $M^0 \leq M$. 
\end{conv}
%
% ============================================================
%
\subsection{Outline of the paper}
%
% ============================================================

Proofs only requiring minor adaptations from those in \cite{r6} are omitted in the original version. This paper is a more extensive version of the proofs appearing in Sections \ref{equ:section_2} and \ref{equ:section_3}.

The outline of the paper is as follows. In Section \ref{equ:section_2} the semigroup $P_t$ corresponding to the generator $\SA^0$ on the state space $\SSS_0$, as introduced in (\ref{equ:seven}) and (\ref{equ:ten}), will be investigated. We start with giving an explicit representation of the semigroup in Subsection \ref{equ:section_2_1}. In Subsection \ref{equ:section_appendix} the canonical measure $\NN_0$ is introduced which is used in Subsection \ref{equ:section_2_2} to prove existence and give a representation of derivatives of the semigroup. In Subsections \ref{equ:section_2_3} and \ref{equ:section_2_4} bounds are derived on the $L^\infty$ norms and on the weighted H\"older norms of those differentiation operators applied to $P_t f$, which appear in the definition of $\SA^0$. Furthermore, at the end of Subsection \ref{equ:section_2_3} the equivalence of the weighted H\"older norm and semigroup norm is shown. Finally, in Section~\ref{equ:section_3} bounds on the resolvent $R_\lambda$ of $P_t$ are deduced from the bounds on $P_t$ found in Section \ref{equ:section_2}. The bounds on the resolvent will then be used to obtain the key estimate (\ref{equ:fourteen}). The remainder of Section \ref{equ:section_3} illustrates how to prove the uniqueness of solutions to the martingale problem {\it MP}($\SA$,$\nu$) from this, as in \cite{r6}.
%
% ============================================================
%
\section{PROPERTIES OF THE SEMIGROUP} \label{equ:section_2}
%
% ============================================================
%  
\subsection{Representation of the semigroup} \label{equ:section_2_1}
%
% ============================================================

In this subsection we shall find an explicit representation of the semigroup $P_t$ corresponding to the generator $\SA^0$ (cf. (\ref{equ:seven})) on the state space $\SSS_0$ and further preliminary results. We assume the coefficients satisfy (\ref{equ:eight}) and Convention \ref{equ:convention_1} holds.

Let us have a look at (\ref{equ:five}) and (\ref{equ:six}) again. For $i \in N_C$ or $j \in N_2$ the processes $x_t^{(i)}$ resp. $x_t^{(j)}$ are Feller branching processes (with immigration). If we condition on these processes, the processes $x_t^{(j)}, j \in N_R$ become independent time-inhomogeneous Brownian motions (with drift), whose distributions are well understood. Thus if the associated process is denoted by $x_t = \left\{ x_t^{(j)} \right\}_{j \in N_R \cup N_{C2} } = \left\{ x_t^{(j)} \right\}_{j \in V}$, the semigroup $P_t f$ has the explicit representation
\begin{align} 
  P_t f (x) =& \; \left( \prodm \limits_{i \in N_{C2}} P_{x_i}^i \right) \!\left[ \int_{\RR^{|N_R|}} f\!\left( \left\{ z_j \right\}_{j \in N_R} , \left\{ x_t^{(i)} \right\}_{i \in N_{C2}} \right) \right. \label{equ:eighteen} \\
  & \left. \; \times \, \prod \limits_{j \in N_R} p_{\gamma_j^0 2 I_t^{(j)}} \!\left( z_j - x_j - b_j^0 t \right) dz_j \right], \nonumber
\end{align}
where $P_{x_i}^i$ is the law of the Feller branching immigration process $x^{(i)}$ on \\ $\SC(\RR_+,\RR_+)$, started at $x_i$ with generator
\begin{align*}
  & \SA^i_0 = b_i^0 \frac{\partial}{\partial x} + \gamma_i^0 x \frac{\partial^2}{\partial x^2}, \\
  & I_t^{(j)} = \int_0^t \sum \limits_{i \in C_j} x_s^{(i)} ds, 
\end{align*}
and for $y \in (0,\infty)$
\[ p_y(z) := \frac{e^{-\frac{z^2}{2y}}}{(2 \pi y)^{1/2}}. \]
%
% ------------------------------------------------------------
%
\begin{rmk}
This also shows that the $\SA^0$ martingale problem is well-posed. 
\end{rmk}
%
% ------------------------------------------------------------
%
For $(y,z) = \left( \left\{ y_j \right\}_{j \in N_R} , \left\{ z_i \right\}_{i \in N_{C2}} \right)$ and $x^{N_R} \equiv \left\{ x_j \right\}_{j \in N_R}$, let
\begin{align} 
  G(y,z) &= G_{t,x^{N_R}}(y,z) = G_{t,x^{N_R}} \!\left( \left\{ y_j \right\}_{j \in N_R} , \left\{ z_i \right\}_{i \in N_{C2}} \right) \label{equ:nineteen} \\
  &= \int_{\RR^{|N_R|}} f\!\left( \left\{ u_j \right\}_{j \in N_R} , \left\{ z_i \right\}_{i \in N_{C2}} \right) \prod \limits_{j \in N_R} p_{\gamma_j^0 2 y_j} \!\left( u_j - x_j - b_j^0 t \right) du_j. \nonumber
\end{align}
%
% ------------------------------------------------------------
%
\begin{nota}
In the following we shall use the notations
\[ E^{N_{C2}} = \left( \prodm \limits_{i \in N_{C2}} P_{x_i}^i \right), \ I_t^{N_R} = \left\{ I_t^{(j)} \right\}_{j \in N_R}, \ x_t^{N_{C2}} = \left\{ x_t^{(i)} \right\}_{i \in N_{C2}} \]
and we shall write $E$ whenever we do not specify w.r.t. which measure we integrate. 
\end{nota}
%
% ------------------------------------------------------------
%
Now (\ref{equ:eighteen}) can be rewritten as
\begin{equation} \label{equ:twenty}
  P_t f(x) = E^{N_{C2}} \!\left[ G_{t,x^{N_R}} \!\left( I_t^{N_R} , x_t^{N_{C2}} \right) \right] = E^{N_{C2}} \!\left[ G \!\left( I_t^{N_R} , x_t^{N_{C2}} \right) \right]. 
\end{equation}
%
% ------------------------------------------------------------
%
\pagebreak
\begin{lem} \label{equ:lemma_7}
Let $j \in N_R$, then 

(a) 
\begin{align*}
  & E^{N_{C2}} \!\left[ \sum \limits_{i \in C_j} x_t^{(i)} \right] = \sum \limits_{i \in C_j} \left( x_i + b_i^0 t \right), \\
  & E^{N_{C2}} \!\left[ \Biggl( \sum \limits_{i \in C_j} x_t^{(i)} \Biggr)^2 \right] = \Biggl( \sum \limits_{i \in C_j} x_i \Biggr)^2 + \sum \limits_{i \in C_j} \left( 2 \left( \sum \limits_{k \in C_j} b_k^0 + \gamma_i^0 \right) x_i \right) t \\
  & \qquad \qquad \qquad \qquad \qquad \qquad + \sum \limits_{i \in C_j} \left( \left( \sum \limits_{k \in C_j} b_k^0 + \gamma_i^0 \right) b_i^0 \right) t^2, \\
 & E^{N_{C2}} \!\left[ \Biggl( \sum \limits_{i \in C_j} \left( x_t^{(i)} - x_i \right) \Biggr)^2 \right] = \sum \limits_{i \in C_j} 2 \gamma_i^0 x_i t + \sum \limits_{i \in C_j} \!\left( \left( \sum \limits_{k \in C_j} b_k^0 + \gamma_i^0 \right) b_i^0 \right) t^2
\end{align*}
and
\[ E^{N_{C2}} \!\left[ I_t^{(j)} \right] = E^{N_{C2}} \!\left[ \int_0^t \sum \limits_{i \in C_j} x_s^{(i)} ds \right] = \sum \limits_{i \in C_j} \left( x_i t + \frac{b_i^0}{2} t^2 \right). \]

(b) 
\[ E^{N_{C2}} \!\left[ \left( I_t^{(j)} \right)^{-p} \right] \leq c(p) t^{-p} \min \limits_{i \in C_j} \!\left\{ (t+x_i)^{-p} \right\} \ \forall p>0. \]
\end{lem}
%
% ------------------------------------------------------------
%
\begin{note}
Observe that the requirement $b_i^0>0$ if $i \in (R \cup C) \cap Z$ as in (\ref{equ:eight}) is crucial for Lemma \ref{equ:lemma_7}(b). As $i \in C_j, j \in N_R$ implies $i \in C \cap Z$, (\ref{equ:eight}) guarantees $b_i^0>0$. The bound (b) cannot be applied to $i \in N_2$ in general, as (\ref{equ:eight}) only gives $b_i^0 \geq 0$ in these cases.
\end{note}
%
% ------------------------------------------------------------

{\it Proof of (a).} The first three results follow from Lemma 7(a) in \cite{r6} together with the independence of the Feller-diffusions under consideration. 

{\it Proof of (b).} Proceeding as in the proof of Lemma 7(b) in \cite{r6} we obtain
\begin{align*}
  E^{N_{C2}} \!\left[ \left( I_t^{(j)} \right)^{-p} \right] &\leq c_p e \int_0^{\infty} E^{N_{C2}} \!\left[ e^{-u^{-1} I_t^{(j)}} \right] u^{-p-1} du \\
&\leq c_p e \min \limits_{i \in C_j} \!\left\{ \int_0^{\infty} P_{x_i}^i \!\left[ e^{-u^{-1} I_t^{(i)}} \right] u^{-p-1} du \right\}
\end{align*}
as $I_t^{(j)} = \sum_{i \in C_j} \int_0^t x_s^{(i)} ds \equiv \sum_{i \in C_j} I_t^{(i)}$, where the Feller-diffusions under consideration are independent. Now we can proceed as in Lemma 7(b) of \cite{r6} to obtain the desired result. \qed %
% ------------------------------------------------------------
%
\begin{lem} \label{equ:lemma_11}
Let $G_{t,x^{N_R}}$ be as in (\ref{equ:nineteen}). Then 

(a) for $j \in N_R$
\begin{equation} \label{equ:thirtysix}
  \left| \frac{\partial G_{t,x^{N_R}}}{\partial x_j} \!\left( \left\{ y_j \right\}_{j \in N_R} , \left\{ z_i \right\}_{i \in N_{C2}} \right) \right| = \left| \frac{\partial G_{t,x^{N_R}}}{\partial x_j} (y,z) \right| \leq \norm{f} (\gamma_j^0 y_j)^{-1/2},
\end{equation}
and more generally for any $k \in \NN$, there is a constant $c_k$ such that
\[ \left| \frac{\partial^k G_{t,x^{N_R}}}{\partial x_j^k} (y,z) \right| \leq c_k \norm{f} y_j^{-k/2}. \]

(b) For $j \in N_R$
\begin{equation} \label{equ:thirtyeight}
  \left| \frac{\partial G_{t,x^{N_R}}}{\partial y_j} (y,z) \right| \leq c_1 \norm{f} y_j^{-1}.
\end{equation}
More generally there are constants $c_k, k \in \NN$ such that for $l_1, l_2, j_1, j_2 \in N_R$,
\[ \left| \frac{\partial^{m_1+m_2+k_1+k_2} G_{t,x^{N_R}}}{\partial x_{l_1}^{m_1} \partial x_{l_2}^{m_2} \partial y_{j_1}^{k_1} \partial y_{j_2}^{k_2} } (y,z) \right| \leq c_{m_1+m_2+k_1+k_2} \norm{f} y_{l_1}^{-m_1/2} y_{l_2}^{-m_2/2} y_{j_1}^{-k_1} y_{j_2}^{-k_2} \]
for all $m_1, m_2, k_1, k_2 \in \NN$. 

(c) Let $y^{N_R} = \left\{ y_j \right\}_{j \in N_R}$ and $z^{N_{C2}} = \left\{ z_i \right\}_{i \in N_{C2}}$, then for all $z^{N_{C2}}$ with $z_i \geq 0$, $i \in N_{C2}$ we have that $\left( x^{N_R},y^{N_R} \right) \rightarrow G_{t,x^{N_R}} \!\left( y^{N_R} , z^{N_{C2}} \right)$ is $\SC^3$ on $\RR^{|N_R|} \times (0,\infty)^{|N_R|}$. 
\end{lem}
%
% ------------------------------------------------------------
%
\begin{proof}
This proceeds as in \cite{r6}, Lemma 11, using the product form of the density. 
\end{proof}
%
% ------------------------------------------------------------
%
\begin{lem} \label{equ:lemma_12}
If $f$ is a bounded Borel function on $\SSS_0$ and $t>0$, then $P_t f \in \SC_b \!\left( \SSS_0 \right)$ with
\[ \left| P_t f (x) - P_t f \!\left( x' \right) \right| \leq c \norm{f} t^{-1} \left| x-x' \right|. \]
\end{lem}
%
% ------------------------------------------------------------
%
\begin{proof}
The outline of the proof is as in the proof of \cite{r6}, Lemma 12. We shall nevertheless show the proof in detail as it illustrates some basic notational issues, which will appear again in later theorems. Note in particular the frequent use of the triangle inequality resulting in additional sums of the form $\sum_{j: j \in \bar{R}_{i_0}}$ in the second part of the proof.

Using (\ref{equ:twenty}), we have for $x, x' \in \RR^{N_R}$,
\begin{align} 
  & \left| P_t f \!\left( x , x^{N_{C2}} \right) - P_t f \!\left( x' , x^{N_{C2}} \right) \right| \label{equ:fourtytwo} \\ 
  & = \left| E^{N_{C2}} \!\left[ G_{t,x} \!\left( I_t^{N_R} , x_t^{N_{C2}} \right) - G_{t,x'} \!\left( I_t^{N_R} , x_t^{N_{C2}} \right) \right] \right| \nonumber \\
  & \leq \; \norm{f} \! \sum \limits_{j \in N_R} \! \frac{| x_j-x_j' |}{\sqrt{\gamma_j^0}} E^{N_{C2}} \!\left[ \left( I_t^{(j)} \right)^{-1/2} \right] \mbox{ (by (\ref{equ:thirtysix}))} \nonumber \\
  & \leq c \norm{f} \! \sum \limits_{j \in N_R} \! \frac{| x_j-x_j' |}{\sqrt{\gamma_j^0}} t^{-1/2} \min \limits_{i \in C_j} \!\left\{ (t+x_i)^{-1/2} \right\} \mbox{ (by Lemma \ref{equ:lemma_7}(b)} ) \nonumber \\
  & \leq c \norm{f} t^{-1} \sum \limits_{j \in N_R} \! | x_j-x_j' |. \nonumber
\end{align}

Next we shall consider $x,x'=x+he_{i_0} \in \RR^{N_{C2}}$ where $i_0 \in N_{C2}$ is arbitrarily fixed. Assume $h>0$ and let $x^h$ denote an independent copy of $x^{(i_0)}$ starting at $h$ but with $b_{i_0}^0 = 0$. Then $x^{(i_0)} + x^h$ has law $P_{x_{i_0}+h}^{i_0}$ (additive property of Feller branching processes) and so if $I_h(t) = \int_0^t x_s^h ds$,
\begin{align*}
  & \left| P_t f \!\left( x^{N_R} , x' \right) - P_t f \!\left( x^{N_R} , x \right) \right| \\
  & = \left| E^{N_{C2}} \!\left[ G_{t,x^{N_R}} \biggl( \left\{ I_t^{(j)} + \1_{\{i_0 \in C_j\}} I_h(t) \right\}_{j \in N_R} , \left\{ x_t^{i} + \1_{\{i=i_0\}} x_t^h \right\}_{i \in N_{C2}} \biggr) \right. \right. \\
  & \left. \left. \qquad \qquad - G_{t,x^{N_R}} \!\left( \left\{ I_t^{(j)} \right\}_{j \in N_R} , x_t^{N_{C2}} \right) \right] \right|.
\end{align*}
For what follows it is important to observe that 
\[ \left\{ j \in N_R: i_0 \in C_j \right\} = \left\{ j: j \in \bar{R}_{i_0} \right\}, \] 
having made the definition of $\bar{R}_i$ necessary. Next we shall use the triangle inequality to first sum up changes in the $j$th coordinates (where $j \in N_R$ s.t. $i_0 \in C_j$) separately in increasing order, followed by the change in the coordinate $i_0$. If $T_h = \inf\{ t \geq 0: x_t^h = 0 \}$ we thus obtain as a bound for the above (recall that $e_k$ denotes the unit vector in the $k$th direction):
\begin{align*}
  & \sum \limits_{j: j \in \bar{R}_{i_0}} c \norm{f} E^{N_{C2}} \!\left[ I_h(t) \left( I_t^{(j)} \right)^{-1} \right] + 2 \norm{f} E[T_h>t] \\
  & = \sum \limits_{j: j \in \bar{R}_{i_0}} c \norm{f} E^{N_{C2}} \!\left[ I_h(t) \right] E^{N_{C2}} \!\left[ \left( I_t^{(j)} \right)^{-1} \right] + 2 \norm{f} E[T_h>t]
\end{align*}
by (\ref{equ:thirtyeight}) and as $\norm{G} \leq \norm{f}$ by the definition of $G$. Next we shall use that $E[T_h>t] \leq \frac{h}{ t \gamma_{i_0}^0 }$ (for reference see equation (\ref{equ:twentyeight}) in Section \ref{equ:section_appendix}). Together with Lemma \ref{equ:lemma_7}(a), (b) we may bound the above by
\[ \sum \limits_{j: j \in \bar{R}_{i_0}} c \norm{f} h t t^{-1} \min \limits_{i \in C_j} \!\left\{ (t+x_i)^{-1} \right\} + 2 \norm{f} \frac{h}{t \gamma_{i_0}^0} \leq c \norm{f} h t^{-1}. \]
The case $x'=x+he_i, i \in N_{C2}$ follows similarly. Note that for $i \in N_2$ only the second term in the above bound is nonzero as the sum is taken over an empty set ($\bar{R}_i = \emptyset$ for $i \in N_2$). Together with (\ref{equ:fourtytwo}) (recall that the $1$-norm and Euclidean norm are equivalent) we obtain the result via triangle inequality. 
\end{proof}
%
% ------------------------------------------------------------
%

Finally, we give elementary calculus inequalities that will be used below. 
%
% ------------------------------------------------------------
%
\begin{lem} \label{equ:lemma_13}
Let $g: \RR_+^d \rightarrow \RR$ be $\SC^2$. Then for all $\Delta, \Delta' > 0, y \in \RR_+^d$ and $I_1, I_2 \subset \{ 1,\ldots,d \}$,
\begin{align*}
  & \frac{ | g(y + \Delta \sum_{i_1 \in I_1} e_{i_1} + \Delta' \sum_{i_2 \in I_2} e_{i_2}) - g(y + \Delta \sum_{i_1 \in I_1} e_{i_1})}{(\Delta \Delta')} \\
  & \; \frac{ - g(y + \Delta' \sum_{i_2 \in I_2} e_{i_2}) + g(y) | }{(\Delta \Delta')} \\
  & \leq \sup \limits_{\{ y' \in \prod_{i \in \{ 1,\ldots,d \}} [y_i,y_i+\Delta+\Delta'] \}} \sum_{i_1 \in I_1} \sum_{i_2 \in I_2} \left| \frac{\partial^2}{\partial y_{i_1} \partial y_{i_2}} g(y') \right|. 
\end{align*}
Also let $f: \RR_+^d \rightarrow \RR$ be $\SC^3$. Then for all $\Delta_1, \Delta_2, \Delta_3 > 0, y \in \RR_+^d$ and $I_1, I_2, I_3 \subset \{ 1,\ldots,d \}$,
\begin{align*}
  & \frac{ | f(y + \Delta_1 \sum_{i_1 \in I_1} e_{i_1} + \Delta_2 \sum_{i_2 \in I_2} e_{i_2} + \Delta_3 \sum_{i_3 \in I_3} e_{i_3}) }{(\Delta_1 \Delta_2 \Delta_3)} \\ 
  & \; \frac{ - f(y + \Delta_1 \sum_{i_1 \in I_1} e_{i_1} + \Delta_3 \sum_{i_3 \in I_3} e_{i_3}) + f(y + \Delta_2 \sum_{i_2 \in I_2} e_{i_2}) }{(\Delta_1 \Delta_2 \Delta_3)} \\
  & \; \frac{ - f(y + \Delta_2 \sum_{i_2 \in I_2} e_{i_2} + \Delta_3 \sum_{i_3 \in I_3} e_{i_3}) + f(y + \Delta_3 \sum_{i_3 \in I_3} e_{i_3}) }{(\Delta_1 \Delta_2 \Delta_3)} \\
  & \; \frac{ - f(y + \Delta_1 \sum_{i_1 \in I_1} e_{i_1} + \Delta_2 \sum_{i_2 \in I_2} e_{i_2}) + f(y + \Delta_1 \sum_{i_1 \in I_1} e_{i_1}) - f(y) |}{(\Delta_1 \Delta_2 \Delta_3)} \\
  & \leq \sup \limits_{\{ y' \in \prod_{i \in \{ 1,\ldots,d \}} [y_i,y_i+\Delta_1+\Delta_2+\Delta_3] \}} \sum_{i_1 \in I_1} \sum_{i_2 \in I_2} \sum_{i_3 \in I_3} \left| \frac{\partial^3}{\partial y_{i_1} \partial y_{i_2} \partial y_{i_3}} f(y') \right|. 
\end{align*}
\end{lem}
%
% ------------------------------------------------------------
%
\begin{proof}
This is an extension of \cite{r6}, Lemma 13, using the triangle inequality to split the terms under consideration into sums of differences in only one coordinate at a time. 
\end{proof}
%
% ============================================================
%
\subsection{Decomposition techniques} \label{equ:section_appendix}
%
% ============================================================

In this subsection we cite relevant material from \cite{r6}, namely Lemma 8, Proposition 9 and Lemma 10. Proofs and references can be found in \cite{r6}. Further background and motivation on the processes under consideration may be found in \cite{r12}, Section II.7.

Let $\{P_x^0: x \geq 0\}$ denote the laws of the Feller branching process $X$ with no immigration (equivalently, the $0$-dimensional squared Bessel process) with generator $\SL^0 f(x) = \gamma x f''(x)$. Recall that the Feller branching process $X$ can be constructed as the weak limit of a sequence of rescaled critical Galton-Watson branching processes.

If $\omega \in \SC(\RR_+,\RR_+)$ let $\zeta(\omega) = \inf\{ t>0: \omega(t) = 0 \}$. There is a unique $\sigma$-finite measure $\NN_0$ on 
\begin{equation} \label{equ:def_C_ex}
  \SC_{ex} = \{ \omega \in \SC(\RR_+,\RR_+): \omega(0) = 0, \zeta(\omega) > 0, \omega(t) = 0 \ \forall t \geq \zeta(\omega) \} 
\end{equation}
such that for each $h>0$, if $\Xi^h$ is a Poisson point process on $\SC_{ex}$ with intensity $h \NN_0$, then
\begin{equation} \label{equ:twentyfour}
  X = \int_{\SC_{ex}} \nu \Xi^h (d\nu) \mbox{ has law } P_h^0.
\end{equation}
Citing \cite{r12}, $\NN_0$ can be thought of being the time evolution of a cluster given that it survives for some positive length of time. The representation (\ref{equ:twentyfour}) decomposes $X$ according to the ancestors at time $0$.

Moreover we also have
\begin{equation} \label{equ:twentyfive}
  \NN_0 [\nu_{\delta} > 0] = (\gamma \delta)^{-1} 
\end{equation}
and for $t>0$
\begin{equation} \label{equ:twentysix}
  \int_{\SC_{ex}} \nu_t d\NN_0(\nu) = 1.
\end{equation}
For $t>0$ let $P_t^*$ denote the probability on $\SC_{ex}$ defined by
\begin{equation} \label{equ:twentyseven}
  P_t^*[A] = \frac{ \NN_0[ A \cap \{ \nu_t > 0 \} ] }{ \NN_0[ \nu_t > 0 ] }.
\end{equation}
%
% ------------------------------------------------------------
%
\begin{lem} \label{equ:lemma_8}
For all $h>0$
\begin{equation} \label{equ:twentyeight}
  P_h^0[ \zeta > t ] = P_h^0[ X_t > 0 ] = 1 - e^{-h/(t \gamma)} \leq \frac{h}{t \gamma}.
\end{equation}
\end{lem}
%
% ------------------------------------------------------------
%
\begin{pro} \label{equ:prop_9}
Let $f: \SC(\RR_+,\RR_+) \rightarrow \RR$ be bounded and continuous. Then for any $\delta > 0$,
\[ \lim \limits_{ h \downarrow 0 } h^{-1} E_h^0 [ f(X) \1_{\{X_{\delta} > 0\}} ] = \int_{\SC_{ex}} f(\nu) \1_{\{\nu_{\delta} > 0\}} d\NN_0(\nu). \]
\end{pro}
%
% ------------------------------------------------------------

%
The representation (\ref{equ:twentyfour}) leads to the following decompositions of the processes $x_t^{(i)}, i \in N_{C2}$ that will be used below. Recall that $x_t^{(i)}$ is the Feller branching immigration process with coefficients $b_i^0 \geq 0, \gamma_i^0 > 0$ starting at $x_i$ and with law $P_{x_i}^i$. In particular, we can make use of the additive property of Feller branching processes.
%
% ------------------------------------------------------------
%
\begin{lem} \label{equ:lemma_10}
Let $0 \leq \rho \leq 1$. 

(a) We may assume
\[ x^{(i)} = X_0' + X_1, \]
where $X_0'$ is a diffusion with generator $\SA_0' f(x) = \gamma_i^0 x f''(x) + b_i^0 f'(x)$ starting at $\rho x_i, X_1$ is a diffusion with generator $\gamma_i^0 x f''(x)$ starting at $(1-\rho) x_i \geq 0$, and $X_0', X_1$ are independent. In addition, we may assume
\begin{equation} \label{equ:thirty}
  X_1(t) = \int_{\SC_{ex}} \nu_t \Xi(d\nu) = \sum \limits_{j=1}^{N_t} e_j(t),
\end{equation}
where $\Xi$ is a Poisson point process on $\SC_{ex}$ with intensity $(1-\rho) x_i \NN_0$, $\{ e_j, j \in \NN \}$ is an iid sequence with common law $P_t^*$, and $N_t$ is a Poisson random variable (independent of the $\{e_j\}$) with mean $\frac{(1-\rho) x_i}{t \gamma_i^0}$. 

(b) We also have
\begin{align*}
  \int_0^t X_1(s) ds &= \int_{\SC_{ex}} \int_0^t \nu_s ds \1_{\{\nu_t \neq 0\}} \Xi(d\nu) + \int_{\SC_{ex}} \int_0^t \nu_s ds \1_{\{\nu_t = 0\}} \Xi(d\nu) \\
  &\equiv \sum \limits_{j=1}^{N_t} r_j(t) + I_1(t)
\end{align*}
and
\begin{equation} \label{equ:thirtytwo}
  \int_0^t x_s^{(i)} ds = \sum \limits_{j=1}^{N_t} r_j(t) + I_2(t),
\end{equation}
where $r_j(t) = \int_0^t e_j(s) ds, I_2(t) = I_1(t) + \int_0^t X_0'(s) ds$. 

(c) Let $\Xi^h$ be a Poisson point process on $\SC_{ex}$ with intensity $h_i \NN_0 \ (h_i>0)$, independent of the above processes. Set $\Xi^{x+h} = \Xi + \Xi^h$ and $X_t^h = \int \nu_t \Xi^h(d\nu)$. Then
\begin{equation} \label{equ:thirtythree}
  X_t^{x+h} \equiv x_t^{(i)} + X^h(t) = \int_{\SC_{ex}} \nu_t \Xi^{x+h}(d\nu) + X_0'(t)
\end{equation}
is a diffusion with generator $\SA_0'$ starting at $x_i + h_i$. In addition
\begin{equation} \label{equ:thirtyfour}
  \int_{\SC_{ex}} \nu_t \Xi^{x+h}(d\nu) = \sum \limits_{j=1}^{N_t'} e_j(t),
\end{equation}
where $N_t'$ is a Poisson random variable with mean $((1-\rho)x_i + h_i)(\gamma_i^0 t)^{-1}$, such that $\{e_j\}$ and $(N_t,N_t')$ are independent. 

Also
\begin{equation} \label{equ:thirtyfive}
  \int_0^t X_s^{x+h} ds = \sum \limits_{j=1}^{N_t'} r_j(t) + I_2(t) + I_3^h(t),
\end{equation}
where $I_3^h(t) = \int_{\SC_{ex}} \int_0^t \nu_s ds \1_{\{\nu_t = 0\}} \Xi^h(d\nu).$
\end{lem} 
%
% ============================================================
%
\subsection{Existence and representation of derivatives of the semigroup} \label{equ:section_2_2}
%
% ============================================================

Let $\SA^0$ and $P_t$ be as in Subsection \ref{equ:section_2_1}. The first and second partial derivatives of $P_t f$ w.r.t. $x_k, x_l, k, l \in N_{C2}$ will be represented in terms of the canonical measure $\NN_0$.

Recall that by (\ref{equ:twenty})
\[ P_t f (x) = E^{N_{C2}} \!\left[ G \!\left( I_t^{N_R} , x_t^{N_{C2}} \right) \right], \]
where $I_t^{N_R} = \left\{ I_t^{(j)} \right\}_{j \in N_R}$ with $I_t^{(j)} = \int_0^t \sum \limits_{i \in C_j} x_s^{(i)} ds.$ 
%
% ------------------------------------------------------------
%
\begin{nota}
If $X \in \SC\!\left( \RR_+,\RR_+^{N_{C2}} \right)$, $\eta, \eta', \theta, \theta' \in \SC_{ex}$ (for the definition of $\SC_{ex}$ see (\ref{equ:def_C_ex})) and $k, l \in N_{C2}$, let
\begin{align*}
  & G_{t,x^{N_R}}^{+k} \biggl( X;\int_0^t \eta_s ds,\theta_t \biggr) \\
  & \equiv G_{t,x^{N_R}} \Biggl( \biggl\{ \int_0^t \sum \limits_{i \in C_j} X_s^i ds + \1_{\{k \in C_j\}} \int_0^t \eta_s ds \biggr\}_{j \in N_R} , \biggl\{ X_t^i + \1_{\{i=k\}} \theta_t \biggr\}_{i \in N_{C2}} \Biggr)
\end{align*}
and
\begin{align*}
  & G_{t,x^{N_R}}^{+k,+l} \biggl( X;\int_0^t \eta_s ds,\theta_t,\int_0^t \eta_s' ds,\theta_t' \biggr) \\
  & \equiv G_{t,x^{N_R}} \Biggl( \biggl\{ \int_0^t \sum \limits_{i \in C_j} X_s^i + \1_{\{k \in C_j\}} \eta_s + \1_{\{l \in C_j\}} \eta_s' ds \biggr\}_{j \in N_R}, \\
  & \qquad \qquad \quad \biggl\{ X_t^i + \1_{\{i=k\}} \theta_t + \1_{\{i=l\}} \theta_t' \biggr\}_{i \in N_{C2}} \Biggr).
\end{align*}
Note that if $k \in N_2$ in the above we have $\1_{\{k \in C_j\}} = 0$ for $j \in N_R$, i.e. 
\begin{align}
  G_{t,x^{N_R}}^{+k} \!\left( X;\int_0^t \! \eta_s ds,\theta_t \right) &= G_{t,x^{N_R}}^{+k} \Bigl( X;0,\theta_t \Bigr), \nonumber \\
  G_{t,x^{N_R}}^{+k,+l} \!\left( X;\int_0^t \! \eta_s ds,\theta_t,\int_0^t \! \eta_s' ds,\theta_t' \right) &= G_{t,x^{N_R}}^{+k,+l} \!\left( X;0,\theta_t,\int_0^t \! \eta_s' ds,\theta_t' \right) \label{equ:fourtythree}
\end{align}
and for $l \in N_2$
\begin{equation}
  G_{t,x^{N_R}}^{+k,+l} \!\left( X;\int_0^t \eta_s ds,\theta_t,\int_0^t \eta_s' ds,\theta_t' \right) = G_{t,x^{N_R}}^{+k,+l} \!\left( X;\int_0^t \eta_s ds,\theta_t,0,\theta_t' \right). \label{equ:fourtyfour} 
\end{equation}
If $X \in \SC\!\left( \RR_+,\RR_+^{N_{C2}} \right)$, $\nu, \nu' \in \SC_{ex}$ and $k, l \in N_{C2}$, let
\[ \Delta G^{+k}_{t,x^{N_R}} (X,\nu) \equiv G_{t,x^{N_R}}^{+k} \!\left( X;\int_0^t \nu_s ds,\nu_t \right) - G_{t,x^{N_R}}^{+k} \Bigl( X;0,0 \Bigr) \]
and
\begin{align}
  & \Delta G^{+k,+l}_{t,x^{N_R}} \!\left( X,\nu,\nu' \right) \label{equ:fourtyfive} \\
  & \equiv G_{t,x^{N_R}}^{+k,+l} \!\left( X;\int_0^t \nu_s ds,\nu_t,\int_0^t \nu_s' ds,\nu_t' \right) - G_{t,x^{N_R}}^{+k,+l} \!\left( X;0,0,\int_0^t \nu_s' ds,\nu_t' \right) \nonumber \\
  & \qquad - G_{t,x^{N_R}}^{+k,+l} \!\left( X;\int_0^t \nu_s ds,\nu_t,0,0 \right) + G_{t,x^{N_R}}^{+k,+l} \Bigl( X;0,0,0,0 \Bigr). \nonumber
\end{align}
\end{nota}
%
% ------------------------------------------------------------
%
\begin{pro} \label{equ:proposition_14}
If $f$ is a bounded Borel function on $\SSS_0$ and $t>0$ then $P_t f \in \SC_b^2(\SSS_0)$ and for $k, l \in V=\{ 1,\ldots,d \}$ 
\[ \norm{(P_t f)_{kl}} \leq c \frac{\norm{f}}{t^2}. \]
Moreover if $f$ is bounded and continuous on $\SSS_0$, then for all $k, l \in N_{C2}$
\begin{align}
  (P_t f)_k (x) &= E^{N_{C2}} \!\left[ \int \Delta G^{+k}_{t,x^{N_R}} \!\left( x^{N_{C2}},\nu \right) d\NN_0(\nu) \right], \label{equ:fourtyseven} \\
  (P_t f)_{kl} (x) &= E^{N_{C2}} \!\left[ \int \int \Delta G^{+k,+l}_{t,x^{N_R}} \!\left( x^{N_{C2}},\nu,\nu' \right) d\NN_0(\nu) d\NN_0(\nu') \right]. \label{equ:fourtyseven_b}
\end{align}
\end{pro}
%
% ------------------------------------------------------------
%
\begin{proof}
The outline of this proof is similar to the one for \cite{r6}, Proposition 14. We shall therefore only mention some changes due to the consideration of more than one catalyst at a time.  

With the help of Lemma \ref{equ:lemma_12} and using that $P_t f = P_{t/2} (P_{t/2} f)$ one can easily show that it suffices to consider bounded continuous $f$. In \cite{r6}, Proposition 14 one only proves the existence of $(P_t f)_{kl} (x)$, $k, l \in N_{C2}$ and its representation in terms of the canonical measure as in (\ref{equ:fourtyseven_b}) based on (\ref{equ:fourtyseven}). From the methods used it should then be clear how the easier formula (\ref{equ:fourtyseven}) may have been found.

Hence, let us also assume $(P_t f)_k$ exists and is given by (\ref{equ:fourtyseven}) for $k \in N_{C2}$. Let $0 < \delta \leq t$. In the first case where $\nu_{\delta}' = \nu_t = 0$, use Lemmas \ref{equ:lemma_13} and \ref{equ:lemma_11}(b) to see that for $k, l \in N_C$
\begin{align}
  & \left| \Delta G^{+k,+l}_{t,x^{N_R}} \!\left( x^{N_{C2}},\nu,\nu' \right) \right| \label{equ:fourtynine} \\
  & = \left| G_{t,x^{N_R}}^{+k,+l} \!\left( x^{N_{C2}};\int_0^t \nu_s ds,0,\int_0^\delta \nu_s' ds,0 \right) - G_{t,x^{N_R}}^{+k,+l} \!\left( x^{N_{C2}};0,0,\int_0^\delta \nu_s' ds,0 \right) \right. \nonumber \\
  & \qquad \left. - G_{t,x^{N_R}}^{+k,+l} \!\left( x^{N_{C2}};\int_0^t \nu_s ds,0,0,0 \right) + G_{t,x^{N_R}}^{+k,+l} \Bigl( x^{N_{C2}};0,0,0,0 \Bigr) \right| \nonumber \\
  & \; = \, \biggl| G_{t,x^{N_R}} \biggl( \biggl\{ \int_0^t \! \sum \limits_{i \in C_j} x_s^{(i)} ds + \1_{\{k \in C_j\}} \int_0^t \! \nu_s ds + \1_{\{l \in C_j\}} \int_0^{\delta} \! \nu_s' ds \biggr\}_{j \in N_R} \!\! , x_t^{N_{C2}} \biggr) \nonumber \\
  & \qquad - G_{t,x^{N_R}} \biggl( \biggl\{ \int_0^t \sum \limits_{i \in C_j} x_s^{(i)} ds + \1_{\{l \in C_j\}} \int_0^{\delta} \nu_s' ds \biggr\}_{j \in N_R} , x_t^{N_{C2}} \biggr) \nonumber \\
  & \qquad - G_{t,x^{N_R}} \biggl( \biggl\{ \int_0^t \sum \limits_{i \in C_j} x_s^{(i)} ds + \1_{\{k \in C_j\}} \int_0^t \nu_s ds \biggr\}_{j \in N_R} , x_t^{N_{C2}} \biggr) \nonumber \\
  & \qquad + G_{t,x^{N_R}} \biggl( \biggl\{ \int_0^t \sum \limits_{i \in C_j} x_s^{(i)} ds \biggr\}_{j \in N_R} ,x_t^{N_{C2}} \biggr) \biggr| \nonumber \\
  & \leq \sum \limits_{j_1: j_1 \in \bar{R}_k} \sum \limits_{j_2: j_2 \in \bar{R}_l} c \norm{f} \left( I_t^{(j_1)} \right)^{-1} \left( I_t^{(j_2)} \right)^{-1} \int_0^{\delta} \nu_s' ds \int_0^t \nu_s ds \nonumber
\end{align}
(compare to (49) in \cite{r6}).

For $k$ or $l \in N_2$ we obtain via (\ref{equ:fourtythree}) and (\ref{equ:fourtyfour})
\[ \left| \Delta G^{+k,+l}_{t,x^{N_R}} \!\left( x^{N_{C2}},\nu,\nu' \right) \right| = 0. \]
This is consistent with (\ref{equ:fourtynine}) if we consider the sum over an empty set to be zero (recall that $\bar{R}_k = R_k \cap N_R$ and thus $\bar{R}_k = \emptyset$ if $k \in N_2$). Hence (\ref{equ:fourtynine}) is a bound for all $k, l \in N_{C2}$. 
%
% additional part starts

If $\nu_{\delta}'=0$ and $\nu_t>0$ then by Lemma \ref{equ:lemma_11}(b) and triangle inequality we obtain for $l \in N_C$
\begin{align}
  & \left| \Delta G^{+k,+l}_{t,x^{N_R}} \left( x^{N_{C2}},\nu,\nu' \right) \right| \label{equ:fifty} \\
  & \; \leq \, \left| G_{t,x^{N_R}}^{+k,+l} \left( x^{N_{C2}};\int_0^t \nu_s ds,\nu_t,\int_0^\delta \nu_s' ds,0 \right) - G_{t,x^{N_R}}^{+k,+l} \left( x^{N_{C2}};\int_0^t \nu_s ds,\nu_t,0,0 \right) \right| \nonumber \\
  & \qquad + \left| G_{t,x^{N_R}}^{+k,+l} \left( x^{N_{C2}};0,0,\int_0^\delta \nu_s' ds,0 \right) - G_{t,x^{N_R}}^{+k,+l} (x^{N_{C2}};0,0,0,0) \right| \nonumber \\
  & \; \leq \, \sum \limits_{j: j \in \bar{R}_l} 2 c \norm{f} \left( I_t^{(j)} \right)^{-1} \int_0^{\delta} \nu_s' ds. \nonumber
\end{align}
For $l \in N_2$ we obtain via (\ref{equ:fourtythree}) and (\ref{equ:fourtyfour})
\[ \left| \Delta G^{+k,+l}_{t,x^{N_R}} \left( x^{N_{C2}},\nu,\nu' \right) \right| = 0, \]
so (\ref{equ:fifty}) is valid for $k, l \in N_{C2}$. 
A similar argument shows if $\nu_{\delta}'>0$ and $\nu_t=0$, 
\[ \left| \Delta G^{+k,+l}_{t,x^{N_R}} \left( x^{N_{C2}},\nu,\nu' \right) \right| \leq \sum \limits_{j: j \in \bar{R}_k} 2 c \norm{f} \left( I_t^{(j)} \right)^{-1} \int_0^t \nu_s ds. \]

Finally, if $\nu_{\delta}>0, \nu_t>0$, then we have the trivial bound
\[ \left| \Delta G^{+k,+l}_{t,x^{N_R}} \left( x^{N_{C2}},\nu,\nu' \right) \right| \leq 4 \norm{f}. \]
%
% additional part ends

Combining all the cases we conclude that
\begin{align}
  & \left| \Delta G^{+k,+l}_{t,x^{N_R}} (x^{N_{C2}},\nu,\nu') \right| \label{equ:g_bar} \\
  & \leq \biggl\{ \1_{\{\nu_{\delta}'=\nu_t=0\}} \left( \sum \limits_{j_1: j_1 \in \bar{R}_k} \sum \limits_{j_2: j_2 \in \bar{R}_l} \left( I_t^{(j_1)} \right)^{-1} \left( I_t^{(j_2)} \right)^{-1} \int_0^{\delta} \nu_s' ds \int_0^t \nu_s ds \right) \nonumber \\
  & \qquad + \1_{\{\nu_{\delta}'=0,\nu_t>0\}} \left( \sum \limits_{j: j \in \bar{R}_l} \left( I_t^{(j)} \right)^{-1} \int_0^{\delta} \nu_s' ds \right) \nonumber \\
  & \qquad + \1_{\{\nu_{\delta}'>0,\nu_t=0\}} \left( \sum \limits_{j: j \in \bar{R}_k} \left( I_t^{(j)} \right)^{-1} \int_0^t \nu_s ds \right) + \1_{\{\nu_{\delta}'>0,\nu_t>0\}} \biggr\} c \norm{f} \nonumber \\
  & \leq \biggl\{ \1_{\{\nu_{\delta}'=\nu_t=0\}} \left( \int_0^t x_s^{(k)} ds \right)^{-1} \left( \int_0^t x_s^{(l)} ds \right)^{-1} \int_0^{\delta} \nu_s' ds \int_0^t \nu_s ds \nonumber \\
  & \qquad + \1_{\{\nu_{\delta}'=0,\nu_t>0\}} \left( \int_0^t x_s^{(l)} ds \right)^{-1} \int_0^{\delta} \nu_s' ds \nonumber \\
  & \qquad + \1_{\{\nu_{\delta}'>0,\nu_t=0\}} \left( \int_0^t x_s^{(k)} ds \right)^{-1} \int_0^t \nu_s ds + \1_{\{\nu_{\delta}'>0,\nu_t>0\}} \biggr\} c \norm{f} \nonumber \\
  & \equiv \bar{g}_{t,\delta} \!\left( x^{N_{C2}},\nu,\nu' \right) \nonumber
\end{align}

The remainder of the proof works similar to the proof in \cite{r6}. Some minor changes are necessary in the proof of continuity from below in $x_2$ (now to be replaced by $x^{N_{C2}}$) following (59) in \cite{r6}, by considering every coordinate on its own. Also, new mixed partial derivatives appear, which can be treated similarly to the ones already appearing in the proof of Proposition 14 in \cite{r6}. 
%
% additional part starts

Let $X_.^h$ be independent of $x^{(l)}$ satisfying
\[ X_t^h = h + \int_0^t \sqrt{2 \gamma_l^0 X_s^h} dB_s', \ (h>0) \]
(i.e. $X^h$ has law $P_h^0$) so that $x^{(l)} + X^h $ has law $P_{x_l+h}^l$. Therefore (\ref{equ:fourtyseven}) together with definition (\ref{equ:fourtyfive}) implies
\begin{align}
  & \frac{1}{h} \left[ (P_t f)_k (x+he_l) - (P_t f)_k (x) \right] \label{equ:fiftyfour} \\
  & = \frac{1}{h} \int \int \int \Delta G^{+k,+l}_{t,x^{N_R}} \left( x^{N_{C2}},\nu,X^h \right) d\NN_0(\nu) dP^{N_{C2}} dP_h^0. \nonumber
\end{align}
In addition (\ref{equ:g_bar}) implies (use also (\ref{equ:twentyfive}) and (\ref{equ:twentysix}) and Lemma \ref{equ:lemma_7}(a),(b) with $p=1$ or $2$ as well as Cauchy-Schwarz' inequality in the second inequality)
\begin{align}
  & \frac{1}{h} \int \int \int \left| \Delta G^{+k,+l}_{t,x^{N_R}} \left( x^{N_{C2}},\nu,X^h \right) \right| \1_{\{X_{\delta}^h=0\}} d\NN_0(\nu) dP^{N_{C2}} dP_h^0 \label{equ:fiftyfive} \\
  & \leq c \norm{f} \left[ \1_{\{k, l \in N_C\}} E^{N_{C2}} \left( \left( \int_0^t x_s^{(k)} ds \right)^{-1} \left( \int_0^t x_s^{(l)} ds \right)^{-1} \right) \right. \nonumber \\
  & \qquad \qquad \times \frac{1}{h} E_h^0 \left( \int_0^{\delta} X_s^h ds \right) \int \int_0^t \nu_s ds d\NN_0(\nu) \nonumber \\
  & \left. \qquad + \1_{\{l \in N_C\}} E^{N_{C2}} \left( \left( \int_0^t x_s^{(l)} ds \right)^{-1} \right) \frac{1}{h} E_h^0 \left( \int_0^{\delta} X_s^h ds \right) \NN_0( \nu_t>0 ) \right] \nonumber \\
  & \leq c \norm{f} \left[ \1_{\{k, l \in N_C\}} \sqrt{ E^{N_{C2}} \left( \left( \int_0^t x_s^{(k)} ds \right)^{-2} \right) } \right. \nonumber \\ 
  & \qquad \quad \times \sqrt{ E^{N_{C2}} \left( \left( \int_0^t x_s^{(l)} ds \right)^{-2} \right) } \frac{1}{h} E_h^0 \left( \int_0^{\delta} X_s^h ds \right) \int \int_0^t \nu_s ds d\NN_0(\nu) \nonumber \\
  & \left. \quad + \1_{\{l \in N_C\}} E^{N_{C2}} \left( \left( \int_0^t x_s^{(l)} ds \right)^{-1} \right) \frac{1}{h} E_h^0 \left( \int_0^{\delta} X_s^h ds \right) \NN_0( \nu_t>0 ) \right] \nonumber \\
  & \leq c \norm{f} \left( \sqrt{ t^{-2} (t+x_k)^{-2} } \sqrt{ t^{-2} (t+x_l)^{-2} } \frac{1}{h} h \delta t + t^{-1} (t+x_l)^{-1} \frac{1}{h} h \delta t^{-1} \right) \nonumber \\
  & \leq c \norm{f} t^{-3} \delta. \nonumber
\end{align}
As $G$ is bounded and continuous on $(0,\infty)^{|N_R|} \times \RR_+^{|N_{C2}|}$, Proposition \ref{equ:prop_9} implies
\begin{align}
  & \lim \limits_{h \downarrow 0} \frac{1}{h} E_h^0 \left( \Delta G^{+k,+l}_{t,x^{N_R}} \left( x^{N_{C2}},\nu,X^h \right) \1_{\{X_{\delta}^h>0\}} \right) \label{equ:fiftysix} \\
  & = \int \Delta G^{+k,+l}_{t,x^{N_R}} \left( x^{N_{C2}},\nu,\nu' \right) \1_{\{\nu_{\delta}'>0\}} d\NN_0(\nu') \nonumber
\end{align}
for all $\delta>0$, pointwise in $(x^{N_{C2}},\nu) \in \SC\!\left( \RR_+,\RR_+^{|N_{C2}|} \right) \times \SC_{ex}$. Use (\ref{equ:g_bar}) to see that
\begin{align*}
  & \frac{1}{h} E_h^0 \left( \Delta G^{+k,+l}_{t,x^{N_R}} \left( x^{N_{C2}},\nu,X^h \right) \1_{\{X_{\delta}^h>0\}} \right) \\
  & \leq c \norm{f} \frac{ P_h^0( X_{\delta}^h > 0 ) }{h} \left[ \1_{\{k \in N_C\}} \left( \int_0^t x_s^{(k)} ds \right)^{-1} \int_0^t \nu_s ds + \1_{\{\nu_t>0\}} \right] \\
  & \leq c \norm{f} \delta^{-1} \left[ \1_{\{k \in N_C\}} \left( \int_0^t x_s^{(k)} ds \right)^{-1} \int_0^t \nu_s ds + \1_{\{\nu_t>0\}} \right],
\end{align*}
the last by (\ref{equ:twentyeight}). The final expression is integrable with respect to $P^{N_{C2}} \times \NN_0$ and so by dominated convergence we conclude from (\ref{equ:fiftysix}) that
\begin{align}
  & \lim \limits_{h \downarrow 0} \frac{1}{h} \int \int \int \Delta G^{+k,+l}_{t,x^{N_R}} \left( x^{N_{C2}},\nu,X^h \right) \1_{\{X_{\delta}^h>0\}} d\NN_0(\nu) dP^{N_{C2}} dP_h^0 \label{equ:fiftyseven} \\
  & = E^{N_{C2}} \left( \int \int \Delta G^{+k,+l}_{t,x^{N_R}} \left( x^{N_{C2}},\nu,\nu' \right) \1_{\{\nu_{\delta}'>0\}} d\NN_0(\nu) d\NN_0(\nu') \right) \nonumber
\end{align}
for all $\delta>0$. Use (\ref{equ:g_bar}) as in the derivation of (\ref{equ:fiftyfive}) to see
\begin{align}
  & E^{N_{C2}} \left[ \int \int \sup \limits_{x^{N_R}} \left| \Delta G^{+k,+l}_{t,x^{N_R}} \left( x^{N_{C2}},\nu,\nu' \right) \right| \1_{\{ \nu_{\delta}'=0 \}} d\NN_0(\nu) d\NN_0(\nu') \right] \label{equ:fiftyeight} \\
  & \leq c \norm{f} t^{-3} \delta. \nonumber
\end{align}
Use (\ref{equ:fiftyfour}), (\ref{equ:fiftyfive}), (\ref{equ:fiftyseven}) and (\ref{equ:fiftyeight}) and take $\delta \downarrow 0$ to conclude
\begin{equation} \label{equ:fiftynine}
  \frac{\partial^+}{\partial x_l^+} (P_t f)_k(x) = E^{N_{C2}} \left[ \int \int \Delta G^{+k,+l}_{t,x^{N_R}} \left( x^{N_{C2}},\nu,\nu' \right) d\NN_0(\nu) d\NN_0(\nu') \right]. 
\end{equation}
Recall from Lemma \ref{equ:lemma_11}(a) that
\[ \left| \frac{\partial G_{t,x^{N_R}}}{\partial x_j} (y) \right| \leq \norm{f} (\gamma_j^0 y_j)^{-1/2}. \]
This, together with (\ref{equ:fiftyeight}) and Lemma \ref{equ:lemma_7}(b), implies for $0 < \delta \leq t$
\begin{align*}
  & \left| E^{N_{C2}} \left[ \int \int \left( \Delta G^{+k,+l}_{t,x^{N_R}} \left( x^{N_{C2}},\nu,\nu' \right) - \Delta G^{+k,+l}_{t,\bar{x}^{N_R}} \left( x^{N_{C2}},\nu,\nu' \right) \right) d\NN_0(\nu) d\NN_0(\nu') \right] \right| \\
  & \leq c \norm{f} t^{-3} \delta + E^{N_{C2}} \left[ \int \int \1_{\{ \nu_{\delta}>0,\nu_{\delta}'>0 \}} \right. \\
  & \left. \qquad \times \sum \limits_{j \in N_R} c \norm{f} \left( I_t^{(j)} \right)^{-1/2} \left| x_j - \bar{x}_j \right| d\NN_0(\nu) d\NN_0(\nu') \right] \\
  & \leq c \norm{f} \left[ t^{-3} \delta + \delta^{-2} \sum \limits_{j \in N_R} t^{-1/2} \min \limits_{i \in C_j} \left\{ (t+x_i)^{-1/2} \right\} \left| x_j - \bar{x}_j \right| \right] \\
  & \leq c \norm{f} \left[ t^{-3} \delta + \delta^{-2} t^{-1} \left| x^{N_R} - \bar{x}^{N_R} \right| \right].
\end{align*}
(\ref{equ:twentyfive}) was used in the next to last line. By first choosing $\delta$ small and then $|x^{N_R}-\bar{x}^{N_R}|$ small, one sees that the right hand side of (\ref{equ:fiftynine}) is continuous in $x^{N_R}$, uniformly in $x^{N_{C2}} \in \RR_+^{|N_{C2}|}$. 

Next let $m \in N_{C2}$. If $x_m^n \uparrow x_m$, then we may construct $\{x^n\}$ such that $x^n \uparrow x^{(m)}$ in $\SC(\RR_+,\RR_+)$, $x^n$ with law $P_{x_m^n}^m$ and $x^{(m)}$ with law $P_{x_m}^m$ (e.g. $x^n-x^{n-1}$ has law $P_{x_m^n-x_m^{n-1}}^0$ and are independent). Then $x^{N_{C2}, n} \rightarrow x^{N_{C2}}$ in $\SC(\RR_+,\RR_+^{|N_{C2}|})$, where $x^{N_{C2}, n}$ has law $\left( \prodm \limits_{i \in N_{C2}} P_{\1(i \neq m) x_i + \1(i=m) x_m^n}^i \right)$ and $x^{N_{C2}}$ has law \\ $\left( \prodm \limits_{i \in N_{C2}} P_{x_i}^i \right)$. Now $\Delta G^{+k,+l}_{t,x^{N_R}} \left( x^{N_{C2},n},\nu,\nu' \right) \rightarrow \Delta G^{+k,+l}_{t,x^{N_R}} \left( x^{N_{C2}},\nu,\nu' \right)$ pointwise (by an elementary argument using (\ref{equ:fourtyfive}) and the continuity of $f$) and (by (\ref{equ:g_bar}))
\[ \left| \Delta G^{+k,+l}_{t,x^{N_R}} \left( x^{N_{C2},n},\nu,\nu' \right) \right| \leq \bar{g}_{t,\delta} (x^{N_{C2},1},\nu,\nu'), \]
which is integrable with respect to $P^{N_{C2}} \times \NN_0 \times \NN_0$ by Lemma \ref{equ:lemma_7}(b). Dominated convergence now shows that
\begin{align*}
  & \lim \limits_{n \rightarrow \infty} E^{N_{C2},n} \left[ \int \int \Delta G^{+k,+l}_{t,x^{N_R}} \left( x^{N_{C2}},\nu,\nu' \right) d\NN_0(\nu) d\NN_0(\nu') \right] \\
  & = E^{N_{C2}} \left[ \int \int \Delta G^{+k,+l}_{t,x^{N_R}} \left( x^{N_{C2}},\nu,\nu' \right) d\NN_0(\nu) d\NN_0(\nu') \right].
\end{align*}
A similar argument holds if $x_m^n \downarrow x_m$, so the right-hand side of (\ref{equ:fiftynine}) is also continuous in $x_m$. Combined with the above this shows that $\frac{\partial^+}{\partial x_l^+} (P_t f)_k(x)$ is continuous in $x \in \SSS_0$. An elementary calculus exercise using the continuity in $x_l$ shows this in fact equals $(P_t f)_{kl} (x)$ and so
\[ (P_t f)_{kl} (x) = E^{N_{C2}} \left[ \int \int \Delta G^{+k,+l}_{t,x^{N_R}} \left( x^{N_{C2}},\nu,\nu' \right) d\NN_0(\nu) d\NN_0(\nu') \right]. \]
This together with (\ref{equ:g_bar}) (take $\delta = t$), Lemma \ref{equ:lemma_7}(b), (\ref{equ:twentyfive}) and (\ref{equ:twentysix}) gives the upper bound
\[ \norm{(P_t f)_{kl}} \leq c \frac{\norm{f}}{t^2}. \]

Let us turn to derivatives with respect to $x_j, j \in N_R$.

Lemma \ref{equ:lemma_7}(b) and dominated convergence allows us to differentiate through the integral sign and conclude for $i,j \in N_R$ (by Lemmas 7(b) and 11(a), (b)) that
\begin{align}
  \dx{j} P_t f (x) &= E^{N_{C2}} \left[ \dx{j} G_{t,x^{N_R}} \left( I_t^{N_R} , x_t^{N_{C2}} \right) \right] \leq c \frac{\norm{f}}{t} \nonumber \\
  \frac{\partial^2}{\partial x_i \partial x_j} P_t f (x) &= E^{N_{C2}} \left[ \frac{\partial^2}{\partial x_i \partial x_j} G_{t,x^{N_R}} \left( I_t^{N_R} , x_t^{N_{C2}} \right) \right] \leq c \frac{\norm{f}}{t^2}. \label{equ:sixtyone}
\end{align}
Now use (\ref{equ:sixtyone}), Lemma \ref{equ:lemma_11}(b) and Lemma \ref{equ:lemma_7}(b) to see that $\frac{\partial^2}{\partial x_i \partial x_j} P_t f (x^{N_R},x^{N_{C2}})$ is continuous in $x^{N_R}$ uniformly in $x^{N_{C2}}$. The weak continuity of \\ $E^{N_{C2}} = ( \prodm \limits_{i \in N_{C2}} P_{x_i}^i )$ in $x^{N_{C2}}$ (e.g. by our usual coupling argument), the continuity of $\frac{\partial^2}{\partial x_i \partial x_j} G_{t,x^{N_R}} \left( y^{N_R}, y^{N_{C2}} \right)$ in $y^{N_R} \in (0,\infty)^{|N_R|}$ (see Lemma \ref{equ:lemma_11}(b)), the bound in Lemma \ref{equ:lemma_11}(b) with $m_1=m_2=1$ resp. $m_1=2,m_2=0$, and Lemma \ref{equ:lemma_7}(b) imply $\frac{\partial^2}{\partial x_i \partial x_j} P_t f (x^{N_R},x^{N_{C2}})$ is continuous in $x^{N_{C2}}$ for each $x^{N_R}$. Therefore $(P_t f)_{ij}$ is jointly continuous.

For the mixed partial, first note by Lemma \ref{equ:lemma_11}(b) that for $i,j \in N_R$
\begin{equation} \label{equ:sixtytwo}
  \left| \frac{\partial}{\partial y_i} \dx{j} G_{t,x^{N_R}} (y^{N_R},y^{N_{C2}}) \right| \leq c \norm{f} y_i^{-1} y_j^{-1/2}.
\end{equation}
Let $k \in N_{C2}$ and argue as for $(P_t f)_{kl}$, using (\ref{equ:sixtytwo}) in place of Lemma \ref{equ:lemma_11}(a), to see that
\[ (P_t f)_{kj} \left( x^{N_R},x^{N_{C2}} \right) = E^{N_{C2}} \left[ \int \dx{j} \Delta G_{t,x^{N_R}}^{+k} \left( x^{N_{C2}}; \int_0^t \nu_s ds, \nu_t \right) d\NN_0(\nu) \right]. \]
From (\ref{equ:sixtytwo}) we have for $0 < \delta \leq t$ together with the triangle inequality,
\begin{align}
  & E^{N_{C2}} \left[ \int \sup \limits_{x^{N_R}} \left| \dx{j} \Delta G_{t,x^{N_R}}^{+k} \left( x^{N_{C2}}; \int_0^t \nu_s ds, \nu_t \right) \right| \1_{\{\nu_{\delta} = 0\}} d\NN_0(\nu) \right] \label{equ:sixtythree}  \\
  & \leq c \norm{f} \sum \limits_{i: i \in \bar{R}_k} E^{N_{C2}} \left[ \left( I_t^{(i)} \right)^{-1} \left( I_t^{(j)} \right)^{-1/2} \right] \int \int_0^{\delta} \nu_s ds d\NN_0(\nu) \nonumber \\
  & \leq c \norm{f} t^{-3} \delta, \nonumber
\end{align}
the last from Lemma \ref{equ:lemma_7}(b) and (\ref{equ:twentysix}). Just as for $(P_t f)_{kl}, k, l \in N_{C2}$, we may use this with Lemma \ref{equ:lemma_11}(b) and dominated convergence to conclude that $(P_t f)_{kj}$ is continuous in $x^{N_R}$ uniformly in $x^{N_{C2}}$. Continuity in $x^{N_{C2}}$ for each $x^{N_R}$ is obtained by an easy modification of the argument for $(P_t f)_{kl}$, using the bound (\ref{equ:sixtythree}). This completes the proof that $P_t f$ is $\SC^2$.

Finally to get a (crude) upper bound on $|(P_t f)_{kj}|$ use (\ref{equ:sixtythree}) with $\delta = t$ and Lemma \ref{equ:lemma_11}(a) to see
\begin{align*}
  |(P_t f)_{kj} (x) | & \leq E^{N_{C2}} \left[ \int \left| \dx{j} \Delta G_{t,x^{N_R}}^{+k} \left( x^{N_{C2}}; \int_0^t \nu_s ds, 0 \right) \right| \1_{\{\nu_t = 0\}} d\NN_0(\nu) \right] \\
  & \quad + E^{N_{C2}} \left[ \int \1_{\{\nu_t > 0\}} \left[ \left| \dx{j} G_{t,x^{N_R}}^{+k} \left( x^{N_{C2}} ; \int_0^t \nu_s ds , \nu_t \right) \right| \right. \right. \\
  & \left. \left. \quad \qquad + \left| \dx{j} G_{t,x^{N_R}}^{+k} \left( x^{N_{C2}} ; 0, 0 \right) \right| \right] d\NN_0(\nu) \right] \\
  & \leq c \norm{f} t^{-2} + c \norm{f} E^{N_{C2}} \left[ \left( I_t^{(j)} \right)^{-1/2} \right] \NN_0(\nu_t > 0) \\
  & \leq c \norm{f} t^{-2},
\end{align*}
by Lemma \ref{equ:lemma_7}(b) and (\ref{equ:twentyfive}).
%
% additional part ends
%
\end{proof}
%
% ============================================================
%
\subsection{$L^{\infty}$ bounds of certain differentiation operators applied to $P_t f$ and equivalence of norms} \label{equ:section_2_3}
%
% ============================================================

We continue to work with the semigroup $P_t$ on the state space $\SSS_0$ corresponding to the generator $\SA^0$. Recall the definitions of the semigroup norm $|f|_\alpha$ from (\ref{equ:f_alpha_definition}) and of the associated Banach space of functions $S^\alpha$ from (\ref{equ:S_alpha_definition}) in what follows.
%
% ------------------------------------------------------------
%
\begin{pro} \label{equ:proposition_16}
If $f$ is a bounded Borel function on $\SSS_0$ then for $j \in N_R$ 
\begin{equation} \label{equ:sixtyfour}
  \left| \dx{j} P_tf (x) \right| \leq \frac{c \norm{f}}{\sqrt{t} \max \limits_{i \in C_j} \{ \sqrt{t+x_i} \}}, 
\end{equation}
and
\begin{equation} \label{equ:sixtyfive}
  \left| \max \limits_{i \in C_j} \{x_i\} \ddx{j} P_tf (x) \right| \leq \frac{c \norm{f}}{t}. 
\end{equation}
If $f \in \SSS^{\alpha}$, then
\begin{equation} \label{equ:sixtysix}
  \left| \dx{j} P_tf (x) \right| \leq \frac{c |f|_{\alpha} t^{\frac{\alpha}{2}-\frac{1}{2}}}{\max \limits_{i \in C_j} \{ \sqrt{t+x_i} \}} \leq c |f|_{\alpha} t^{\frac{\alpha}{2}-1}, 
\end{equation}
and
\begin{equation} \label{equ:sixtyseven}
  \left| \max \limits_{i \in C_j} \{x_i\} \ddx{j} P_tf (x) \right| \leq c |f|_{\alpha} t^{\frac{\alpha}{2}-1}. 
\end{equation}
\end{pro}
%
% ------------------------------------------------------------
%
\begin{proof}
The proof proceeds as in \cite{r6}, Proposition 16 except for minor changes. 
%
% additional part starts

We begin with the first derivative for $f$ bounded Borel measurable. Use (\ref{equ:fourtytwo}), Proposition \ref{equ:proposition_14} (for the existence of $(P_t f)_j$) and Lemma \ref{equ:lemma_7}(b) to see that
\[ \left| \dx{j} P_t f (x) \right| \leq c \norm{f} E^{N_{C2}} \left[ \left( I_t^{(j)} \right)^{-1/2} \right] \leq \frac{c \norm{f}}{\sqrt{t} \max \limits_{i \in C_j} \{ \sqrt{t+x_i} \}}. \]

We now turn to the second derivative. Note that $\SA^0$ (cf. (\ref{equ:seven})) and $\dx{j}$ commute and therefore the semigroup $P_t$ and $\dx{j}$ commute. Therefore a double application of
\[ (\ref{equ:sixtyfour}) \iff \left\{ \sqrt{t} \max \limits_{i \in C_j} \{ \sqrt{t+x_i} \} \left| \dx{j} P_tf (x) \right| \leq c \norm{f} \right\} \]
gives
\begin{align*}
  & \left| \max \limits_{i \in C_j} \{x_i\} \ddx{j} P_tf (x) \right| \\
  & \leq \frac{c}{t} \left| \sqrt{\frac{t}{2}} \max \limits_{i \in C_j} \left\{ \sqrt{\frac{t}{2}+x_i} \right\} \dx{j} P_{\frac{t}{2}} \left( \sqrt{\frac{t}{2}} \max \limits_{i \in C_j} \left\{ \sqrt{\frac{t}{2}+x_i} \right\} \dx{j} P_{\frac{t}{2}} f \right)\!(x) \right| \\
  & \leq \frac{c}{t} \bignorm{ \sqrt{\frac{t}{2}} \max \limits_{i \in C_j} \left\{ \sqrt{\frac{t}{2}+x_i} \right\} \dx{j} P_{\frac{t}{2}} f (x) } \\
  & \leq \frac{c \norm{f}}{t}.
\end{align*}
This proves the first two inequalities and also shows that
\[ \lim \limits_{t \rightarrow \infty} \bignorm{ \dx{j} P_t f } = 0. \]
%
% additional part ends

If $f \in \SSS^{\alpha}$, we proceed as in \cite{r2} and write
\[ \left| \dx{j} P_{2t} f (x) - \dx{j} P_t f (x) \right| = \left| \dx{j} P_t (P_t f - f)(x) \right|. \]
Applying the estimate (\ref{equ:sixtyfour}) to $g = P_t f - f$ and using the definition of $|f|_{\alpha}$ we get
\[ \left| \dx{j} P_{2t} f (x) - \dx{j} P_t f (x) \right| \leq \frac{c \norm{g}}{\sqrt{t} \max \limits_{i \in C_j} \{ \sqrt{t+x_i} \}} \leq \frac{c |f|_{\alpha} t^{\alpha/2}}{\sqrt{t} \max \limits_{i \in C_j} \{ \sqrt{t+x_i} \}}. \]
This together with 
\[ (\ref{equ:sixtyfour}) \Rightarrow \lim \limits_{t \rightarrow \infty} \left| \dx{j} P_t f(x) \right| = 0 \]
implies that
\begin{align*}
  \left| \dx{j} P_t f (x) \right| &\leq \sum \limits_{k=0}^{\infty} \left| \dx{j} \left( P_{2^k t} f - P_{2^{(k+1)} t} f \right) (x) \right| \\ 
  &\leq |f|_{\alpha} \sum \limits_{k=0}^{\infty} \left( 2^k t \right)^{\frac{\alpha}{2}-\frac{1}{2}} \frac{c}{\max \limits_{i \in C_j} \{ \sqrt{2^k t + x_i} \}} \\
  &\leq |f|_{\alpha} t^{\frac{\alpha}{2}-\frac{1}{2}} \frac{c}{\max \limits_{i \in C_j} \{ \sqrt{t + x_i} \}}.
\end{align*}
This then immediately yields (\ref{equ:sixtysix}). Use (\ref{equ:sixtyfive}) to derive (\ref{equ:sixtyseven}) in the same way as (\ref{equ:sixtyfour}) was used to prove (\ref{equ:sixtysix}).
\end{proof}
%
% ------------------------------------------------------------
%
\begin{nota} \label{equ:nota_R_S}
If $w>0$, set $p_j(w) = \frac{w^j}{j!} e^{-w}$. For $\{r_j(t)\}$ and $\{e_j(t)\}$ as in Lemma \ref{equ:lemma_10}, let $R_k = R_k(t) = \sum_{j=1}^k r_j(t)$ and $S_k = S_k(t) = \sum_{j=1}^k e_j(t)$. 
\end{nota}
%
% ------------------------------------------------------------
%
\begin{nota} \label{equ:notation_changes}
If $X \in \SC\!\left( \RR_+,\RR_+^{N_{C2}} \right)$, $Y, Y', Z, Z' \in \SC(\RR_+,\RR_+)$, $\eta, \eta'$, $\theta, \theta' \in \SC_{ex}$ and $m, n, k, l \in N_{C2}$, where $m \neq n$ let
\begin{align*}
  & G_{t,x^{N_R}}^{m,n,+k,+l} \!\left( X, Y_t, Z_t, Y_t', Z_t';\int_0^t \eta_s ds,\theta_t,\int_0^t \eta_s' ds,\theta_t' \right) \\
  & \equiv G_{t,x^{N_R}} \Biggl( \biggl\{ \int_0^t \sum \limits_{i \in C_j \backslash \{m,n\} } X_s^i ds + \1_{\{m \in C_j\}} Y_t + \1_{\{n \in C_j\}} Y_t' \\ 
  & \qquad \qquad \qquad \qquad \qquad \qquad \qquad + \int_0^t \1_{\{k \in C_j\}} \eta_s + \1_{\{l \in C_j\}} \eta_s' ds \biggr\}_{j \in N_R} , \\
  & \qquad \qquad \quad \biggl\{ \1_{\{i \notin \{m,n\}\}} X_t^i + \1_{\{i=m\}} Z_t + \1_{\{i=n\}} Z_t' \\
  & \qquad \qquad \qquad \qquad \qquad \qquad \qquad + \1_{\{i=k\}} \theta_t + \1_{\{i=l\}} \theta_t' \biggr\}_{i \in N_{C2}} \Biggr).
\end{align*}
The notation indicates that the one-dimensional coordinate processes \\
$\int_0^t X_s^m ds, X_t^m$ resp. $\int_0^t X_s^n ds, X_t^n$ will be replaced by the processes $Y_t, Z_t$ resp. $Y_t', Z_t'$ (note that for $m \in N_2$ this only implies a change from $X_t^m$ into $Z_t$). Additionally, we add $\int_0^t \nu_s ds, \theta_t, \int_0^t \nu_s' ds$ and $\theta_t'$ as before. The terms
\begin{equation} \label{equ:more_defs}
  G_{t,x^{N_R}}^{m,+k,+l}, G_{t,x^{N_R}}^{m,+k}, G_{t,x^{N_R}}^{m,n,+l}, G_{t,x^{N_R}}^{m,n}, G_{t,x^{N_R}}^{m}, \Delta G_{t,x^{N_R}}^{m,+k,+l} \mbox{ etc.} 
\end{equation}
will then be defined in a similar way, where for instance $G_{t,x^{N_R}}^{m}$ only refers to replacing the processes $\int_0^t X_s^m ds, X_t^m$ via $Y_t, Z_t$ but doesn't involve adding processes.
\end{nota}
%
% ------------------------------------------------------------
%
\begin{pro} \label{equ:proposition_17}
If $f$ is a bounded Borel function on $\SSS_0$, then for $i \in N_{C2}$
\begin{equation} \label{equ:seventy}
  \left| \dx{i} P_tf (x) \right| \leq \frac{c \norm{f}}{\sqrt{t} \sqrt{t+x_i}}, 
\end{equation}
and
\begin{equation} \label{equ:seventyone}
  \left| x_i \ddx{i} P_tf (x) \right| \leq \frac{c x_i \norm{f}}{t (t+x_i)} \leq \frac{c \norm{f}}{t}.
\end{equation}
If $f \in \SSS^{\alpha}$, then
\[ \left| \dx{i} P_tf (x) \right| \leq \frac{c |f|_{\alpha} t^{\frac{\alpha}{2}-\frac{1}{2}}}{\sqrt{t+x_i}} \leq c |f|_{\alpha} t^{\frac{\alpha}{2}-1}, \]
and
\[ \left| x_i \ddx{i} P_tf (x) \right| \leq c |f|_{\alpha} t^{\frac{\alpha}{2}-1}. \]
\end{pro}
%
% ------------------------------------------------------------
%
\begin{proof}
The outline of the proof is the same as for \cite{r6}, Proposition 17. As in the proof of Proposition \ref{equ:proposition_14} we assume w.l.o.g. that $f$ is bounded and continuous.
%
% additional part starts

From Proposition \ref{equ:proposition_14} we have for $k \in N_{C2}$
\begin{align*} 
  (P_t f)_k (x) =& \; E^{N_{C2}} \!\left[ \int \Delta G^{+k}_{t,x^{N_R}} \!\left( x^{N_{C2}},\nu \right) d\NN_0(\nu) \right] \\
  =& \; E^{N_{C2}} \biggl[ \int \biggl\{ G_{t,x^{N_R}}^{+k} \!\left( x^{N_{C2}};\int_0^t \nu_s ds,0 \right) \\
  & \qquad \qquad \qquad - G_{t,x^{N_R}}^{+k} \!\left( x^{N_{C2}};0,0 \right) \biggr\} \1_{\{\nu_t=0\}} d\NN_0(\nu) \biggr] \\
  & \; + \, E^{N_{C2}} \biggl[ \int \biggl\{ G_{t,x^{N_R}}^{+k} \!\left( x^{N_{C2}};\int_0^t \nu_s ds,\nu_t \right) \\
  & \qquad \qquad \qquad \quad - G_{t,x^{N_R}}^{+k}\!\left( x^{N_{C2}};0,0 \right) \biggr\} \1_{\{\nu_t>0\}} d\NN_0(\nu) \bigg] \\
  \equiv& \; E_1 + E_2.
\end{align*}
By Lemmas \ref{equ:lemma_11} and \ref{equ:lemma_7}, the triangle inequality and (\ref{equ:twentysix})
\begin{align} 
  |E_1| &\leq c \norm{f} \sum \limits_{j: j \in \bar{R}_k} \int \int_0^t \nu_s ds d\NN_0(\nu) E^{N_{C2}} \!\left[ \left( I_t^{(j)} \right)^{-1} \right] \label{equ:seventyfour} \\
  &\leq c \norm{f} t t^{-1} \sum \limits_{j: j \in \bar{R}_k} \min \limits_{i \in C_j} \!\left\{ (t+x_i)^{-1} \right\} \nonumber \\
  &\leq c \norm{f} (t+x_k)^{-1}. \nonumber
\end{align}

Next we shall use the decomposition of Lemma \ref{equ:lemma_10} with $\rho=0$.
%
% additional part ends
%
% ------------------------------------------------------------
%
\begin{nota}
We have $\NN_0[\cdot \cap \{\nu_t>0\}] = (\gamma t)^{-1} P_t^*[\cdot]$ on $\{\nu_t>0\}$, where we used (\ref{equ:twentyseven}) and (\ref{equ:twentyfive}). Whenever we change integration w.r.t. $\NN_0$ to integration w.r.t. $P_t^*$ we shall denote this by $\stackrel{(*)}{=}$. 
\end{nota}
%
% ------------------------------------------------------------
%
% additional part starts

Using (\ref{equ:thirty}) and (\ref{equ:thirtytwo}) we can rewrite $E_2$ into
\begin{align*}
  E_2 &= E \biggl[ \int \biggl\{ G_{t,x^{N_R}}^{k,+k} \!\left( x^{N_{C2}}, R_{N_t} + I_2(t), S_{N_t} + X_0'(t) ;\int_0^t \nu_s ds,\nu_t \right) \\
  & \qquad \qquad \qquad - G_{t,x^{N_R}}^{k,+k} \!\left( x^{N_{C2}}, R_{N_t} + I_2(t), S_{N_t} + X_0'(t);0,0 \right) \biggr\} \1_{\{\nu_t>0\}} d\NN_0(\nu) \biggr] \\
  &\stackrel{(*)}{=} \frac{c}{t} E \!\left[ G_{t,x^{N_R}}^{k} \!\left( x^{N_{C2}}, R_{N_t+1} + I_2(t), S_{N_t+1} + X_0'(t) \right) \right. \\
  & \qquad \qquad \left. - G_{t,x^{N_R}}^{k} \!\left( x^{N_{C2}}, R_{N_t} + I_2(t), S_{N_t} + X_0'(t) \right) \right]. 
\end{align*}

Let $w = \frac{x_k}{\gamma_k^0 t}$ and recall that $\norm{G} \leq \norm{f}$. The independence of $N_t$ from $(\{ \int_0^t x_s^{(i)} ds, i \in C_j \backslash \{k\}, j \in N_R \}, x_t^{(N_{C2}) \backslash \{k\}}, I_2(t), X_0'(t),\{e_l\},\{r_l\})$ and summation by parts yields
\begin{align*}
  |E_2| &= \frac{c}{t} \biggl| \sum_{n=0}^{\infty} p_n(w) E \biggl[ G_{t,x^{N_R}}^{k} \!\left( x^{N_{C2}}, R_{n+1} + I_2(t), S_{n+1} + X_0'(t) \right) \\
  & \qquad \qquad \qquad \qquad - G_{t,x^{N_R}}^{k} \!\left( x^{N_{C2}}, R_n + I_2(t), S_n + X_0'(t) \right) \biggr] \biggr| \\
  &\leq \frac{c}{t} \norm{f} \sum \limits_{n=1}^{\infty} | p_{n-1}(w) - p_n(w) | + \frac{c}{t} e^{-w} \norm{f} \\
  &\leq \frac{c}{t} \norm{f} \sum_{n=0}^{\infty} p_n(w) \frac{|n-w|}{w} \\
  &\leq \frac{c}{t} \norm{f} w^{-1} \left( \left( E \!\left[ (N_t-w)^2 \right] \right)^{1/2} \wedge E[N_t+w] \right) \\
  &= \frac{c}{t} \norm{f} w^{-1} \left( \sqrt{w} \wedge 2w \right) \\
  &= \frac{c}{t} \norm{f} \sqrt{t} \left( \frac{1}{\sqrt{x_k}} \wedge 2 \frac{1}{\sqrt{t}} \right) \\
  &\leq \frac{c \norm{f}}{\sqrt{t} \sqrt{t+x_k}}.
\end{align*}
This and (\ref{equ:seventyfour}) give (\ref{equ:seventy}).
%
% additional part ends

Next consider second derivatives in $k$. The representation of $(P_t f)_{kk}$ in Proposition \ref{equ:proposition_14} and symmetry allow us to write for $k \in N_{C2}$ (i.e. $l=k$) 
\begin{align*}
  (P_t f)_{kk} (x) = & \ E^{N_{C2}} \!\left[ \int \! \int \! \Delta G^{+k,+k}_{t,x^{N_R}} \!\left( x^{N_{C2}},\nu,\nu' \right) \1_{\{\nu_t=0,\nu_t'=0\}} d\NN_0(\nu) d\NN_0(\nu') \right] \\
  & \!\!\! + 2 E^{N_{C2}} \biggl[ \int \! \int \! \Delta G^{+k,+k}_{t,x^{N_R}} \!\left( x^{N_{C2}},\nu,\nu' \right) \1_{\{\nu_t=0,\nu_t'>0\}} d\NN_0(\nu) d\NN_0(\nu') \biggr] \\
  & \!\!\! + E^{N_{C2}} \biggl[ \int \! \int \! \Delta G^{+k,+k}_{t,x^{N_R}} \!\left( x^{N_{C2}},\nu,\nu' \right) \1_{\{\nu_t>0,\nu_t'>0\}} d\NN_0(\nu) d\NN_0(\nu') \biggr] \\
  \equiv & \ E_1 + 2 E_2 + E_3.
\end{align*}

The idea for bounding $|E_1|, |E_2|$ and $|E_3|$ is similar to the one in \cite{r6}.  
%
% additional part starts

Use Lemma \ref{equ:lemma_11}(b) and Lemma \ref{equ:lemma_13} to show that
\begin{align*}
  |E_1| &\leq \left| E^{N_{C2}} \left[ \int \int G_{t,x^{N_R}}^{+k,+k} \left( x^{N_{C2}};\int_0^t \nu_s ds,0,\int_0^t \nu_s' ds,0 \right) \right. \right. \\
  & \qquad \qquad \qquad \; - \, G_{t,x^{N_R}}^{+k,+k} \left( x^{N_{C2}};0,0,\int_0^t \nu_s' ds,0 \right) \\
  & \qquad \qquad \qquad \; - \, G_{t,x^{N_R}}^{+k,+k} \left( x^{N_{C2}};\int_0^t \nu_s ds,0,0,0 \right) \\
  & \qquad \qquad \qquad \left. \left. \; + \, G_{t,x^{N_R}}^{+k,+k} (x^{N_{C2}};0,0,0,0) d\NN_0(\nu) d\NN_0(\nu') \right] \right| \\
  &\leq c \norm{f} E^{N_{C2}} \left[ \sum \limits_{j_1: j_1 \in \bar{R}_k} \sum \limits_{j_2: j_2 \in \bar{R}_k} \left( I_t^{(j_1)} \right)^{-1} \left( I_t^{(j_2)} \right)^{-1} \right] \\ 
  & \times \int \int_0^t \nu_s ds d\NN_0(\nu) \int \int_0^t \nu_s' ds d\NN_0(\nu') \\
  &\leq c \norm{f} \sum \limits_{j_1: j_1 \in \bar{R}_k} \sum \limits_{j_2: j_2 \in \bar{R}_k} t^{-1} \min \limits_{i \in C_{j_1}} \{ (t+x_i)^{-1} \} t^{-1} \min \limits_{i \in C_{j_2}} \{ (t+x_i)^{-1} \} t^2 \\ 
  &\leq \frac{c \norm{f}}{(t+x_k)^2}, 
\end{align*}
where in the next to last line we have used (\ref{equ:twentysix}) and Lemma \ref{equ:lemma_7}(b).

For $E_2$ we may drop $\1_{\{\nu_t=0\}}$ from the expression for $\Delta G^{+k,+k}_{t,x^{N_R}} (x^{N_{C2}},\nu,\nu')$, regroup terms, and use Lemma \ref{equ:lemma_11}(b), triangle inequality and (\ref{equ:twentyfive}) to write
\begin{align*}
  |E_2| &\leq E^{N_{C2}} \left[ \int \int \left( \left| G_{t,x^{N_R}}^{+k,+k} \left( x^{N_{C2}};\int_0^t \nu_s ds,0,\int_0^t \nu_s' ds,\nu_t' \right) \right. \right. \right. \\
  & \qquad \qquad \qquad \quad \left. \; - \, G_{t,x^{N_R}}^{+k,+k} \left( x^{N_{C2}};0,0,\int_0^t \nu_s' ds,\nu_t' \right) \right| \\
  & \qquad \qquad \quad \; + \, \left| G_{t,x^{N_R}}^{+k,+k} \left( x^{N_{C2}};\int_0^t \nu_s ds,0,0,0 \right) \right. \\
  & \qquad \qquad \qquad \left. \left. \left. \; - \, G_{t,x^{N_R}}^{+k,+k} (x^{N_{C2}};0,0,0,0) \right| \right) \1_{\{\nu_t'>0\}} d\NN_0(\nu) d\NN_0(\nu') \right] \\
  &\leq c \norm{f} E^{N_{C2}} \left[ \sum \limits_{j: j \in \bar{R}_k} \left( I_t^{(j)} \right)^{-1} \right] \int \int_0^t \nu_s ds d\NN_0(\nu) \NN_0(\nu_t'>0) \\
  &\leq c \norm{f} t^{-1} \sum \limits_{j: j \in \bar{R}_k} \min \limits_{i \in C_j} \{ (t+x_i)^{-1} \} t t^{-1} \\
  &\leq \frac{c \norm{f}}{t (t+x_k)},
\end{align*}
where in the next to last line we have again used (\ref{equ:twentysix}) and Lemma \ref{equ:lemma_7}(b).
%
% additional part ends

The decomposition of Lemma \ref{equ:lemma_10} (cf. (\ref{equ:thirty}) and (\ref{equ:thirtytwo})) with $\rho = 0$ gives
\begin{align}
  |E_3| \stackrel{(*)}{=}& \; \frac{c}{t^2} \biggl| E \biggl[ \int \int \Bigl\{ G_{t,x^{N_R}}^{k,+k,+k} \Bigl( x^{N_{C2}}, R_{N_t} + I_2(t), S_{N_t} + X_0'(t); \label{equ:seventyfive} \\
  & \qquad \qquad \qquad \qquad \qquad \qquad \qquad \qquad \int_0^t \! \nu_s ds,\nu_t,\int_0^t \! \nu_s' ds,\nu_t' \Bigr) \nonumber \\
  & \quad \!\! - \, G_{t,x^{N_R}}^{k,+k,+k} \!\left( x^{N_{C2}}, R_{N_t} + I_2(t), S_{N_t} + X_0'(t);0,0,\int_0^t \! \nu_s' ds,\nu_t' \right) \nonumber \\
  & \quad \!\! - \, G_{t,x^{N_R}}^{k,+k,+k} \!\left( x^{N_{C2}}, R_{N_t} + I_2(t), S_{N_t} + X_0'(t);\int_0^t \! \nu_s ds,\nu_t,0,0 \right) \nonumber \\
  & \quad \!\! + \, G_{t,x^{N_R}}^{k,+k,+k} \Bigl( x^{N_{C2}}, R_{N_t} + I_2(t), S_{N_t} + X_0'(t);0,0,0,0 \Bigr) \Bigr\} \nonumber \\
  & \!\! \times dP_t^*(\nu) dP_t^*(\nu') \biggr] \biggr|, \nonumber 
\end{align}
where for instance
\begin{align*}
  & G_{t,x^{N_R}}^{k,+k,+k} \Bigl( x^{N_{C2}}, R_{N_t} + I_2(t), S_{N_t} + X_0'(t); \int_0^t \! \nu_s ds,\nu_t,\int_0^t \! \nu_s' ds,\nu_t' \Bigr) \\
  & = G_{t,x^{N_R}} \Biggl( \biggl\{ \int_0^t \sum \limits_{i \in C_j \backslash \{k\} } X_s^i ds + \1_{\{k \in C_j\}} \left( R_{N_t} + I_2(t) \right) \\ 
  & \qquad \qquad \qquad \qquad \qquad \qquad \qquad + \int_0^t \1_{\{k \in C_j\}} \left( \nu_s + \nu_s' \right) ds \biggr\}_{j \in N_R} , \\
  & \qquad \qquad \quad \biggl\{ \1_{\{i \neq k\}} X_t^i + \1_{\{i=k\}} \left( S_{N_t} + X_0'(t) \right) \\
  & \qquad \qquad \qquad \qquad \qquad \qquad \qquad + \1_{\{i=k\}} \left( \nu_t + \nu_t' \right) \biggr\}_{i \in N_{C2}} \Biggr)
\end{align*}
by Notation \ref{equ:notation_changes} and the comment following it. 

Recall that $R_k = R_k(t) = \sum_{j=1}^k r_j(t)$ and $S_k = S_k(t) = \sum_{j=1}^k e_j(t)$ with $\{r_j(t)\}$ and $\{e_j(t)\}$ as in Lemma \ref{equ:lemma_10}. In particular, $\{e_j, j \in \NN\}$ is iid with common law $P_t^*$ and $r_j(t) = \int_0^t e_j(s) ds$.

We obtain (recall the definition of $G_{t,x^{N_R}}^k$ from (\ref{equ:more_defs}))
\begin{align*}
  |E_3| =& \; \frac{c}{t^2} \biggl| E \biggl[ G_{t,x^{N_R}}^k \Bigl( x^{N_{C2}}, R_{N_t+2} + I_2(t), S_{N_t+2} + X_0'(t) \Bigr) \\
  & \qquad \quad \!\! - \, 2 G_{t,x^{N_R}}^k \!\left( x^{N_{C2}}, R_{N_t+1} + I_2(t), S_{N_t+1} + X_0'(t) \right) \\
  & \qquad \quad \!\! + \, G_{t,x^{N_R}}^k \Bigl( x^{N_{C2}}, R_{N_t} + I_2(t), S_{N_t} + X_0'(t) \Bigr) \biggr] \biggr|.
\end{align*}
Observe that in case $k \in N_2$ the above notation $G_{t,x^{N_R}}^{k} ( x^{N_{C2}}, R_{N_t} + I_2(t), \\ S_{N_t} + X_0'(t))$ only indicates that $x_t^{(k)}$ gets changed into $S_{N_t} + X_0'(t)$; for $k \in N_2$ the indicated change of $\int_0^t x_s^{(k)} ds$ into $R_{N_t} + I_2(t)$ has no impact on the term under consideration.

Let $w = x_k / (\gamma_k^0 t)$. The independence of $N_t$ from $(\{ \int_0^t x_s^{(i)} ds, i \in C_j \backslash \{k\}, j \in N_R \}, x_t^{(N_{C2}) \backslash \{k\}}, I_2(t), X_0'(t),\{e_l\},\{r_l\})$ yields
\begin{align*}
  |E_3| =& \; \frac{c}{t^2} \Bigl| \sum \limits_{n=0}^{\infty} p_n(w) E \Bigl[ G_{t,x^{N_R}}^k \!\left( x^{N_{C2}}, R_{n+2} + I_2(t), S_{n+2} + X_0'(t) \right) \\
  & \qquad \qquad \qquad \quad - 2 G_{t,x^{N_R}}^k \!\left( x^{N_{C2}}, R_{n+1} + I_2(t), S_{n+1} + X_0'(t) \right) \\
  & \qquad \qquad \qquad \quad + G_{t,x^{N_R}}^k \!\left( x^{N_{C2}}, R_n + I_2(t), S_n + X_0'(t) \right) \Bigr] \Bigr|.
\end{align*}
Sum by parts twice and use $|G| \leq \norm{f}$ to bound the above by
\begin{align*}
  & c \norm{f} \frac{1}{x_k t} \left| w (3 p_0(w) + p_1(w)) + \sum \limits_{n=2}^{\infty} w \left( p_{n-2}(w) - 2p_{n-1}(w) + p_n(w) \right) \right| \\
  & \leq c \norm{f} \frac{1}{x_k t} \left( w p_0(w) + w p_1(w) + \sum_{n=2}^{\infty} p_n(w) \frac{| (w-n)^2 - n |}{w} \right) \\                           \\
  & \leq c \norm{f} \frac{1}{x_k t} \left( 2 p_1(w) + \sum_{n=0}^{\infty} p_n(w) \frac{ (w-n)^2 + n }{w} \right) \\ 
  & \leq c \norm{f} \frac{1}{x_k t}.
\end{align*}

We obtain another bound on $|E_3|$ if we use the trivial bound $|G| \leq \norm{f}$ in (\ref{equ:seventyfive}). This yields $|E_3| \leq c \norm{f} t^{-2}$ and so
\[ |E_3| \leq \frac{c \norm{f}}{t(t+x_k)}. \]

Combine the bounds on $|E_1|, |E_2|$ and $|E_3|$ to obtain (\ref{equ:seventyone}).

The bounds for $f \in \SSS^\alpha$ are obtained from the above just as in the proof of Proposition \ref{equ:proposition_16}. 
\end{proof}
%
% ------------------------------------------------------------
%

Recall Convention \ref{equ:convention_1}, as stated in (\ref{equ:fifteen}), for the definition of $M^0$ in what follows.
%
% ------------------------------------------------------------
%
\begin{nota}
Set $J_t^{(j)} = \gamma_j^0 2 I_t^{(j)}, j \in N_R$. 
\end{nota}
%
% ------------------------------------------------------------
%
\begin{lem} \label{equ:lemma_18}
For each $M \geq 1$, $\alpha \in (0,1)$ and $d \in \NN$ there is a $c=c(M,\alpha,d) > 0$ such that if $M^0 \leq M$, then
\begin{equation} \label{equ:seventysix}
  |fg|_{\alpha} \leq c \wabsv{f} \norm{g} + \norm{f} |g|_{\alpha} 
\end{equation}
and
\begin{equation} \label{equ:seventyseven}
  \normx{fg}_{\alpha} \leq c \left( \normw{f} \norm{g} + \norm{f} |g|_{\alpha}  \right). 
\end{equation}
\end{lem}
%
% ------------------------------------------------------------
%
\begin{proof}
Compared to the proof of \cite{r6}, Lemma 18, the derivation of a bound for the second error term $E_2$ below becomes more involved. Again the triangle-inequality has to be used to express multi-dimensional coordinate changes via one-dimensional ones. 

Let $\left( x^{N_R}, x^{N_{C2}} \right) \in \RR^{|N_R|} \times \RR_+^{|N_{C2}|}$ and define $\tilde{f}(y) = f(y)-f(x)$. Then (\ref{equ:eighteen}) gives
\begin{align}
  & |P_t(fg) (x) - (fg) (x)| \label{equ:seventyeight} \\
  & \leq |P_t(\tilde{f}g) (x)| + |f(x)| |P_tg (x) - g (x)| \nonumber \\
  & \leq \; \norm{g} E^{N_{C2}} \!\left[ \int_{\RR^{|N_R|}} \left| \tilde{f} \!\left( z^{N_R}, x_t^{N_{C2}} \right) \right| \prod \limits_{j \in N_R} p_{J_t^{(j)}} \!\left( z_j -x_j -b_j^0 t \right) dz_j \right] \nonumber \\
  & \qquad + \norm{f} |g|_{\alpha} t^{\alpha/2}. \nonumber
\end{align}

The above expectation can be bounded by three terms as follows:
\begin{align}
  & E^{N_{C2}} \!\left[ \int \left| \tilde{f} \!\left( z^{N_R}, x_t^{N_{C2}} \right) \right| \prod \limits_{j \in N_R} p_{J_t^{(j)}} \!\left( z_j -x_j -b_j^0 t \right) dz_j \right] \label{equ:seventynine} \\
  & \leq E^{N_{C2}} \biggl[ \int \left\{ \left| \tilde{f} \!\left( z^{N_R}, x_t^{N_{C2}} \right) - \tilde{f} \!\left( z^{N_R}, x^{N_{C2}} \right) \right| \right. \nonumber \\
  & \qquad \qquad \qquad \quad + \left| f \!\left( z^{N_R}, x^{N_{C2}} \right) - f \!\left( x^{N_R} + b_{N_R}^0 t, x^{N_{C2}} \right) \right| \nonumber \\
  & \left. \qquad \qquad \qquad \quad + \left| f \!\left( x^{N_R} + b_{N_R}^0 t, x^{N_{C2}} \right) - f \!\left( x^{N_R}, x^{N_{C2}} \right) \right| \right\} \nonumber \\
  & \quad \qquad \; \times \prod \limits_{j \in N_R} p_{J_t^{(j)}} \!\left( z_j -x_j -b_j^0 t \right) dz_j \biggr] \nonumber \\
  & \equiv E_1 + E_2 + E_3. \nonumber
\end{align}

For all three terms we shall use the triangle inequality to sum up changes in different coordinates separately.

The definition of $\ainorm{f}{i}$ gives
\begin{align*}
  E_1 &\leq \sum \limits_{i \in N_{C2}} \ainorm{f}{i} E^{N_{C2}} \!\left[ \left( \left| x_t^{(i)} - x_i \right|^{\alpha} x_i^{-\alpha/2} \right) \wedge \left| x_t^{(i)} - x_i \right|^{\alpha/2} \right] \\
  &\leq \sum \limits_{i \in N_{C2}} \ainorm{f}{i} \left( \left( E^{N_{C2}} \!\left[ \left| x_t^{(i)} - x_i \right|^2 \right]^{\alpha/2} x_i^{-\alpha/2} \right) \wedge E^{N_{C2}} \!\left[ \left| x_t^{(i)} - x_i \right|^2 \right]^{\alpha/4} \right).
\end{align*}
We now proceed as in the derivation of a bound on $E_1$ in the proof of Lemma 18 in \cite{r6}, using Lemma \ref{equ:lemma_7}(a) (alternatively compare with estimation of $E_2$ below). We finally obtain
\[ E_1 \leq c \sum \limits_{i \in N_{C2}} \ainorm{f}{i} t^{\alpha/2} 2^{\alpha/2} \leq c \wabsv{f} t^{\alpha/2} 2^{\alpha/2}. \]

Similarly we have
\begin{align*}
  E_2 \leq& \; \sum \limits_{k \in N_R} \min \limits_{i: k \in \bar{R}_i} \biggl\{ \ainorm{f}{i} E^{N_{C2}} \biggl[ \int \left( \left( \left| z_k - (x_k + b_k^0 t) \right|^{\alpha} x_i^{-\alpha/2} \right) \wedge \right. \\
  & \left. \; \left| z_k - (x_k + b_k^0 t) \right|^{\alpha/2} \right) \prod \limits_{j \in N_R} p_{J_t^{(j)}} \!\left( z_j -x_j -b_j^0 t \right) dz_j \biggr] \biggr\} \\
  \leq& \; c \sum \limits_{k \in N_R} \min \limits_{i: k \in \bar{R}_i} \!\left\{ \ainorm{f}{i} E^{N_{C2}} \!\left[ \left( \left| J_t^{(k)} \right|^{\alpha/2} x_i^{-\alpha/2} \right) \wedge \left| J_t^{(k)} \right|^{\alpha/4} \right] \right\} \\
  \leq& \; c \sum \limits_{k \in N_R} \min \limits_{i: k \in \bar{R}_i} \biggl\{ \ainorm{f}{i} \biggl( \biggl( E^{N_{C2}} \!\left[ \left| J_t^{(k)} \right| \right]^{\alpha/2} x_i^{-\alpha/2} \biggr) \wedge E^{N_{C2}} \!\left[ \left| J_t^{(k)} \right| \right]^{\alpha/4} \biggr) \biggr\}
\end{align*}
as $\int |z|^{\beta} p_J(z) dz \leq c J^{\beta/2}$ for $\beta \in (0,1)$. Next use Lemma \ref{equ:lemma_7}(a) which shows that $E^{N_{C2}} \!\left[ J_t^{(k)} \right] = \gamma_k^0 2 E^{N_{C2}} \!\left[ I_t^{(k)} \right] \leq \sum_{l \in C_k} c M^2 (t^2 + x_l t)$. Put this in the above bound on $E_2$ to see that $E_2$ can be bounded by
\begin{align*}
  & c \sum \limits_{k \in N_R} \! \min \limits_{i: k \in \bar{R}_i} \!\left\{ \ainorm{f}{i} \left( \left( \left( \sum \limits_{l \in C_k} (t^2 + x_l t) \right)^{\alpha/2} \!\!\! x_i^{-\alpha/2} \right) \wedge \left( \sum \limits_{l \in C_k} (t^2 + x_l t) \right)^{\alpha/4} \right) \right\} \\
  & \stackrel{k \in N_R}{\leq} c \wabsv{f} \sum \limits_{k \in N_R} \left( \left( \sum \limits_{l \in C_k} \frac{ t^2 + x_l t }{ \max \limits_{i: k \in \bar{R}_i} x_i } \right)^{\alpha/2} \wedge \left( \sum \limits_{l \in C_k} \left( t^2 + t \max \limits_{i: k \in \bar{R}_i} x_i \right) \right)^{\alpha/4} \right) \\
  & \;\; \stackrel{k \in N_R}{\leq} c \wabsv{f} t^{\alpha/2} \sum \limits_{k \in N_R} \left( \left( \frac{t}{ \max \limits_{i: k \in \bar{R}_i} x_i } + 1 \right)^{\alpha/2} \wedge \left( 1 + \frac{ \max \limits_{i: k \in \bar{R}_i} x_i }{t} \right)^{\alpha/4} \right) \\
  & \quad \leq \; c \wabsv{f} t^{\alpha/2} 2^{\alpha/2}.
\end{align*}

For the third term $E_3$ we finally have
\begin{align*}
  E_3 &\leq \sum \limits_{k \in N_R} \min \limits_{i: k \in \bar{R}_i} \!\left\{ \ainorm{f}{i} \left( \left( \left| b_k^0 t \right|^{\alpha} x_i^{-\alpha/2} \right) \wedge \left( \left| b_k^0 t \right|^{\alpha/2} \right) \right) \right\} \\
  &\leq c \wabsv{f} \sum \limits_{k \in N_R} \left| b_k^0 t \right|^{\alpha/2} \\
  &\leq c \wabsv{f} t^{\alpha/2}.
\end{align*}

Put the above bounds on $E_1, E_2$ and $E_3$ into (\ref{equ:seventynine}) and then in (\ref{equ:seventyeight}) to conclude that
\[ | P_t (fg) (x) - (fg) (x) | \leq \left( \norm{g} c \wabsv{f} + \norm{f} |g|_{\alpha} \right) t^{\alpha/2} \]
and so by definition of the semigroup norm
\[ |fg|_{\alpha} \leq c \wabsv{f} \norm{g} + \norm{f} |g|_{\alpha}. \] 
This gives (\ref{equ:seventysix}) and (\ref{equ:seventyseven}) is then immediate. 
\end{proof}
%
% ------------------------------------------------------------
%
\begin{thm} \label{equ:theorem_19}
There exist $0 < c_1 \leq c_2$ such that
\begin{equation} \label{equ:eighty}
  c_1 \wabsv{f} \leq |f|_{\alpha} \leq c_2 \wabsv{f}.
\end{equation}
This implies that $\SC_w^{\alpha} = \SSS^{\alpha}$ and so $\SSS^{\alpha}$ contains $\SC^1$ functions with compact support in $\SSS_0$.
\end{thm}
%
% ------------------------------------------------------------
%
\begin{proof}
The idea of the proof was taken from the proof of Theorem 19 in \cite{r6}. The second inequality in (\ref{equ:eighty}) follows immediately by setting $g=1$ in Lemma \ref{equ:lemma_18}. For the first inequality let $x, h \in \SSS_0$, $t>0$ and use Propositions \ref{equ:proposition_16} and \ref{equ:proposition_17} to see that
\begin{align}
  & |f(x+h) - f(x)| \label{equ:eightyone} \\
  & \leq | P_t f (x+h) - f(x+h) | + | P_t f (x) - f(x) | + | P_t f (x+h) - P_t f (x) | \nonumber \\
  & \leq 2 |f|_{\alpha} t^{\alpha/2} + | P_t f (x+h) - P_t f (x) | \nonumber \\
  & \leq 2 |f|_{\alpha} t^{\alpha/2} + c |f|_{\alpha} t^{\frac{\alpha}{2}-\frac{1}{2}} \left( \sum \limits_{j \in N_R} \frac{|h_j|}{ \max \limits_{l \in C_j} \{ \sqrt{t+x_l} \} } + \sum \limits_{i \in N_{C2}} \frac{h_i}{ \sqrt{t+x_i} } \right), \nonumber
\end{align}
where we used the triangle inequality together with $h_l \geq 0, l \in C_j \subset N_{C2}$ for all $j \in N_R$. 

By setting $t=|h|$ and bounding $\left( \max_{l \in C_j} \{ \sqrt{t+x_l} \} \right)^{-1}$ and $\left( \sqrt{t+x_i} \right)^{-1}$ by $\left( \sqrt{t} \right)^{-1}$ we obtain as a first bound on (\ref{equ:eightyone})
\begin{equation} \label{equ:eightytwo}
  c |f|_{\alpha} |h|^{\alpha/2}.
\end{equation}
Next only consider $h \in \SSS_0$ such that there exists $i \in N_{C2}$ and $j \in \{i\} \cup \bar{R}_i$ such that $h_j \neq 0$ and $h_k=0$ if $k \notin \{i\} \cup \bar{R}_i$. (\ref{equ:eightyone}) becomes
\begin{align*}
  & |f(x+h) - f(x)| \\
  & \leq 2 |f|_{\alpha} t^{\alpha/2} + c |f|_{\alpha} t^{\frac{\alpha}{2}-\frac{1}{2}} \left( \sum \limits_{j: j \in \bar{R}_i} \frac{|h_j|}{ \max \limits_{l \in C_j} \{ \sqrt{t+x_l} \} } + \frac{h_i}{ \sqrt{t+x_i} } \right) \\
  & \leq 2 |f|_{\alpha} t^{\alpha/2} + c |f|_{\alpha} t^{\frac{\alpha}{2}-\frac{1}{2}} \frac{1}{ \sqrt{t+x_i} } |h|.
\end{align*}
In case $x_i>0$ set $t=\frac{|h|^2}{x_i}$ and bound $\left( \sqrt{t+x_i} \right)^{-1}$ by $\left( \sqrt{x_i} \right)^{-1}$ to get as a second upper bound
\begin{equation} \label{equ:eightythree}
  c |f|_{\alpha} x_i^{-\alpha/2} |h|^{\alpha}.
\end{equation}
The first inequality in (\ref{equ:eighty}) is now immediate from (\ref{equ:eightytwo}) and (\ref{equ:eightythree}) and the proof is complete. 
\end{proof}
%
% ------------------------------------------------------------
%
\begin{note}
Special care was needed when choosing $h \in \SSS_0$ in the last part of the proof as it only works for those $h$ which are to be considered in the definition of $\wabsv{\cdot}$. Note that this was the main reason to define the weighted H\"older norms for $\bar{R}_i$ instead of $R_i$.
\end{note}
%
% ------------------------------------------------------------
%
\begin{rmk} \label{equ:rmk_norms}
The equivalence of the two norms will prove to be crucial later in Section \ref{equ:section_3}, where we show the uniqueness of solutions to the martingale problem {\it MP}($\SA$,$\nu$) as stated in Theorem \ref{equ:theorem_4}. All the estimates of Section \ref{equ:section_2} are obtained in terms of the semigroup norm. In Section \ref{equ:section_3} we shall further need estimates on the norm of products of certain functions. At this point we shall have to rely on the result of Lemma \ref{equ:lemma_18} for weighted H\"older norms. The equivalence of norms now yields a similar result in terms of the semigroup norm.
\end{rmk}
%
% ============================================================
%
\subsection{Weighted H\"older bounds of certain differentiation operators applied to $P_t f$} \label{equ:section_2_4}
%
% ============================================================

The $x_j, j \in N_R$ derivatives are much easier. 
%
% ------------------------------------------------------------
%
\begin{nota}
We shall need the following slight extension of our notation for $E^{N_{C2}}$: 
\[ E^{N_{C2}} = E_{x^{N_{C2}}}^{N_{C2}} = \left( \prodm \limits_{i \in N_{C2}} P_{x_i}^i \right). \]
\end{nota}
%
% ------------------------------------------------------------
%
\begin{nota} \label{equ:notation_T_k}
To ease notation let
\[ T_k^{-\frac{1}{2}} \!\left( t, x^{N_{C2}} \right) \equiv   
  \begin{cases}
    \min \limits_{l \in C_k} \!\left\{ (t+x_l)^{-1/2} \right\}, & k \in N_R, \\
    (t+x_k)^{-1/2}, & k \in N_{C2}. \\
  \end{cases} \]
\end{nota}
%
% ------------------------------------------------------------
%
\begin{pro} \label{equ:proposition_22}
If $f$ is a bounded Borel function on $\SSS_0$, then for all $x, h \in \SSS_0$, $j \in N_R$, $i \in C_j$ and arbitrary $k \in V$,
\begin{equation} \label{equ:eightyseven}
  \left| \dx{j}P_tf (x+h_k e_k) - \dx{j}P_tf (x) \right| \leq \frac{c \norm{f}}{t^{3/2}} |h_k| T_k^{-\frac{1}{2}} \!\left( t, x^{N_{C2}} \right)
\end{equation}
and
\begin{equation} \label{equ:eightyeight}
  \left| (x+h_k e_k)_i \dfdx{P_tf}{j} (x+h_k e_k) - x_i \dfdx{P_tf}{j} (x) \right| \leq \frac{c \norm{f}}{t^{3/2}} |h_k| T_k^{-\frac{1}{2}} \!\left( t, x^{N_{C2}} \right). 
\end{equation}
If $f \in \SSS^{\alpha}$, then
\begin{equation} \label{equ:eightynine}
  \left| \dx{j}P_tf (x+h_k e_k) - \dx{j}P_tf (x) \right| \leq c |f|_{\alpha} t^{\frac{\alpha}{2}-\frac{3}{2}} |h_k| T_k^{-\frac{1}{2}} \!\left( t, x^{N_{C2}} \right) 
\end{equation}
and
\begin{equation} \label{equ:ninety}
  \left| (x+h_k e_k)_i \dfdx{P_tf}{j} (x+h_k e_k) - x_i \dfdx{P_tf}{j} (x) \right| \leq c |f|_{\alpha} t^{\frac{\alpha}{2}-\frac{3}{2}} |h_k| T_k^{-\frac{1}{2}} \!\left( t, x^{N_{C2}} \right). 
\end{equation}
\end{pro}
%
% ------------------------------------------------------------
%
\begin{proof}
The focus will be on proving (\ref{equ:eightyeight}) as (\ref{equ:eightyseven}) is simpler. Again, it suffices to consider $f$ bounded and continuous. For increments in $x_k$, $k \in N_R$ the statement follows as in the proof of \cite{r6}, Proposition 22. 
%
% additional part starts

Indeed, from (\ref{equ:sixtyone}) and Lemma \ref{equ:lemma_11}(b) we have
\begin{align*}
  & \left| x_i \dfdx{P_tf}{j} (x+h_k e_k) - x_i \dfdx{P_tf}{j} (x) \right| \\
  & = x_i \left| E^{N_{C2}} \left[ \frac{\partial^2}{\partial x_j^2} G_{t,x^{N_R}+h_k e_k} \left( I_t^{N_R} , x_t^{N_{C2}} \right) - \frac{\partial^2}{\partial x_j^2} G_{t,x^{N_R}} \left( I_t^{N_R} , x_t^{N_{C2}} \right) \right] \right| \\
  & \leq x_i c \norm{f} E^{N_{C2}} \left[ \left( I_t^{(j)} \right)^{-1} \left( I_t^{(k)} \right)^{-1/2} \right] |h_k| \\
  & \leq x_i c \norm{f} \sqrt{ E^{N_{C2}} \left[ \left( I_t^{(j)} \right)^{-2} \right] } \sqrt{ E^{N_{C2}} \left[ \left( I_t^{(k)} \right)^{-1} \right] } |h_k| \\
  & \stackrel{j, k \in N_R}{\leq} x_i c \norm{f} t^{-1} \min \limits_{l \in C_j} \{ (t+x_l)^{-1} \} t^{-1/2} \min \limits_{l \in C_k} \{ (t+x_l)^{-1/2} \} |h_k| \\
  & \hspace{80mm} \mbox{ by Lemma \ref{equ:lemma_7}(b) } \\
  & \leq c \norm{f} t^{-3/2} |h_k| \min \limits_{l \in C_k} \{ (t+x_l)^{-1/2} \} \mbox{ as } i \in C_j.
\end{align*}
Note that at this point (and later), $i \in C_j$ is essential to make the proof work. 
%
% additional part ends

Now consider increments in $x_k$, $k \in N_{C2}$. We start with observing that for $h_k \geq 0$
\begin{align*}
  & (x_i + \delta_{ki} h_i) \dfdx{P_tf}{j} (x+h_k e_k) - x_i \dfdx{P_tf}{j} (x) \\
  & = \delta_{ki} h_i E_{x^{N_{C2}}+h_k e_k}^{N_{C2}} \!\left[ \frac{\partial^2}{\partial x_j^2} G_{t,x^{N_R}} \!\left( I_t^{N_R} , x_t^{N_{C2}} \right) \right] \\
  & \qquad + x_i \left( E_{x^{N_{C2}}+h_k e_k}^{N_{C2}} - E_{x^{N_{C2}}}^{N_{C2}} \right) \!\left[ \frac{\partial^2}{\partial x_j^2} G_{t,x^{N_R}} \!\left( I_t^{N_R} , x_t^{N_{C2}} \right) \right] \\
  & \equiv E_1 + E_2,
\end{align*}
by arguing as in the proof of \cite{r6}, Proposition 22. The bound on $E_1$ is derived as in that proof, using Lemmas \ref{equ:lemma_11}(a) and \ref{equ:lemma_7}(b), namely
%
% additional part starts
%
\begin{align*}
  |E_1| &\leq \delta_{ki} h_i c \norm{f} E_{x^{N_{C2}}+h_k e_k}^{N_{C2}} \left[ \left( I_t^{(j)} \right)^{-1} \right] \\
  &\stackrel{j \in N_R}{\leq} \delta_{ki} h_i c \norm{f} t^{-1} \min \limits_{l \in C_j} \{ (t+x_l)^{-1} \} \\
  &\leq \delta_{ki} h_k c \norm{f} t^{-1} (t+x_k)^{-1} \\
  &\leq \delta_{ki} h_k c \norm{f} t^{-3/2} (t+x_k)^{-1/2},
\end{align*}
where we used $i \in C_j$.
%
% additional part ends

For $E_2$ we use the decompositions (\ref{equ:thirtythree}), (\ref{equ:thirtyfour}), (\ref{equ:thirtyfive}) and notation from Lemma \ref{equ:lemma_10} with $\rho=\frac{1}{2}$. Recall the notation $G_{t,x^{N_R}}^k$ from (\ref{equ:more_defs}) and the definition of $R_k$ and $S_k$ as in Notation \ref{equ:nota_R_S}. Then 
\begin{align*}
  |E_2| =& \; x_i \left| E \!\left[ \frac{\partial^2}{\partial x_j^2} G_{t,x^{N_R}}^k \!\left( x^{N_{C2}}, R_{N_t'} + I_2(t) + I_3^h(t), S_{N_t'} + X_0'(t) \right) \right. \right. \\
  & \qquad \quad \left. \left. - \frac{\partial^2}{\partial x_j^2} G_{t,x^{N_R}}^k \!\left( x^{N_{C2}}, R_{N_t} + I_2(t), S_{N_t} + X_0'(t) \right) \right] \right| \\
  \leq& \; x_i \left| E \!\left[ \frac{\partial^2}{\partial x_j^2} G_{t,x^{N_R}}^k \!\left( x^{N_{C2}}, R_{N_t'} + I_2(t) + I_3^h(t), S_{N_t'} + X_0'(t) \right) \right. \right. \\
  & \qquad \quad \left. \left. - \frac{\partial^2}{\partial x_j^2} G_{t,x^{N_R}}^k \!\left( x^{N_{C2}}, R_{N_t'} + I_2(t), S_{N_t'} + X_0'(t) \right) \right] \right| \\
  & + x_i \left| E \!\left[ \frac{\partial^2}{\partial x_j^2} G_{t,x^{N_R}}^k \!\left( x^{N_{C2}}, R_{N_t'} + I_2(t), S_{N_t'} + X_0'(t) \right) \right. \right. \\
  & \qquad \qquad \left. \left. - \frac{\partial^2}{\partial x_j^2} G_{t,x^{N_R}}^k \!\left( x^{N_{C2}}, R_{N_t} + I_2(t), S_{N_t} + X_0'(t) \right) \right] \right| \\
  \equiv& \; E_{2a} + E_{2b}.
\end{align*}

$E_{2a}$ can be bounded as in \cite{r6}, using Lemmas \ref{equ:lemma_11}(b) and \ref{equ:lemma_7}(b), and the independence of $x^{N_{C2}}$ and $I_3^h(t)$. Indeed,
%
% additional part starts
%
\begin{align*}
  E_{2a} &\leq x_i c \norm{f} E^{N_{C2}} \left[ \left( I_t^{(j)} \right)^{-1} \sum \limits_{l: l \in \bar{R}_k} \left( I_t^{(l)} \right)^{-1} \right] E\!\left[ I_3^h(t) \right] \\
  &\leq x_i c \norm{f} t^{-1} \min \limits_{m \in C_j} \{ (t+x_m)^{-1} \} \sum \limits_{l: l \in \bar{R}_k} t^{-1} \min \limits_{m \in C_l} \{ (t+x_m)^{-1} \} \\
  & \; \times \, h_k \int \int_0^t \nu_s ds d\NN_0(\nu) \\
  &\leq x_i c \norm{f} t^{-1} (t+x_i)^{-1} t^{-1} (t+x_k)^{-1} h_k t \\
  &\leq c \norm{f} t^{-3/2} h_k (t+x_k)^{-1/2}, 
\end{align*}
where (\ref{equ:twentysix}) is used in the next to last line as well as $i \in C_j$. Also note that for $k \in N_2$, $E_{2a}=0$ by using the definition of $G_{t,x^{N_R}}^k$ and arguing similarly to (\ref{equ:fourtythree}) and (\ref{equ:fourtyfour}).
%
% additional part ends

Next turn to $E_{2b}$. Recall that $S_n = S_n(t) = \sum_{l=1}^n e_l(t)$, $R_n = R_n(t) = \sum_{l=1}^n r_l(t)$ and $p_k(w) = e^{-w} w^k / k!$. In the first term of $E_{2b}$ we may condition on $N_t'$ as it is independent from the other random variables and in the second term we do the same for $N_t$. Thus, if $w' = w + \frac{h_k}{\gamma_k^0 t}$ and $w = \frac{x_k}{2 \gamma_k^0 t}$, then by Lemma \ref{equ:lemma_11}(a) and Lemma \ref{equ:lemma_7}(b),
\begin{align*}
  & E_{2b} \\
  & = \; x_i \biggl| \sum \limits_{n=0}^{\infty} \left( p_n\!\left( w' \right) - p_n(w) \right) E \biggl[ \frac{\partial^2}{\partial x_j^2} G_{t,x^{N_R}}^k \!\left( x^{N_{C2}}, R_n + I_2(t), S_n + X_0'(t) \right) \biggr] \biggr| \\
  & \leq \; c x_i \sum \limits_{n=0}^{\infty} \left| \int_w^{w'} p_n'(u) du \right| \norm{f} \\
  & \quad \left. \; \times \,
    \begin{cases}
      E^{N_{C2}} \!\left[ \left( I_t^{(j)} \right)^{-1} \right], & k \notin C_j, \\
      \min \limits_{i \in C_j \backslash \{k\} } \!\left\{ E^{N_{C2}} \!\left[ \left( \int_0^t x_s^{(i)} ds \right)^{-1} \right] \right\} \wedge E \!\left[ \left( \int_0^t X_0'(s) ds \right)^{-1} \right] , & k \in C_j
    \end{cases} \right\} \\
  & \leq \; c \norm{f} x_i \sum \limits_{n=0}^{\infty} \left| \int_w^{w'} p_n'(u) du \right| t^{-1} \min \limits_{l \in C_j} \!\left\{ (t+x_l)^{-1} \right\},
\end{align*}
where we used that $X_0'$ starts at $\frac{x_k}{2}$ and thus by Lemma \ref{equ:lemma_7}(b)
\[ E \!\left[ \left( \int_0^t X_0'(s) ds \right)^{-1} \right] \leq c t^{-1} \left( t+\frac{x_k}{2} \right)^{-1} \leq c t^{-1} (t+x_k)^{-1}. \]
We therefore obtain with $i \in C_j$ 
\begin{align*}
  E_{2b} &\leq c \norm{f} x_i \int_w^{w'} \sum \limits_{n=0}^{\infty} p_n(u) \frac{|n-u|}{u} du t^{-1} (t+x_i)^{-1} \\
  &\leq c \norm{f} \left( \left( \int_w^{w'} \frac{1}{\sqrt{u}} du \right) \wedge \left( \int_w^{w'} 2 du \right) \right) t^{-1}, 
\end{align*}
where we used $\sum_{n=0}^{\infty} p_n(u) \frac{|n-u|}{u} = \frac{1}{u} E|N-u| \leq \frac{1}{u} \sqrt{E|N-u|^2} = \frac{1}{\sqrt{u}}$ and $\sum_{n=0}^{\infty} p_n(u) \frac{|n-u|}{u} \leq \sum_{n=0}^{\infty} p_n(u) \left( \frac{n}{u} + 1 \right) = \frac{E|N|}{u} + 1 = 2$ with $N$ being Poisson distributed with parameter $u$. Hence
\[ E_{2b} \leq c \norm{f} (w'-w) \left( \frac{1}{\sqrt{w}} \wedge 2 \right) t^{-1} = c \norm{f} \frac{h_k}{t} \left( \frac{\sqrt{t}}{\sqrt{x_k}} \wedge 2 \right) t^{-1}. \]
As $\left( \frac{1}{\sqrt{x_k}} \wedge \frac{2}{\sqrt{t}} \right) \leq c \frac{1}{\sqrt{t+x_k}}$ we finally get
\[ E_{2b} \leq c \norm{f} t^{-3/2} h_k (t+x_k)^{-1/2}. \] 

The bounds (\ref{equ:eightynine}) and (\ref{equ:ninety}) can be derived from the first two by an argument similar to the one used in the proof of Proposition \ref{equ:proposition_16} (alternatively refer to the end of the proof of Proposition 22 in \cite{r6}).
\end{proof} 
%
% ------------------------------------------------------------
%

In what follows recall Notation \ref{equ:notation_T_k}.
%
% ------------------------------------------------------------
%
\begin{pro} \label{equ:proposition_23}
If $f$ is a bounded Borel function on $\SSS_0$, then for all $x, h \in \SSS_0$, $i \in N_{C2}$ and arbitrary $k \in V$, 
\begin{equation} \label{equ:ninetytwo}
  \left| \dx{i}P_tf (x+h_k e_k) - \dx{i}P_tf (x) \right| \leq \frac{c \norm{f}}{t^{3/2}} |h_k| T_k^{-\frac{1}{2}} \!\left( t, x^{N_{C2}} \right)
\end{equation}
and
\begin{equation} \label{equ:ninetythree}
  \left| (x+h_k e_k)_i \dfdx{P_tf}{i} (x+h_k e_k) - x_i \dfdx{P_tf}{i} (x) \right| \leq \frac{c \norm{f}}{t^{3/2}} |h_k| T_k^{-\frac{1}{2}} \!\left( t, x^{N_{C2}} \right). 
\end{equation}
If $f \in \SSS^{\alpha}$, then
\[ \left| \dx{i}P_tf (x+h_k e_k) - \dx{i}P_tf (x) \right| \leq c |f|_{\alpha} t^{\frac{\alpha}{2}-\frac{3}{2}} |h_k| T_k^{-\frac{1}{2}} \!\left( t, x^{N_{C2}} \right) \]
and
\[ \left| (x+h_k e_k)_i \dfdx{P_tf}{i} (x+h_k e_k) - x_i \dfdx{P_tf}{i} (x) \right| \leq c |f|_{\alpha} t^{\frac{\alpha}{2}-\frac{3}{2}} |h_k| T_k^{-\frac{1}{2}} \!\left( t, x^{N_{C2}} \right). \]
\end{pro}
%
% ------------------------------------------------------------
%
\begin{proof}
Proposition \ref{equ:proposition_23} is an extension of Proposition 23 in \cite{r6}. The last two inequalities follow from the first two by an argument similar to the one used in the proof of Proposition \ref{equ:proposition_16} (alternatively refer to the end of the proof of Proposition 22 in \cite{r6}). As the proof of (\ref{equ:ninetytwo}) is similar to, but much easier than, that of (\ref{equ:ninetythree}), we only prove the latter. As usual we may assume $f$ is bounded and continuous.

Recall the notation $\Delta G_{t,x^{N_R}}^{+i,+i}(X,\nu,\nu')$ from (\ref{equ:fourtyfive}). Proposition \ref{equ:proposition_14} gives
\begin{equation} \label{equ:ninetysix}
  (P_t f)_{ii} (x) = \sum \limits_{n=1}^4 E^{N_{C2}} \!\left[ \Delta_n G_{t,x^{N_R}} \!\left( x^{N_{C2}} \right) \right], 
\end{equation}
where
\begin{align*}
  \Delta_1 G_{t,x^{N_R}}(X) &\equiv \int \int \Delta G^{+i,+i}_{t,x^{N_R}} (X,\nu,\nu') \1_{\{\nu_t=\nu_t'=0\}} d\NN_0(\nu) d\NN_0(\nu'), \\
  \Delta_2 G_{t,x^{N_R}}(X) &\equiv \int \int \Delta G^{+i,+i}_{t,x^{N_R}} (X,\nu,\nu') \1_{\{\nu_t>0,\nu_t'=0\}} d\NN_0(\nu) d\NN_0(\nu'), \\
  \Delta_3 G_{t,x^{N_R}}(X) &\equiv \int \int \Delta G^{+i,+i}_{t,x^{N_R}} (X,\nu,\nu') \1_{\{\nu_t=0,\nu_t'>0\}} d\NN_0(\nu) d\NN_0(\nu')
\end{align*}
and
\begin{align*}
  \Delta_4 G_{t,x^{N_R}}(X) &\equiv \int \int \Delta G^{+i,+i}_{t,x^{N_R}} (X,\nu,\nu') \1_{\{\nu_t>0,\nu_t'>0\}} d\NN_0(\nu) d\NN_0(\nu') \\
  &\stackrel{(*)}{=} \frac{c}{t^2} \int \int \Delta G^{+i,+i}_{t,x^{N_R}} (X,\nu,\nu') \1_{\{\nu_t>0,\nu_t'>0\}} dP_t^*(\nu) dP_t^*(\nu').
\end{align*}

Let us consider first the increments in $x_k, k \in N_{C2}$. Increments in $x_k, k \in N_R$ will follow at the end of this section in Lemma \ref{equ:lemma_27}. Let $h_k \geq 0$ and use (\ref{equ:ninetysix}) to obtain
\begin{align}
  & \left| (x+h_k e_k)_i (P_tf)_{ii} (x+h_k e_k) - x_i (P_tf)_{ii} (x) \right| \label{equ:ninetyseven} \\
  & \leq \sum \limits_{n=1}^4 \left| x_i \left( E_{x^{N_{C2}} + h_k e_k}^{N_{C2}} - E_{x^{N_{C2}}}^{N_{C2}} \right) \!\left[ \Delta_n G_{t,x^{N_R}}\!\left( x^{N_{C2}} \right) \right] \right| \nonumber \\
  & \qquad + h_k \left| (P_t f)_{kk}(x+h_k e_k) \right|. \nonumber
\end{align}
The last term on the right hand side can be bounded via (\ref{equ:seventyone}) as follows:
\[ h_k \left| (P_t f)_{kk}(x+h_k e_k) \right| \leq h_k \frac{c \norm{f}}{t(t+x_k)} \leq c \norm{f} h_k t^{-3/2} (t+x_k)^{-1/2}, \]
where we used $h_k \geq 0$.

In the following Lemmas \ref{equ:lemma_24}, \ref{equ:lemma_25} and \ref{equ:lemma_26} we again use the decompositions from Lemma \ref{equ:lemma_10} with $\rho=\frac{1}{2}$ to bound the first four terms in (\ref{equ:ninetyseven}).
%
% ------------------------------------------------------------
%
\begin{lem} \label{equ:lemma_24}
For $k \in N_{C2}$ (and $i \in N_{C2}$) we have
\[ \left| x_i \left( E_{x^{N_{C2}} + h_k e_k}^{N_{C2}} - E_{x^{N_{C2}}}^{N_{C2}} \right) \!\left[ \Delta_1 G_{t,x^{N_R}}\!\left( x^{N_{C2}} \right) \right] \right| \leq \frac{ c \norm{f} }{ t^{3/2} (t+x_k)^{1/2} } h_k. \]
\end{lem}
%
% ------------------------------------------------------------
%
\begin{proof}
This Lemma corresponds to Lemma 24 in \cite{r6}. In \cite{r6} one considered $\Delta G^{+i,+i}(\cdot)$ as a second order difference, thus obtaining terms involving $(t+x_i)^{-2}$. In our setting this method will not work for $i \neq k$ as we do in fact need terms of the form $(t+x_i)^{-1} (t+x_k)^{-1}$. Instead, we shall bound the left hand side by reasoning as for the $E_2$-term in Proposition 22 of \cite{r6} (part of the proof can be found in this paper in the proof of Proposition \ref{equ:proposition_22}), but with $\frac{\partial^2}{\partial x_j^2} G(\cdot), j \in N_R$ replaced by $\Delta G^{+i,+i}(\cdot), i \in N_{C2}$. 
%
% additional part starts

Let
\[ E \equiv \left| x_i \left( E_{x^{N_{C2}} + h_k e_k}^{N_{C2}} - E_{x^{N_{C2}}}^{N_{C2}} \right) \left[ \Delta_1 G_{t,x^{N_R}}\left( x^{N_{C2}} \right) \right] \right|. \]
First bound $E$ as follows (using the obvious analogies to former notations and Lemma \ref{equ:lemma_10} with $\rho=\frac{1}{2}$).
\begin{align*}
  E &= x_i \left| E \left[ \Delta_1 G_{t,x^{N_R}}^k \left( x^{N_{C2}}, \sum \limits_{l=1}^{N_t'} r_l(t) + I_2(t) + I_3^h(t), \sum \limits_{l=1}^{N_t'} e_l(t) + X_0'(t) \right) \right. \right. \\
  & \qquad \quad \left. \left. - \Delta_1 G_{t,x^{N_R}}^k \left( x^{N_{C2}}, \sum \limits_{l=1}^{N_t} r_l(t) + I_2(t), \sum \limits_{l=1}^{N_t} e_l(t) + X_0'(t) \right) \right] \right| 
\end{align*}
\begin{align*}
  &\leq x_i \left| E \left[ \Delta_1 G_{t,x^{N_R}}^k \left( x^{N_{C2}}, \sum \limits_{l=1}^{N_t'} r_l(t) + I_2(t) + I_3^h(t), \sum \limits_{l=1}^{N_t'} e_l(t) + X_0'(t) \right) \right. \right. \\
  & \qquad \quad \left. \left. - \Delta_1 G_{t,x^{N_R}}^k \left( x^{N_{C2}}, \sum \limits_{l=1}^{N_t'} r_l(t) + I_2(t), \sum \limits_{l=1}^{N_t'} e_l(t) + X_0'(t) \right) \right] \right| \\
  & + x_i \left| E \left[ \Delta_1 G_{t,x^{N_R}}^k \left( x^{N_{C2}}, \sum \limits_{l=1}^{N_t'} r_l(t) + I_2(t), \sum \limits_{l=1}^{N_t'} e_l(t) + X_0'(t) \right) \right. \right. \\
  & \qquad \qquad \left. \left. - \Delta_1 G_{t,x^{N_R}}^k \left( x^{N_{C2}}, \sum \limits_{l=1}^{N_t} r_l(t) + I_2(t), \sum \limits_{l=1}^{N_t} e_l(t) + X_0'(t) \right) \right] \right| \\
  &\equiv E_a + E_b.
\end{align*}

Observe that $E_a=0$ for $k \in N_2$. If $k \notin N_2$ but $i \in N_2$, then \\ $\Delta G^{+i,+i}_{t,x^{N_R}}(X,\nu,\nu') \1_{\{\nu_t=\nu_t'=0\}} = 0$ (see (\ref{equ:fourtynine})) and thus $\Delta_1 G_{t,x^{N_R}}^k(\cdot) \equiv 0$. If $i, k \in N_C$ the integrand in $E_a$ is a third order difference to which we may apply Lemma \ref{equ:lemma_13} and Lemma \ref{equ:lemma_11}(b). Together with the independence of $x^{N_{C2}}$ and $\Xi^h$ we conclude that
\begin{align*}
  E_a &\leq x_i c \norm{f} \sum \limits_{j_1: j_1 \in \bar{R}_i} \sum \limits_{j_2: j_2 \in \bar{R}_i} \sum \limits_{j_3: j_3 \in \bar{R}_k} E^{N_{C2}} \left[ \left( I_t^{(j_1)} \right)^{-1} \left( I_t^{(j_2)} \right)^{-1} \left( I_t^{(j_3)} \right)^{-1} \right] \\
  & \; \times \, E\!\left[ I_3^h(t) \right] \int \int_0^t \nu_s ds d\NN_0(\nu) \int \int_0^t \nu_s' ds d\NN_0(\nu') \\
  &\leq x_i c \norm{f} t^{-1} (t+x_i)^{-1} t^{-1} (t+x_i)^{-1} t^{-1} (t+x_k)^{-1} h_k t t^2 \\
  &\leq c \norm{f} h_k t^{-3/2} (t+x_k)^{-1/2}.
\end{align*}
This holds in particular for all $i, k \in N_{C2}$ by our precedent observations. 

By arguing as for $E_a$ we can assume w.l.o.g. that $i \in N_C$ when estimating $E_b$. The independence of $N_t'$ from the other random variables appearing in the first term in $E_b$ allows us to condition on its value as already done in Proposition \ref{equ:proposition_22}. The same is true for $N_t$ in the second term in $E_b$. Let $w = \frac{x_k}{2 \gamma_k^0 t}$ and $w' = w + \frac{h_k}{\gamma_k^0 t}$ and note that $\Delta_1 G_{t,x^{N_R}}$ is a second order difference to which we may apply Lemma \ref{equ:lemma_13}. Together with Lemma \ref{equ:lemma_11}(b) and Lemma \ref{equ:lemma_7}(b) this gives
\begin{align*}
  E_b &= x_i \left| \sum \limits_{n=0}^{\infty} \left( p_n(w') - p_n(w) \right) \right. \\
  & \left. \; \times \, E \left[ \Delta_1 G_{t,x^{N_R}}^k \left( x^{N_{C2}}, R_n(t) + I_2(t), S_n(t) + X_0'(t) \right) \right] \right| \\
  &\leq c x_i \sum \limits_{n=0}^{\infty} \left| \int_w^{w'} p_n'(u) du \right| \norm{f} \sum \limits_{j_1: j_1 \in \bar{R}_i} \sum \limits_{j_2: j_2 \in \bar{R}_i} \\
  & \quad \left. E \left[
    \begin{cases}
      \left( I_t^{(j_1)} \right)^{-1}, & k \notin C_{j_1}, \\
      \min \limits_{i \in C_{j_1} \backslash \{k\} } \left\{ \left( \int_0^t x_s^{(i)} ds \right)^{-1} \right\} \wedge & \ \\
      \qquad \left( \int_0^t X_0'(s) ds \right)^{-1} , & k \in C_{j_1}
    \end{cases} \right\} \right. \\
  & \quad \left. \times \left.
    \begin{cases}
      \left( I_t^{(j_2)} \right)^{-1}, & k \notin C_{j_2}, \\
      \min \limits_{i \in C_{j_2} \backslash \{k\} } \left\{ \left( \int_0^t x_s^{(i)} ds \right)^{-1} \right\} \wedge & \  \\
      \qquad \left( \int_0^t X_0'(s) ds \right)^{-1} , & k \in C_{j_2}
    \end{cases} \right\} \right] \left( \int \int_0^t \nu_s ds d\NN_0(\nu) \right)^2 \\
  &\leq c \norm{f} x_i \sum \limits_{n=0}^{\infty} \left| \int_w^{w'} p_n'(u) du \right| t^{-2} (t+x_i)^{-2} t^2,
\end{align*}
where we used again that $X_0'$ starts at $\frac{x_k}{2}$ and thus by Lemma \ref{equ:lemma_7}(b) \\
$E \left[ \left( \int_0^t X_0'(s) ds \right)^{-1} \right] \leq c t^{-1} (t+\frac{x_k}{2})^{-1} \leq c t^{-1} (t+x_k)^{-1}$. 

Now proceed as in the estimation of $E_{2b}$ in the proof of Proposition \ref{equ:proposition_22} to obtain 
\begin{align*}
  E_b &\leq c \norm{f} t^{-1/2} h_k (t+x_k)^{-1/2} t^{-2} (t+x_i)^{-1} t^2 \\
  &\leq c \norm{f} h_k t^{-3/2} (t+x_k)^{-1/2}.
\end{align*}
The above bounds on $E_a$ and $E_b$ give the required result. 
%
% additional part ends
%
\end{proof}
%
% ------------------------------------------------------------
%
\begin{lem} \label{equ:lemma_25}
For $k \in N_{C2}$ (and $i \in N_{C2}$) and $n=2,3$ we have
\begin{equation} \label{equ:def_E}
  \left| x_i \left( E_{x^{N_{C2}} + h_k e_k}^{N_{C2}} - E_{x^{N_{C2}}}^{N_{C2}} \right) \!\left[ \Delta_n G_{t,x^{N_R}}\!\left( x^{N_{C2}} \right) \right] \right| \leq \frac{ c \norm{f} }{ t^{3/2} (t+x_k)^{1/2} } h_k. 
\end{equation}
\end{lem}
%
% ------------------------------------------------------------
%
\begin{proof}
By symmetry we only need to consider $n=2$. As before let $w = \frac{x_k}{2 \gamma_k^0 t}$, $w' = w + \frac{h_k}{\gamma_k^0 t}$, $S_n = \sum_{l=1}^n e_l(t)$ and $R_n = \sum_{l=1}^n r_l(t)$. Let $Q_h$ be the law of $I_3^h(t)$ as defined after (\ref{equ:thirtyfive}). As this random variable is independent of the others appearing below we may condition on it and use (\ref{equ:thirtythree}), (\ref{equ:thirtyfour}) and (\ref{equ:thirtyfive}) to conclude
\begin{align*}
  & x_i E_{x^{N_{C2}} + h_k e_k}^{N_{C2}} \!\left[ \Delta_2 G_{t,x^{N_R}} \!\left( x^{N_{C2}} \right) \right] \\
  & = x_i E \biggl[ \int \int \int \biggl\{ G_{t,x^{N_R}}^{k,+i,+i} \biggl( x^{N_{C2}}, I_2(t) + z + R_{N_t'}, X_0'(t) + S_{N_t'}; \\
  & \qquad \qquad \qquad \qquad \qquad \qquad \qquad \qquad \qquad \quad \int_0^t \nu_s ds, \nu_t, \int_0^t \nu_s' ds, 0 \biggr) \\
  & \quad - G_{t,x^{N_R}}^{k,+i,+i} \!\left( x^{N_{C2}}, I_2(t) + z + R_{N_t'}, X_0'(t) + S_{N_t'}; 0, 0, \int_0^t \nu_s' ds, 0 \right) \\
  & \quad - G_{t,x^{N_R}}^{k,+i,+i} \!\left( x^{N_{C2}}, I_2(t) + z + R_{N_t'}, X_0'(t) + S_{N_t'}; \int_0^t \nu_s ds, \nu_t, 0, 0 \right) \\
  & \quad + G_{t,x^{N_R}}^{k,+i,+i} \Bigl( x^{N_{C2}}, I_2(t) + z + R_{N_t'}, X_0'(t) + S_{N_t'}; 0, 0, 0, 0 \Bigr) \biggr\} \\
  & \qquad \qquad \times \1_{\{\nu_t>0\}} \1_{\{\nu_t'=0\}} d\NN_0(\nu) d\NN_0(\nu') dQ_h(z) \biggr].
\end{align*}
When working under $E_{x^{N_{C2}}}^{N_{C2}}$ there is no $I_3^h(t)$ term. Hence we obtain the same formula with $z$ replaced by $0$ and $N_t'$ replaced by $N_t$. The difference of these terms can be bounded by a difference dealing with the change from $z$ to $0$ and the change from $N_t'$ to $N_t$ separately. For the second term we recall that $p_n(u)=e^{-u} u^n / n!$ and observe that $N_t'$ is independent of the other random variables. Hence we may condition on its value to see that the l.h.s. of (\ref{equ:def_E}) is at most
\begin{align*}
  & x_i \biggl| E \biggl[ \int \int \int \biggl\{ \Delta G_{t,x^{N_R}}^{k,+i,+i} \Bigl( x^{N_{C2}}, I_2(t) + z + R_{N_t'}, X_0'(t) + S_{N_t'}; \\
  & \qquad \qquad \qquad \qquad \qquad \qquad \qquad \qquad \qquad \qquad \quad \int_0^t \nu_s ds, \nu_t, \int_0^t \nu_s' ds, 0 \Bigr) \\
  & \quad \; - \Delta G_{t,x^{N_R}}^{k,+i,+i} \!\left( x^{N_{C2}}, I_2(t) + R_{N_t'}, X_0'(t) + S_{N_t'}; \int_0^t \nu_s ds, \nu_t, \int_0^t \nu_s' ds, 0 \right) \biggr\} \\
  & \quad \qquad \quad \times \1_{\{\nu_t>0\}} \1_{\{\nu_t'=0\}} d\NN_0(\nu) d\NN_0(\nu') dQ_h(z) \biggr] \biggr| \\
  & \; \; + \, x_i \biggl| \sum \limits_{n=0}^{\infty} (p_n(w')-p_n(w)) E \biggl[ \int \int \Delta G_{t,x^{N_R}}^{k,+i,+i} \Bigl( x^{N_{C2}}, \\
  & \qquad \qquad \qquad \qquad \qquad \quad I_2(t) + R_n, X_0'(t) + S_n; \int_0^t \nu_s ds, \nu_t, \int_0^t \nu_s' ds, 0 \Bigr) \\
  & \quad \qquad \qquad \qquad \qquad \qquad \qquad \times \1_{\{\nu_t>0\}} \1_{\{\nu_t'=0\}} d\NN_0(\nu) d\NN_0(\nu') \biggr] \biggr| \\
  & \equiv E_a + E_b.
\end{align*}
The first term can be rewritten as the sum of two second order differences (one in $z$, one in $\int_0^t \nu_s' ds$). Together with Lemma \ref{equ:lemma_13}, Lemma \ref{equ:lemma_11}(b) and Lemma \ref{equ:lemma_7}(b) we therefore obtain (terms including empty sums are again understood as being zero)
\begin{align*}
  E_a \leq & \; 2 x_i c \norm{f} \sum \limits_{j_1: j_1 \in \bar{R}_i} \sum \limits_{j_2: j_2 \in \bar{R}_k} \\
  & \quad E \!\left[ \left.
    \begin{cases}
      \left( I_t^{(j_1)} \right)^{-1}, & k \notin C_{j_1}, \\
      \min \limits_{m \in C_{j_1} \backslash \{k\} } \!\left\{ \left( \int_0^t x_s^{(m)} ds \right)^{-1} \right\} \wedge \left( \int_0^t X_0'(s) ds \right)^{-1} , & k \in C_{j_1}
    \end{cases} \right\} \right. \\
  & \left. \qquad \quad \times \; \left( \int_0^t X_0'(s) ds \right)^{-1} \right] \int \int_0^t \nu_s' ds d\NN_0(\nu') \NN_0[\nu_t>0] \int z dQ_h(z) \\
  &\leq \; x_i c \norm{f} t^{-2} (t+x_i)^{-1} (t+x_k)^{-1} t t^{-1} h_k t \\
  &\leq \; c \norm{f} h_k t^{-3/2} (t+x_k)^{-1/2}.
\end{align*}
Turning to $E_b$ observe that we have the sum of two first order differences (both in $\int_0^t \nu_s' ds$). Together with the triangle inequality, Lemma \ref{equ:lemma_11}(b) and Lemma \ref{equ:lemma_7}(b) we therefore obtain
\begin{align*}
  E_b \leq& \; c x_i \sum \limits_{n=0}^{\infty} \left| \int_w^{w'} p_n'(u) du \right| \norm{f} \sum \limits_{j_1: j_1 \in \bar{R}_i} \\
  & \left. E \!\left[
    \begin{cases}
      \left( I_t^{(j_1)} \right)^{-1}, & k \notin C_{j_1}, \\
      \min \limits_{m \in C_{j_1} \backslash \{k\} } \!\left\{ \left( \int_0^t x_s^{(m)} ds \right)^{-1} \right\} \wedge \left( \int_0^t X_0'(s) ds \right)^{-1} , & k \in C_{j_1}
    \end{cases} \right\} \right] \\
  & \; \times \, \int \int_0^t \nu_s' ds d\NN_0(\nu') \NN_0[\nu_t>0] \\
  \leq& \; c x_i \sum \limits_{n=0}^{\infty} \left| \int_w^{w'} p_n'(u) du \right| \norm{f} t^{-1} (t+x_i)^{-1} t t^{-1}.
\end{align*}

Now proceed again as in the estimation of $E_{2b}$ in the proof of Proposition \ref{equ:proposition_22} to get
\begin{align*}
  E_b &\leq c x_i t^{-1/2} h_k (t+x_k)^{-1/2} \norm{f} t^{-1} (t+x_i)^{-1} t t^{-1} \\
  &\leq c \norm{f} h_k t^{-3/2} (t+x_k)^{-1/2}.
\end{align*}
The above bounds on $E_a$ and $E_b$ give the required result. 
\end{proof}
%
% ------------------------------------------------------------
%
\begin{nota}
Let
\[ G_{t,x^{N_R}}^{m,n \neq m} \!\left( X, Y_t, Z_t, Y_t', Z_t' \right) \equiv 
  \begin{cases}
    G_{t,x^{N_R}}^{m,n} \!\left( X, Y_t, Z_t, Y_t', Z_t' \right) & \mbox{ if } n \neq m \cr
    G_{t,x^{N_R}}^m \!\left( X, Y_t, Z_t\right) & \mbox{ if } n = m.
  \end{cases} \]
Expressions such as $G_{t,x^{N_R}}^{m,n \neq m,+k,+l} \!\left( X, Y_t, Z_t, Y_t', Z_t';\int_0^t \eta_s ds,\theta_t,\int_0^t \eta_s' ds,\theta_t' \right)$ \\ will be defined similarly. 
\end{nota}
%
% ------------------------------------------------------------
%
\begin{lem} \label{equ:lemma_26}
For $k \in N_{C2}$ (and $i \in N_{C2}$) we have
\[ \left| x_i \left( E_{x^{N_{C2}} + h_k e_k}^{N_{C2}} - E_{x^{N_{C2}}}^{N_{C2}} \right) \!\left[ \Delta_4 G_{t,x^{N_R}}\!\left( x^{N_{C2}} \right) \right] \right| \leq \frac{ c \norm{f} }{ t^{3/2} (t+x_k)^{1/2} } h_k. \]
\end{lem}
%
% ------------------------------------------------------------
%
\begin{proof}
Let
\begin{equation} \label{equ:error_in_case_4}
  E \equiv x_i \left| \left( E_{x^{N_{C2}} + h_k e_k}^{N_{C2}} - E_{x^{N_{C2}}}^{N_{C2}} \right) \!\left[ \Delta_4 G_{t,x^{N_R}}\!\left( x^{N_{C2}} \right) \right] \right|. 
\end{equation}
We use the same setting and notation as in Lemma \ref{equ:lemma_25}. 
Proceeding as in the estimation of the l.h.s. in (\ref{equ:def_E}), thereby not only decomposing $x^{(k)}$ but also $x^{(i)}$ (the respective parts of the decomposition of $x^{(k)}$ and $x^{(i)}$ are designated via upper indices $k$ resp. $i$ and are independent for $k \neq i$), we have
\begin{align*}
  & x_i E_{x^{N_{C2}} + h_k e_k}^{N_{C2}} \!\left[ \Delta_4 G_{t,x^{N_R}} \!\left( x^{N_{C2}} \right) \right] \\
  & = x_i E \biggl[ \int \int \int \Delta G_{t,x^{N_R}}^{k,i \neq k,+i,+i} \!\left( x^{N_{C2}}, I_2^{(k)}(t) + z + R_{N_t^{'(k)}}^{(k)}, X_0^{'(k)}(t) \right. \\
  & \qquad \quad \left. + S_{N_t^{'(k)}}^{(k)}, I_2^{(i)}(t) + R_{N_t^{(i)}}^{(i)}, X_0^{'(i)}(t) + S_{N_t^{(i)}}^{(i)}; \int_0^t \nu_s ds, \nu_t, \int_0^t \nu_s' ds, \nu_t' \right) \\
  & \qquad \qquad \times \1_{\{\nu_t>0\}} \1_{\{\nu_t'>0\}} d\NN_0(\nu) d\NN_0(\nu') dQ_h(z) \biggr].
\end{align*}

Now let for $k=i$
\[ \hat{G}_n(z) \equiv E \!\left[ G_{t,x^{N_R}}^k \!\left( x^{N_{C2}}, I_2^{(k)}(t) + z + R_n^{(k)}, X_0^{'(k)}(t) + S_n^{(k)} \right) \right], \]
respectively for $k \neq i$,
\begin{align*}
  \hat{G}_n \!\left( z,N_t^{'(k)} \right) & \equiv E \Bigl[ G_{t,x^{N_R}}^{k,i} \Bigl( x^{N_{C2}}, I_2^{(k)}(t) + z + R_{N_t^{'(k)}}^{(k)}, X_0^{'(k)}(t) + S_{N_t^{'(k)}}^{(k)}, \\
  & \qquad \qquad \qquad \quad I_2^{(i)}(t) + R_n^{(i)}, X_0^{'(i)}(t) + S_n^{(i)} \Bigr) \Bigr].
\end{align*}
Note that the expectation in the definition of $\hat{G}_n \!\left( z,N_t^{'(k)} \right)$ excludes the random variable $N_t^{'(k)}$. Use $w^{'(k)} = \frac{x_k}{2\gamma_k^0 t} + \frac{h_k}{\gamma_k^0 t}$ (i.e. $\rho=1/2$) to obtain for $k=i$
\begin{align}
  & x_k E_{x^{N_{C2}} + h_k e_k}^{N_{C2}} \!\left[ \Delta_4 G_{t,x^{N_R}} \!\left( x^{N_{C2}} \right) \right] \label{equ:term4_k_equ_i} \\
  & \stackrel{(*)}{=} c \frac{x_k}{t^2} \sum \limits_{n=0}^{\infty} p_n \!\left( w^{'(k)} \right) \int \left( \hat{G}_{n+2} -2 \hat{G}_{n+1} + \hat{G}_n \right)\!(z) dQ_h(z), \nonumber
\end{align}
and use $w^{(i)} = \frac{x_i}{2\gamma_i^0 t}$ to obtain for $k \neq i$
\begin{align}
  & x_i E_{x^{N_{C2}} + h_k e_k}^{N_{C2}} \!\left[ \Delta_4 G_{t,x^{N_R}} \!\left( x^{N_{C2}} \right) \right] \label{equ:term4_k_neq_i} \\
  & \stackrel{(*)}{=} c \frac{x_i}{t^2} \sum \limits_{n=0}^{\infty} p_n \!\left( w^{(i)} \right) E \biggl[ \int \left( \hat{G}_{n+2} -2 \hat{G}_{n+1} + \hat{G}_n \right)\!\left( z,N_t^{'(k)} \right) dQ_h(z) \biggr]. \nonumber
\end{align}

A similar argument holds for $x_i E_{x^{N_{C2}}}^{N_{C2}} \!\left[ \Delta_4 G_{t,x^{N_R}} \!\left( x^{N_{C2}} \right) \right]$. Indeed, if $k=i$ replace $z$ by $0$ and replace $w^{'(k)}$ by $w^{(k)} = \frac{x_k}{2\gamma_k^0 t}$ in (\ref{equ:term4_k_equ_i}). If $k \neq i$ replace $z$ by $0$ and replace $N_t^{'(k)}$ by $N_t^{(k)}$ in (\ref{equ:term4_k_neq_i}).

Let us first investigate the case $k=i$. Define
\[ \hat{H}_n(z) = \hat{G}_n(z) - \hat{G}_n(0) \]
to get for $E$ as in (\ref{equ:error_in_case_4}),
\begin{align*}
  E \leq& \; c \frac{x_k}{t^2} \left| \sum \limits_{n=0}^{\infty} p_n \!\left( w^{'(k)} \right) \int \left( \hat{H}_{n+2} -2 \hat{H}_{n+1} + \hat{H}_n \right)\!(z) dQ_h(z) \right| \\
  & \; + \, c \frac{x_k}{t^2} \left| \sum \limits_{n=0}^{\infty} \left( p_n \!\left( w^{'(k)} \right) - p_n \!\left( w^{(k)} \right) \right) \left( \hat{G}_{n+2} -2 \hat{G}_{n+1} + \hat{G}_n \right)\!(0) \right| \\
  \equiv& \; E_1 + E_2.
\end{align*}

We can bound $E_1$ by
\[ c \frac{x_k}{t^2} \sum \limits_{n=0}^{\infty} \left| \left( p_{n-2} -2 p_{n-1} + p_n \right)\!\left( w^{'(k)} \right) \right| \sup \limits_{n \geq 0} \!\left| \int \hat{H}_n(z) dQ_h(z) \right|, \] 
where $p_n(w) \equiv 0$ if $n<0$. By using $q_n(w)=w p_n(w)$ and $\sum_{n=0}^{\infty} |(q_{n-2} -2 q_{n-1} + q_n)(w)| \leq 2$ (see \cite{r6}, (109)) we obtain
\[ E_1 \leq c \frac{x_k}{t^2} \frac{1}{w^{'(k)}} \sup \limits_{n \geq 0} \!\left| \int \hat{H}_n(z) dQ_h(z) \right|. \]
Next observe that $\hat{H}_n(z)$ is zero for $k \in N_2$ (recall that for $k \in N_2$ the indicated change from $\int_0^t x_s^{(k)} ds$ into $I_2^{(k)}(t) + z + R_n^{(k)}$ resp. $I_2^{(k)}(t) + R_n^{(k)}$ has no impact on the terms under consideration) and is a first order difference for $k \in N_C$ for which we obtain as usual
\begin{align*}
  \left| \int \hat{H}_n(z) dQ_h(z) \right| &\leq c \norm{f} t^{-1} (t+x_k)^{-1} \int z dQ_h(z) \\
  &\leq c \norm{f} t^{-1} (t+x_k)^{-1} h_k t \\
  &\leq c \norm{f} h_k t^{-1/2} (t+x_k)^{-1/2}.
\end{align*}
Together with $w^{(k)} = \frac{x_k}{2\gamma_k^0 t}$ and $w^{'(k)} = \frac{x_k}{2\gamma_k^0 t} + \frac{h_k}{\gamma_k^0 t}$ this gives
\[ E_1 \leq c \norm{f} h_k t^{-3/2} (t+x_k)^{-1/2}. \]

For $E_2$ we obtain with $\norm{G} \leq \norm{f}$ and Fubini's theorem
\begin{align}
  E_2 \leq& \; c \norm{f} \frac{x_k}{t^2} \sum \limits_{n=0}^{\infty} \Bigl| \left( p_{n-2} -2 p_{n-1} + p_n \right)\!\left( w^{'(k)} \right) \label{equ:integrand_du} \\
  & \; - \, \left( p_{n-2} -2 p_{n-1} + p_n \right)\!\left( w^{(k)} \right) \Bigr| \nonumber \\
  \leq& \; c \norm{f} \frac{x_k}{t^2} \int_{w^{(k)}}^{w^{'(k)}} \sum \limits_{n=0}^{\infty} \left| \left( p_{n-2}' -2 p_{n-1}' + p_n' \right)\!(u) \right| du. \nonumber
\end{align}
As $p_n(u) = e^{-u} \frac{u^n}{n!}$ we have $p_n'(u) = -p_n(u) + p_{n-1}(u)$ and thus we obtain in case $0<u<1$ for the integrand
\[ \sum \limits_{n=0}^{\infty} \left| \left( p_{n-2}' -2 p_{n-1}' + p_n' \right)\!(u) \right| \leq 8. \]
For $u \geq 1$ we obtain for the integrand as an upper bound
\begin{align*}
  & p_0(u) + p_1(u) \left| 3\frac{1}{u} -1 \right| + \sum \limits_{n=2}^{\infty} p_n(u) \left| \frac{n(n-1)(n-2)}{u^3} - 3 \frac{n(n-1)}{u^2} + 3 \frac{n}{u} - 1 \right| \\
  & \leq e^{-u} (1 +3 +u) + \frac{1}{u^3} \sum \limits_{n=2}^{\infty} p_n(u) \left| ( n-u )^3 - 3 n ( n-u ) + 2 n \right| \\
 & \leq e^{-u} (4+u) + \frac{1}{u^3} \left( E\!\left| N_u-u \right|^3 + 3 \sqrt{ E N_u^2 E \!\left( N_u-u \right)^2 } + 2 E N_u \right),
\end{align*}
where $N_u$ is Poisson with mean $u$. Note that $E|N_u-u|^m \leq c_m u^{m/2}$ for $m \in \NN$ and $u \geq 1$. We also have $E N_u = u$ and $E N_u^2 = u^2+u$. This yields as an upper bound for the integrand in (\ref{equ:integrand_du}) for $u \geq 1$
\[ c u^{-\frac{3}{2}} + \frac{1}{u^3} \left( c_3 u^{3/2} + 3 \sqrt{ \left( u^2+u \right) c_2 u^1 } + 2 u \right) \leq c u^{-\frac{3}{2}}. \]
We thus get for $E_2$
\begin{align*}
  E_2 &\leq c \norm{f} \frac{x_k}{t^2} \int_{w^{(k)}}^{w^{'(k)}} \left( u + \frac{1}{2 \gamma_k^0} \right)^{-3/2} du \\
  &\leq c \norm{f} \frac{x_k}{t^2} \left| w^{'(k)} - w^{(k)} \right| \left( w^{(k)} + \frac{1}{2 \gamma_k^0} \right)^{-3/2} \\
  &\leq c \norm{f} \frac{x_k}{t^2} \frac{h_k}{t} \left( \frac{ x_k+t }{2 \gamma_k^0 t} \right)^{-3/2} \\
  &\leq c \norm{f} h_k t^{-3/2} \left( t+x_k \right)^{-1/2}.
\end{align*}

Together with the bound on $E_1$ the assertion now follows for $k=i$.

Next investigate the case $k \neq i$. Define
\begin{align*}
  \hat{H}_n^1\!\left( z,N_t^{'(k)} \right) &= \hat{G}_n\!\left( z,N_t^{'(k)} \right) - \hat{G}_n\!\left( 0,N_t^{'(k)} \right), \\
  \hat{H}_n^2\!\left( N_t^{'(k)},N_t^{(k)} \right) &= \hat{G}_n\!\left( 0,N_t^{'(k)} \right) - \hat{G}_n\!\left( 0,N_t^{(k)} \right)
\end{align*}
to get
\begin{align*}
  E \leq& \; c \frac{x_i}{t^2} \biggl| \sum \limits_{n=0}^{\infty} p_n \!\left( w^{(i)} \right) E \biggl[ \int \left( \hat{H}_{n+2}^1 -2 \hat{H}_{n+1}^1 + \hat{H}_n^1 \right)\!\left( z,N_t^{'(k)} \right) dQ_h(z) \biggr] \biggr| \\
  & \!\! + \, c \frac{x_i}{t^2} \biggl| \sum \limits_{n=0}^{\infty} p_n \!\left( w^{(i)} \right) E \biggl[ \int \left( \hat{H}_{n+2}^2 -2 \hat{H}_{n+1}^2 + \hat{H}_n^2 \right)\!\left( N_t^{'(k)},N_t^{(k)} \right) dQ_h(z) \biggr] \biggr|.
\end{align*}
Recall that the expectation in the definition of $\hat{G}_n \!\left( z,N_t^{'(k)} \right)$ and thus of \\ $\hat{H}_n^1\!\left( z,N_t^{'(k)} \right)$ excludes the random variable $N_t^{'(k)}$. To bound $E$ we thus take expectation w.r.t. $N_t^{'(k)}$, too. Rewriting this yields
\begin{align*}
  E \leq& \; c \frac{x_i}{t^2} \sum \limits_{n=0}^{\infty} \left| \left( p_{n-2} -2 p_{n-1} + p_n \right)\!\left( w^{(i)} \right) \right| \\
  & \!\!\! \times \, \sup \limits_{n \geq 0} \biggl\{ E \!\left[ \left| \int \!\hat{H}_n^1\!\left( z,N_t^{'(k)} \right) dQ_h(z) \right| \right] + E \!\left[ \left|\int\!\hat{H}_n^2\!\left( N_t^{'(k)},N_t^{(k)} \right) dQ_h(z) \right| \right] \biggr\}
\end{align*} 
and by using $q_n(w)=w p_n(w)$ and $\sum_{n=0}^{\infty} |(q_{n-2} -2 q_{n-1} + q_n)(w)| \leq 2$ again we obtain 
\begin{align*}
  E &\leq c \frac{x_i}{t^2} \frac{1}{w^{(i)}} \sup \limits_{n \geq 0} \biggl\{ E \!\left[ \left| \int \hat{H}_n^1\!\left( z,N_t^{'(k)} \right) dQ_h(z) \right| \right] \\
  & \qquad \qquad \qquad \quad + E \!\left[ \left|\int \hat{H}_n^2\!\left( N_t^{'(k)},N_t^{(k)} \right) dQ_h(z) \right| \right] \biggr\}.
\end{align*} 
Next observe that $\hat{H}_n^1 \!\left( z,N_t^{'(k)} \right)$ is zero for $k \in N_2$ and is a first order difference for $k \in N_C$ for which we obtain
\begin{align*}
  \left| \int \hat{H}_n^1\!\left( z,N_t^{'(k)} \right) dQ_h(z) \right| &\leq c \norm{f} t^{-1} (t+x_k)^{-1} \int z dQ_h(z) \\
  &\leq c \norm{f} h_k t^{-1/2} (t+x_k)^{-1/2}.
\end{align*}
The other term can be bounded as follows:
\begin{align*}
  \left| \hat{H}_n^2\!\left( N_t^{'(k)},N_t^{(k)} \right) \right| & \leq \sum_{N=0}^{\infty} \left| p_N\!\left( w^{'(k)} \right) - p_N\!\left( w^{(k)} \right) \right| \\
  & \qquad \times \Bigl| E \Bigl[ G_{t,x^{N_R}}^{k,i} \Bigl( x^{N_{C2}}, I_2^{(k)}(t) + R_N^{(k)}, X_0^{'(k)}(t) + S_N^{(k)}, \\
  & \qquad \qquad \qquad \qquad \qquad \quad I_2^{(i)}(t) + R_n^{(i)}, X_0^{'(i)}(t) + S_n^{(i)} \Bigr) \Bigr] \Bigr| \\
  & \leq \sum_{N=0}^{\infty} \left| p_N\!\left( w^{'(k)} \right) - p_N\!\left( w^{(k)} \right) \right| \norm{G},
\end{align*}
where $w^{(k)} = \frac{x_k}{2\gamma_k^0 t}$ and $w^{'(k)} = \frac{x_k}{2\gamma_k^0 t} + \frac{h_k}{\gamma_k^0 t}$. As done before in the proof of Proposition \ref{equ:proposition_22} we use
\[ \sum_{N=0}^{\infty} \left| \int_{w^{(k)}}^{w^{'(k)}} p_N'(u) du \right| \leq c t^{-1/2} h_k (t+x_k)^{-1/2} \]
to finally get with $\norm{G} \leq \norm{f}$
\[ \left| \int \hat{H}_n^2\!\left( N_t^{'(k)},N_t^{(k)} \right) dQ_h(z) \right| \leq c t^{-1/2} h_k (t+x_k)^{-1/2} \norm{f}. \]

Plugging our results into our estimate for $E$ we get
\[ E \leq c \frac{x_i}{t^2} \frac{t}{x_i} c \norm{f} h_k t^{-1/2} (t+x_k)^{-1/2} \leq c \norm{f} h_k t^{-3/2} (t+x_k)^{-1/2}, \]
which proves our assertion.
\end{proof}
%
% ------------------------------------------------------------

Finally we consider the increments in $x_k$, $k \in N_R$.
%
% ------------------------------------------------------------
%
\begin{lem} \label{equ:lemma_27}
If $f$ is a bounded Borel function on $\SSS_0$, then for all $x, h \in \SSS_0$, $i \in N_{C2}$ and $k \in N_R$
\[ \left| x_i \dfdx{P_tf}{i} (x+h_k e_k) - x_i \dfdx{P_tf}{i} (x) \right| \leq \frac{c \norm{f}}{t^{3/2}} |h_k| \min \limits_{l \in C_k} \!\left\{ (t+x_l)^{-1/2} \right\}. \]
\end{lem}
%
% ------------------------------------------------------------
%
\begin{proof}
Except for the necessary adaptations, already used in the proofs of the preceding assertions, the proof proceeds analogously to Lemma 27 in \cite{r6}. 
%
% additional part starts

Let
\[ \Delta_n^k G_{t,x^{N_R}} (X) \equiv \Delta_n G_{t,x^{N_R}+h_k e_k} (X) - \Delta_n G_{t,x^{N_R}} (X) \]
so that by (\ref{equ:ninetysix}),
\begin{equation} \label{equ:onehundredsixteen}
  x_i \dfdx{P_tf}{i} (x+h_k e_k) - x_i \dfdx{P_tf}{i} (x) = \sum \limits_{n=1}^4 x_i E^{N_{C2}} \left[ \Delta_n^k G_{t,x^{N_R}} (x^{N_{C2}}) \right]. 
\end{equation}

Considering the term $n=1$ in the above sum first, write
\begin{align*}
  & x_i E^{N_{C2}} \left[ \Delta_1^k G_{t,x^{N_R}} (x^{N_{C2}}) \right] \\
  & = x_i E^{N_{C2}} \left[ \int \int \1_{\{\nu_t=\nu_t'=0\}} \left( \right. \right. \\
  & \qquad G_{t,x^{N_R}+h_k e_k}^{+i,+i} \left( x^{N_{C2}}; \int_0^t \nu_s ds, 0, \int_0^t \nu_s' ds, 0 \right) \\
  & \qquad - G_{t,x^{N_R}+h_k e_k}^{+i,+i} \left( x^{N_{C2}}; 0, 0, \int_0^t \nu_s' ds, 0 \right) \\
  & \qquad - G_{t,x^{N_R}+h_k e_k}^{+i,+i} \left( x^{N_{C2}}; \int_0^t \nu_s ds, 0, 0, 0 \right) \\
  & \qquad + G_{t,x^{N_R}+h_k e_k}^{+i,+i} \left( x^{N_{C2}}; 0, 0, 0, 0 \right) \\
  & \qquad - G_{t,x^{N_R}}^{+i,+i} \left( x^{N_{C2}}; \int_0^t \nu_s ds, 0, \int_0^t \nu_s' ds, 0 \right) \\
  & \qquad + G_{t,x^{N_R}}^{+i,+i} \left( x^{N_{C2}}; 0, 0, \int_0^t \nu_s' ds, 0 \right) \\
  & \qquad + G_{t,x^{N_R}}^{+i,+i} \left( x^{N_{C2}}; \int_0^t \nu_s ds, 0, 0, 0 \right) \\
  & \qquad \left. \left. - G_{t,x^{N_R}}^{+i,+i} \left( x^{N_{C2}}; 0, 0, 0, 0 \right) \right) d\NN_0(\nu) d\NN_0(\nu') \right].
\end{align*}
The expression in square brackets on the right hand side is a third order difference (first order in $x_k$ and second order in $y_j, i \in C_j$ where $j \in N_R$) to which we may apply Lemma \ref{equ:lemma_13} (resp. an extension similar to Lemma \ref{equ:lemma_13} with $\Delta_1 \in \RR$ non-zero), the mean value theorem and the bound from Lemma \ref{equ:lemma_11}(b), and conclude
\begin{align}
  & \left| x_i E^{N_{C2}} \left[ \Delta_1^k G_{t,x^{N_R}} (x^{N_{C2}}) \right] \right| \label{equ:onehundredseventeen} \\
  & \leq c x_i \norm{f} E^{N_{C2}} \left[ \left( I_t^{(k)} \right)^{-1/2} \sum_{j_1: j_1 \in \bar{R}_i} \sum_{j_2: j_2 \in \bar{R}_i} \left( I_t^{(j_1)} \right)^{-1} \left( I_t^{(j_2)} \right)^{-1} \right] \nonumber \\
  & \qquad \times |h_k| \left( \int \int_0^t \nu_s ds d\NN_0(\nu) \right)^2 \nonumber \\
  & \leq c \norm{f} x_i t^{-5/2} \min \limits_{l \in C_k} \{ (t+x_l)^{-1/2} \} (t+x_i)^{-2} |h_k| t^2 \nonumber \\
  & \leq c \norm{f} |h_k| t^{-3/2} \min \limits_{l \in C_k} \{ (t+x_l)^{-1/2} \}. \nonumber
\end{align}
Considering the term $n=2$ in the above sum next, observe that
\begin{align*}
  & \left| x_i E^{N_{C2}} \left[ \Delta_2^k G_{t,x^{N_R}} (x^{N_{C2}}) \right] \right| \\
  & \leq \left| x_i E^{N_{C2}} \left[ \int \int \1_{\{\nu_t>0,\nu_t'=0\}} \left( \right. \right. \right. \\
  & \qquad G_{t,x^{N_R}+h_k e_k}^{+i,+i} \left( x^{N_{C2}}; \int_0^t \nu_s ds, \nu_t, \int_0^t \nu_s' ds, 0 \right) \\
  & \qquad - G_{t,x^{N_R}+h_k e_k}^{+i,+i} \left( x^{N_{C2}}; \int_0^t \nu_s ds, \nu_t, 0, 0 \right) \\
  & \qquad - G_{t,x^{N_R}}^{+i,+i} \left( x^{N_{C2}}; \int_0^t \nu_s ds, \nu_t, \int_0^t \nu_s' ds, 0 \right) \\
  & \qquad \left. \left. \left. + G_{t,x^{N_R}}^{+i,+i} \left( x^{N_{C2}}; \int_0^t \nu_s ds, \nu_t, 0, 0 \right) \right) d\NN_0(\nu) d\NN_0(\nu') \right] \right| \\
  & \quad + \left| x_i E^{N_{C2}} \left[ \int \int \1_{\{\nu_t>0,\nu_t'=0\}} \left( \right. \right. \right. \\
  & \qquad - G_{t,x^{N_R}+h_k e_k}^{+i,+i} \left( x^{N_{C2}}; 0, 0, \int_0^t \nu_s' ds, 0 \right) \\
  & \qquad + G_{t,x^{N_R}+h_k e_k}^{+i,+i} \left( x^{N_{C2}}; 0, 0, 0, 0 \right) \\
  & \qquad + G_{t,x^{N_R}}^{+i,+i} \left( x^{N_{C2}}; 0, 0, \int_0^t \nu_s' ds, 0 \right) \\
  & \qquad \left. \left. \left. - G_{t,x^{N_R}}^{+i,+i} \left( x^{N_{C2}}; 0, 0, 0, 0 \right) \right) d\NN_0(\nu) d\NN_0(\nu') \right] \right|.
\end{align*}
Both expressions on the right hand side are second order differences (first order in $x_k$ and first order in $y_j, i \in C_j$ where $j \in N_R$) to which we may apply Lemma \ref{equ:lemma_13}, the mean value theorem and the bound from Lemma \ref{equ:lemma_11}(b), and conclude
\begin{align}
  & \left| x_i E^{N_{C2}} \left[ \Delta_2^k G_{t,x^{N_R}} (x^{N_{C2}}) \right] \right| \label{equ:onehundrednineteen} \\
  & \leq c x_i \norm{f} E^{N_{C2}} \left[ \left( I_t^{(k)} \right)^{-1/2} \sum_{j_1: j_1 \in \bar{R}_i} \left( I_t^{(j_1)} \right)^{-1} \right] |h_k| \nonumber \\
  & \qquad \times \, \int \int_0^t \nu_s' ds d\NN_0(\nu') \NN_0(\nu_t>0) \nonumber \\
  & \leq c \norm{f} x_i t^{-3/2} \min \limits_{l \in C_k} \{ (t+x_l)^{-1/2} \} (t+x_i)^{-1} |h_k| t t^{-1} \nonumber \\
  & \leq c \norm{f} |h_k| t^{-3/2} \min \limits_{l \in C_k} \{ (t+x_l)^{-1/2} \}. \nonumber
\end{align}
By symmetry the same bound holds for the $n=3$ term.

To bound the $n=4$ term in (\ref{equ:onehundredsixteen}), we use the notation and setting introduced in Lemma \ref{equ:lemma_10} (with $\rho=1/2$). For $h_k$ fixed, let
\[ DG_{t,x^{N_R}}(y) = G_{t,x^{N_R}+h_k e_k}(y) - G_{t,x^{N_R}}(y) \]
and use the evident changes to former notations. The Mean Value Theorem and the bounds in Lemma \ref{equ:lemma_11}(a) imply
\begin{equation} \label{equ:onehundredtwenty}
  \left| DG_{t,x^{N_R}}(y) \right| \leq c \norm{f} \left( y_k \right)^{-1/2} |h_k|.
\end{equation}
We obtain
\begin{align*}
  & x_i E^{N_{C2}} \left[ \Delta_4^k G_{t,x^{N_R}} (x^{N_{C2}}) \right] \\
  & = x_i E^{N_{C2}} \left[ \int \int \1_{\{\nu_t>0,\nu_t'>0\}} \left( \right. \right. \\
  & \qquad DG_{t,x^{N_R}}^{+i,+i} \left( x^{N_{C2}}; \int_0^t \nu_s ds, \nu_t, \int_0^t \nu_s' ds, \nu_t' \right) \\
  & \qquad - DG_{t,x^{N_R}}^{+i,+i} \left( x^{N_{C2}}; 0, 0, \int_0^t \nu_s' ds, \nu_t' \right) \\
  & \qquad - DG_{t,x^{N_R}}^{+i,+i} \left( x^{N_{C2}}; \int_0^t \nu_s ds, \nu_t, 0, 0 \right) \\
  & \qquad \left. \left. + DG_{t,x^{N_R}}^{+i,+i} \left( x^{N_{C2}}; 0, 0, 0, 0 \right)  \right) d\NN_0(\nu) d\NN_0(\nu') \right].
\end{align*}
Take $w = \frac{x_i}{2\gamma_i^0 t}$ and recall that $q_n(w)=w p_n(w)$ with $\sum \limits_{n=0}^{\infty} |q_{n-2}(w) -2 q_{n-1}(w) + q_n(w)| \leq 2$, where $q_n=0$ if $n<0$. Then
\begin{align*}
  & x_i E^{N_{C2}} \left[ \Delta_4^k G_{t,x^{N_R}} (x^{N_{C2}}) \right] \\
  & \stackrel{(*)}{=} x_i \frac{c}{t^2} E \left[ DG_{t,x^{N_R}}^i \left( x^{N_{C2}}; I_2(t) + R_{N_t+2}(t), X_0'(t) + S_{N_t+2}(t) \right) \right. \\
  & \qquad \qquad \quad - 2 DG_{t,x^{N_R}}^i \left( x^{N_{C2}}; I_2(t) + R_{N_t+1}(t), X_0'(t) + S_{N_t+1}(t) \right) \\
  & \qquad \qquad \quad \left. + DG_{t,x^{N_R}}^i \left( x^{N_{C2}}; I_2(t) + R_{N_t}(t), X_0'(t) + S_{N_t}(t) \right) \right].
\end{align*}
Using (\ref{equ:onehundredtwenty}) and Lemma \ref{equ:lemma_7}(b) this can be bounded as follows
\begin{align}
  & \left| x_i E^{N_{C2}} \left[ \Delta_4^k G_{t,x^{N_R}} (x^{N_{C2}}) \right] \right| \label{equ:onehundredtwentyone} \\
  & = x_i \frac{c}{t^2} \frac{1}{w} \sum \limits_{n=0}^{\infty} \left| q_{n-2} (w) -2 q_{n-1} (w) + q_n (w) \right| \nonumber \\
  & \qquad \times \left| E \left[ DG_{t,x^{N_R}}^i \left( x^{N_{C2}}; I_2(t) + R_n(t), X_0'(t) + S_n(t) \right) \right] \right| \nonumber \\
  & \stackrel{k \in N_R}{\leq} x_i \frac{c}{t^2} \frac{t}{x_i} \norm{f} t^{-1/2} \min \limits_{l \in C_k} \{ (t+x_l)^{-1/2} \} |h_k| \nonumber \\
  & \leq c \norm{f} |h_k| t^{-3/2} \min \limits_{l \in C_k} \{ (t+x_l)^{-1/2} \}. \nonumber
\end{align}
Putting (\ref{equ:onehundredseventeen}), (\ref{equ:onehundrednineteen}) and (\ref{equ:onehundredtwentyone}) into (\ref{equ:onehundredsixteen}), we complete the proof of Lemma \ref{equ:lemma_27}.
%
% additional part ends
%
\end{proof}
% 
% ------------------------------------------------------------

{\it Continuation of the proof of Proposition \ref{equ:proposition_23}.} Use Lemmas \ref{equ:lemma_24}, \ref{equ:lemma_25} and \ref{equ:lemma_26} in (\ref{equ:ninetyseven}) together with the calculation following (\ref{equ:ninetyseven}) to obtain the bound for increments in $x_k, k \in N_{C2}$. Lemma \ref{equ:lemma_27} gives the corresponding bound for increments in $x_k, k \in N_R$ which completes the proof of (\ref{equ:ninetythree}). 
\end{proof} 
%
% ============================================================
%
\section{PROOF OF UNIQUENESS} \label{equ:section_3}
%
% ============================================================

As in Section 3, \cite{r6}, it is relatively straightforward to use the results from the previous sections on the semigroup $P_t$ to prove bounds on the resolvent $R_{\lambda}$ of $P_t$. 

 We shall then use these bounds to complete the proof of uniqueness of solutions to the martingale problem {\it MP}($\SA$,$\nu$) satisfying Hypothesis \ref{equ:hypothesis_1} and \ref{equ:hypothesis_2}, where $\nu$ is a probability on 
\[ \SSS = \left\{ x \in \RR_+^d: \prod_{j \in R} \left( \sum_{i \in C_j} x_i + x_j \right) > 0 \right\} \]
(recall (\ref{equ:state_space}) and Lemma \ref{equ:lemma_5}) and
\begin{equation} \label{equ:recall_def_SA}
  \SA f(x) = \sum_{j \in R} \gamma_j(x) \left( \sum_{i \in C_j} x_i \right) x_j f_{jj}(x) + \sum_{j \notin R} \gamma_j(x) x_j f_{jj}(x) + \sum_{j \in V} b_j(x) f_j(x). 
\end{equation}
The proof of uniqueness is identical to the one in \cite{r6} except for minor changes such as the replacement of $x_{c_j}$ by $\sum_{i \in C_j} x_i$ at the appropriate places. Note in particular the change in the definition of the state space $\SSS$.

In what follows we shall give a sketch of the proofs and indicate where statements have to be modified. For explicit calculations the reader is referred to \cite{r6}, Sections 3 and 4.
%
% ------------------------------------------------------------
%
\begin{nota} 
For $i \in N_{C2}$ let
\begin{equation} \label{equ:defi_sum_bar}
  \bar{y}_i = \left( \left\{ y_j \right\}_{j \in \bar{R}_i}, y_i \right), \bar{y}_i \bar{e}_i = \sum_{j \in \bar{R}_i} y_j e_j + y_i e_i \mbox{ and } \bar{\RR}_i = \RR^{|\bar{R}_i|} \times \RR_+, 
\end{equation}
where we understand this to be $\bar{y}_i = \left( y_i \right)$ in case $i \in N_2$, i.e. $\bar{R}_i = \emptyset$. For $f \in \SC_b^2(\SSS_0)$ let
\begin{equation} \label{equ:onehundredtwentyfive}
  \frac{\partial f}{\partial \bar{x}_i} = \left( \left\{ \dx{j}f \right\}_{j \in \bar{R}_i}, \dx{i}f \right), \left| \frac{\partial f}{\partial \bar{x}_i} \right| = \sum \limits_{j \in \bar{R}_i} \left| \dx{j}f \right| + \left| \dx{i}f\right| 
\end{equation}
and
\begin{equation} \label{equ:onehundredtwentysix}
  \bignorm{ \frac{\partial f}{\partial \bar{x}_i} } = \sup \!\left\{ \left| \frac{\partial f}{\partial \bar{x}_i} (x) \right| : x \in \SSS_0 \right\}, 
\end{equation}
where $\SSS_0 = \{ x \in \RR^d: x_i \geq 0 \mbox{ for all } i \in N_{C2} \}$ as defined in (\ref{equ:ten}). Also introduce
\[ \Delta_i f = \left( \left\{ x_i \ddx{j}f \right\}_{j \in \bar{R}_i}, x_i \ddx{i}f \right). \]
Define $|\Delta_i f|$ and $\norm{\Delta_i f}$ similarly to (\ref{equ:onehundredtwentyfive}) and (\ref{equ:onehundredtwentysix}). 
\end{nota}
%
% ------------------------------------------------------------
%
With the help of these notations $\SA^0$ (see (\ref{equ:seven})) can be rewritten to
\begin{align}
  \SA^0 f(x) &= \sum_{j \in V} b_j^0 f_j(x) + \sum_{j \in N_R} \gamma_j^0 \left( \sum_{i \in C_j} x_i \right) f_{jj}(x) + \sum_{i \in N_{C2}} \gamma_i^0 x_i f_{ii}(x) \label{equ:A0_new} \\
  &= \sum_{i \in N_{C2}} \Big\langle \underline{b^0}_{\; i},\frac{\partial f}{\partial \bar{x}_i} (x) \Big\rangle + \sum_{i \in N_{C2}} \Big\langle \underline{\gamma^0}_{\; i},\Delta_i f (x) \Big\rangle, \nonumber
\end{align}
where $\langle \cdot,\cdot \rangle$ denotes the standard scalar product in $\RR^k, k \in \NN$. To prevent overcounting in case $\bar{R}_{i_1} \cap \bar{R}_{i_2} \neq \emptyset$ for $i_1 \neq i_2$, $i_1, i_2 \in N_C$ (see also definition (\ref{equ:defi_sum_bar})) the vector $\overline{b^0}_i$ was replaced by $\underline{b^0}_{\; i}$ in the above formula, where $\underline{b^0}_{\; i}$ has certain coordinates set to zero so that the above equality holds. The same applies to the vector $\underline{\gamma^0}_{\; i}$. The details are left to the interested reader.
%
% ------------------------------------------------------------
%
\begin{thm} \label{equ:thm_34}
There is a constant $c$ such that for all $f \in \SC_w^{\alpha}(\SSS_0), \lambda \geq 1$ and $k, i \in N_{C2}$,
\begin{align*}
  (a) \quad & \bignorm{ \frac{\partial R_{\lambda} f}{\partial \bar{x}_k} } + \bignorm{ \Delta_k R_{\lambda} f } \leq c \lambda^{-\alpha/2} \wabsv{f}. \\
  (b) \quad & \left| \frac{\partial R_{\lambda} f}{\partial \bar{x}_k} \right|_{\SC_w^{\alpha}} + \left| \Delta_k R_{\lambda} f \right|_{\SC_w^{\alpha}} \leq c \wabsv{f}.
\end{align*}
\end{thm}
%
% ------------------------------------------------------------
%
\begin{note}
This result is slightly weaker than the corresponding Theorem 34 in \cite{r6} as $\ainorm{f}{k}$ is replaced by $\wabsv{f}$ in (a). 
\end{note}
%
% ------------------------------------------------------------
%
\begin{proof}
Firstly we obtain a result similar to Proposition 30 in \cite{r6}. This is an easy consequence of Proposition \ref{equ:proposition_16} and Proposition \ref{equ:proposition_17}, using the equivalence of norms shown in Theorem \ref{equ:theorem_19} and states that there is a constant $c$ such that 

(a) For all $f \in \SC_w^{\alpha}(\SSS_0), t>0, x \in \SSS_0$, and $i \in N_{C2}$,
\begin{equation} \label{equ:prop30_a1}
  \left| \frac{\partial P_tf}{\partial \bar{x}_i} (x) \right| \leq c \wabsv{f} t^{\alpha/2-1/2} (t+x_i)^{-1/2} \leq c \wabsv{f} t^{\alpha/2-1}, 
\end{equation}
and
\begin{equation} \label{equ:prop30_a2}
  \norm{ \Delta_i P_tf } \leq c \wabsv{f} t^{\alpha/2-1}. 
\end{equation}

(b) For all $f$ bounded and Borel on $\SSS_0$ and all $i \in N_{C2}$,
\[ \bignorm{ \frac{\partial P_tf}{\partial \bar{x}_i} } \leq c \norm{f} t^{-1}. \]

Note in particular that Theorem \ref{equ:theorem_19} gave $\SC_w^{\alpha} = \SSS^{\alpha}$ and that every function in $\SC_w^{\alpha}(\SSS_0)$ is by definition bounded. 

Secondly, an easy consequence of Propositions \ref{equ:proposition_22}, \ref{equ:proposition_23} and the triangle inequality, using the equivalence of norms shown in Theorem \ref{equ:theorem_19} and the equivalence of the maximum norm and Euclidean norm of finite dimensional vectors, is a result similar to Proposition 32, \cite{r6}: There is a constant $c$ such that for all $f \in \SC_w^{\alpha}(\SSS_0), i, k \in N_{C2}$ and $\bar{h}_i \in \bar{\RR}_i$, 

(a) 
\begin{equation} \label{equ:onehundredthirty}
  \left| \frac{\partial P_t f}{\partial \bar{x}_k} \!\left( x + \bar{h}_i \bar{e}_i \right) - \frac{\partial P_t f}{\partial \bar{x}_k} \!\left( x \right) \right| \leq c \wabsv{f} t^{-3/2+\alpha/2} (t+x_i)^{-1/2} \left| \bar{h}_i \right|, 
\end{equation}

(b) 
\begin{equation} \label{equ:onehundredthirtyone}
  \left| \Delta_k (P_t f) \!\left( x + \bar{h}_i \bar{e}_i \right) - \Delta_k (P_t f) \!\left( x \right) \right| \leq c \wabsv{f} t^{-3/2+\alpha/2} (t+x_i)^{-1/2} \left| \bar{h}_i \right|. 
\end{equation}

Finally recall that $R_{\lambda} f (x) = \int_0^{\infty} e^{-\lambda t} P_t f (x) dt$ is the resolvent associated with $P_t$. Now the remainder of the proof works as in the proof of Theorem 34 in \cite{r6}: Part (a) of Theorem \ref{equ:thm_34} is obtained by integrating (\ref{equ:prop30_a1}) resp. (\ref{equ:prop30_a2}) over time. 
%
% additional part starts
%
For part (b) let $\tilde{t}>0$ and $\bar{h}_i \in \bar{\RR}_i$, then
\begin{align}
  & \left| \Delta_k \left( R_{\lambda} f \right) \left( x+\bar{h}_i \bar{e}_i \right) - \Delta_k \left( R_{\lambda} f \right) \left( x \right) \right| \label{equ:onehundredforty} \\
  & \leq \int_0^{\tilde{t}} \left| \Delta_k \left( P_t f \right) \left( x+\bar{h}_i \bar{e}_i \right) \right| + \left| \Delta_k \left( P_t f \right) \left( x \right) \right| dt \nonumber \\
  & \qquad + \left| \int_{\tilde{t}}^{\infty} e^{-\lambda t} \left[ \Delta_k \left( P_t f \right) \left( x+\bar{h}_i \bar{e}_i \right) - \Delta_k \left( P_t f \right) \left( x \right) \right] dt \right| \nonumber \\
  & \leq \int_0^{\tilde{t}} c t^{\alpha/2-1} \wabsv{f} dt + \int_{\tilde{t}}^{\infty} c \wabsv{f} t^{\alpha/2-3/2} (t+x_i)^{-1/2} \left| \bar{h}_i \right| dt, \nonumber
\end{align}
where we used (\ref{equ:prop30_a2}) to bound the first term and (\ref{equ:onehundredthirtyone}) to bound the second. The above is at most
\[ c \wabsv{f} \tilde{t}^{\alpha/2} + c \wabsv{f} \tilde{t}^{\alpha/2-1/2} x_i^{-1/2} \left| \bar{h}_i \right|. \]
Set $\tilde{t} = x_i^{-1} \left| \bar{h}_i \right|^2$, to conclude that
\begin{equation} \label{equ:onehundredfortyone}
  \left| \Delta_k \left( R_{\lambda} f \right) \left( x+\bar{h}_i \bar{e}_i \right) - \Delta_k \left( R_{\lambda} f \right) \left( x \right) \right| \leq c \wabsv{f} x_i^{-\alpha/2} \left| \bar{h}_i \right|^{\alpha}.
\end{equation}
Use $(t+x_i)^{-1/2} \leq t^{-1/2}$ in (\ref{equ:onehundredforty}) to conclude that for any $\tilde{t}>0$,
\[ \left| \Delta_k \left( R_{\lambda} f \right) \left( x+\bar{h}_i \bar{e}_i \right) - \Delta_k \left( R_{\lambda} f \right) \left( x \right) \right| \leq c \wabsv{f} \tilde{t}^{\alpha/2} + c \wabsv{f} \tilde{t}^{\alpha/2-1/2} \tilde{t}^{-1/2} \left| \bar{h}_i \right|. \]
Now set $\tilde{t} = \left| \bar{h}_i \right|$ to conclude
\begin{equation} \label{equ:onehundredfortytwo}
  \left| \Delta_k \left( R_{\lambda} f \right) \left( x+\bar{h}_i \bar{e}_i \right) - \Delta_k \left( R_{\lambda} f \right) \left( x \right) \right| \leq c \wabsv{f} \left| \bar{h}_i \right|^{\alpha/2}.
\end{equation}
(\ref{equ:onehundredfortyone}) and (\ref{equ:onehundredfortytwo}) together imply the required bound on $\left| \Delta_k R_{\lambda} f \right|_{\SC_w^{\alpha}}$ in (b). 

The required bound on $\left| \frac{\partial R_{\lambda} f}{\partial \bar{x}_k} \right|_{\SC_w^{\alpha}}$ is proved in the same way.
%
% additional part ends
%
\end{proof}
%
% ------------------------------------------------------------
%
\begin{proof}[Proof of Theorem \ref{equ:theorem_4}]
The existence of a solution to the martingale problem for {\it MP}($\SA,\nu$) follows by standard methods (a result of Skorokhod yields existence of approximating solutions, then use a tightness-argument), e.g. see the proof of Theorem 1.1 in \cite{r1}. Note in particular that Lemma \ref{equ:lemma_5} ensures that solutions remain in $\SSS \subset \RR_+^d$. The uniform boundedness in $M$ of the term $E\!\left[ \sum_i \left| X_T^{M,i} \right| \right]$ that appears in the proof of Theorem 1.1 in \cite{r1} can easily be replaced by the uniform boundedness in $M$ of $E\!\left[ \sum_{i \in V} (X_T^{M,i})^2 \right]$ via a Gronwall-type argument.

At the end of this section we shall reduce the proof of uniqueness to the following theorem. The theorem investigates uniqueness of a perturbation of the operator $\SA^0$ as defined in (\ref{equ:seven}) (also refer to (\ref{equ:A0_new})) with coefficients satisfying (\ref{equ:eight}) and (\ref{equ:nine}). $\SA^0$ is the generator of a unique diffusion on $\SSS\!\left( x^0 \right)$ given by (\ref{equ:ten}) with semigroup $P_t$ and resolvent $R_{\lambda}$ given by (\ref{equ:eleven}). For the definition of $M^0$ refer to (\ref{equ:fifteen}). 

In what follows $x^0 \in \SSS$ will be arbitrarily fixed.
%
% ------------------------------------------------------------
%
\begin{thm} \label{equ:thm_36_37}
 Assume that
\begin{align} 
  \tilde{\SA} f(x) =& \; \sum_{j \in N_R} \tilde{\gamma}_j(x) \left( \sum_{i \in C_j} x_i \right) f_{jj}(x) \label{equ:onehundredfortythree} \\
  & \; + \, \sum_{j \in N_{C2}} \tilde{\gamma}_j(x) x_j f_{jj}(x) + \sum_{j \in V} \tilde{b}_j(x) f_j(x), \ x \in \SSS\!\left( x^0 \right), \nonumber
\end{align}
where $\tilde{b}_k: \SSS\!\left( x^0 \right) \rightarrow \RR$ and $\tilde{\gamma}_k: \SSS\!\left( x^0 \right) \rightarrow (0,\infty)$,
\[ \tilde{\Gamma} = \sum \limits_{k=1}^d \bignormw{ \tilde{\gamma}_k } + \bignormw{ \tilde{b}_k } < \infty. \]
Let
\[ \tilde{\epsilon}_0 = \sum \limits_{k=1}^d \bignorm{ \tilde{\gamma}_k - \gamma_k^0 } + \bignorm{ \tilde{b}_k - b_k^0 }, \]
where $b_k^0, \gamma_k^0, k \in V$ satisfy (\ref{equ:eight}). Let $\SB f = (\tilde{\SA} - \SA^0) f$. 

(a) There exists $\epsilon_1 = \epsilon_1(M^0) > 0$ and $\lambda_1 = \lambda_1(M^0,\tilde{\Gamma}) \geq 0$ such that if $\tilde{\epsilon}_0 \leq \epsilon_1$ and $\lambda \geq \lambda_1$ then $\SB R_{\lambda}: \SC_w^{\alpha} \rightarrow \SC_w^{\alpha}$ is a bounded operator with $\normx{\SB R_{\lambda}} \leq 1/2$. 

(b) If we assume additionally that $\tilde{\gamma}_k$ and $\tilde{b}_k$ are H\"older continuous of index $\alpha \in (0,1)$, constant outside a compact set and $\tilde{b}_k|_{\{x_k=0\}} \geq 0$ for all $k \in V \backslash N_R$, then the martingale problem {\it MP}($\tilde{\SA},\nu$) has a unique solution for each probability $\nu$ on $\SSS\!\left( x^0 \right)$.
\end{thm}
%
% ------------------------------------------------------------
%
\begin{proof}
Let $\tilde{R}_\lambda$ be the associated resolvent operator of the perturbation operator $\tilde{\SA}$. Using the definition $\SB = \tilde{\SA} - \SA^0$ and recalling (\ref{equ:A0_new}) we get for $f \in \SC_w^{\alpha}$ that
\begin{align*}
  \normw{ \SB R_{\lambda} f } \leq& \; \sum_{i \in N_{C2}} \bignormw{ \bigg\langle \underline{\left( \tilde{b}(x) - b^0 \right)}_{\; i},\frac{\partial R_\lambda f}{\partial \bar{x}_i} (x) \bigg\rangle } \\
  & \; + \, \sum_{i \in N_{C2}} \bignormw{ \bigg\langle \underline{\left( \tilde{\gamma}(x) - \gamma^0 \right)}_{\; i},\Delta_i R_\lambda f (x) \bigg\rangle }. 
\end{align*}
Using (\ref{equ:seventysix}) (recall in particular the discussion on the reasons for using two different norms from Remark \ref{equ:rmk_norms}) we obtain for instance for arbitrary $i \in N_C$ and $j \in \bar{R}_i$ 
\begin{align*}
  & \left| \left( \tilde{b}_j(x) - b_j^0 \right) \frac{ \partial R_{\lambda} f }{\partial x_j } (x) \right|_{\SC_w^{\alpha}} \\
  & \leq c \left[ \bignormw{ \tilde{b}_j(x) - b_j^0 } \bignorm{ \frac{ \partial R_{\lambda} f }{\partial x_j } (x) } + \bignorm{ \tilde{b}_j(x) - b_j^0 } \left| \frac{ \partial R_{\lambda} f }{\partial x_j } (x) \right|_{\alpha} \right] \\
  & \leq c \left[ \left( \tilde{\Gamma} + M^0 \right) \lambda^{-\alpha/2} \wabsv{f} + \tilde{\epsilon}_0 \wabsv{f} \right]
\end{align*} 
by Theorem \ref{equ:thm_34}, (\ref{equ:eighty}) and the assumptions of this theorem. By arguing similarly for the other terms we get indeed $\normw{ \SB R_{\lambda} f } \leq \frac{1}{2} \normw{f}$ for $\lambda$ big enough thus finishing the proof of part (a).

For part (b) we proceed as in the proof of \cite{r6}, Theorem 37. 
%
% additional part starts
%
% ............................................................
%
Existence of solutions to {\it MP}($\tilde{\SA},\nu$) is standard and the assumptions on the coefficients $\{ \tilde{b}_k \}$ ensure solutions remain in $\SSS(x^0)$. Hence, we only need consider uniqueness. By conditioning we may assume $\nu = \delta_x, x \in \SSS(x^0)$ (see p. 136 of [B]). By Krylov's Markov selection theorem it suffices to show uniqueness of a strong Markov family $\{ P^{x'}, x' \in \SSS(x^0) \}$ of solutions to {\it MP}($\tilde{\SA},\delta_x$) (see the proof of Proposition 2.1 in [ABBP]). Let $(\tilde{R}_{\lambda}, \lambda > 0)$ be the associated resolvent operators. 
%
% ............................................................
%
\begin{lem} \label{equ:lem_38}
For $f \in \SC_w^{\alpha}$, $\tilde{R}_{\lambda} f = R_{\lambda} f + \tilde{R}_{\lambda} \SB R_{\lambda} f$. 
\end{lem}
%
% ............................................................
%
\begin{proof}
An easy application of Fatou's Lemma shows that \\ $\tilde{E}_x(x_t^{(i)}) \leq x_i + \norm{ \tilde{b}_i } t$ for all $i \in N_{C2}$ (recall these coordinates are non-negative). This implies the square functions of the martingale part of each coordinate are integrable. It follows that for $g \in \SC_b^2(\SSS(x^0))$, 
\[ M_g(t) \equiv g(x_t) - g(x_0) - \int_0^t \tilde{\SA} g(x_s) ds \]
is a martingale and so
\[ \tilde{E}_x(g(x_t)) = g(x) + \int_0^t \tilde{E}_x \left( \tilde{\SA} g (x_s) \right) ds. \]
Multiply by $\lambda e^{- \lambda t}$ and integrate over $t \geq 0$ to see that for $g \in \SC_b^2(\SSS(x^0))$,
\begin{equation} \label{equ:onehundredfortysix}
  \lambda \tilde{R}_{\lambda} g = g + \tilde{R}_{\lambda} (\tilde{\SA} g) = g + \tilde{R}_{\lambda} ( \SB g ) + \tilde{R}_{\lambda} ( \SA^0 g ).
\end{equation}
Let $f \in \SC_w^{\alpha}$ and for $\delta>0$, set $g_{\delta}(x) \equiv \int_{\delta}^{\infty} e^{- \lambda t} P_t f (x) dt$. Proposition \ref{equ:proposition_14} implies that $g_{\delta} \in \SC_b^2(\SSS(x^0))$. Moreover using the bounds (\ref{equ:prop30_a1}) and (\ref{equ:prop30_a2}) it is easy to verify that for $k \in V$,
\begin{equation} \label{equ:onehundredfortyseven}
  (g_{\delta})_k(x) \rightarrow (R_{\lambda} f)_k(x) \mbox{ as } \delta \downarrow 0 \mbox{ uniformly in } x \in \SSS(x^0),
\end{equation}
for $j \in N_R, i \in C_j$
\begin{equation} \label{equ:onehundredfortyeight}
  x_i (g_{\delta})_{jj} (x) \rightarrow x_i (R_{\lambda} f)_{jj} (x) \mbox{ as } \delta \downarrow 0 \mbox{ uniformly in } x \in \SSS(x^0), 
\end{equation}
and for $j \in N_{C2}$
\begin{equation} \label{equ:onehundredfortynine}
  x_j (g_{\delta})_{jj} (x) \rightarrow x_j (R_{\lambda} f)_{jj} (x) \mbox{ as } \delta \downarrow 0 \mbox{ uniformly in } x \in \SSS(x^0). 
\end{equation}
Since $\{ \tilde{b}_k \}, \{ \tilde{\gamma}_k \}$ are bounded, (\ref{equ:onehundredfortyseven}), (\ref{equ:onehundredfortyeight}), (\ref{equ:onehundredfortynine}) imply that
\begin{equation} \label{equ:onehundredfifty}
  \SB g_{\delta} \rightarrow \SB R_{\lambda} f \mbox{ as } \delta \downarrow 0 \mbox{ uniformly on } \SSS(x^0). 
\end{equation}

An easy calculation using $\dot{P}_t g_{\delta} = P_t \SA^0 g_{\delta} \rightarrow \SA^0 g_{\delta}$ as $t \downarrow 0$ shows that
\begin{equation} \label{equ:onehundredfiftyone}
  \SA^0 g_{\delta} = \lambda g_{\delta} - e^{- \lambda \delta} P_{\delta} f \rightarrow \lambda R_{\lambda} f - f \mbox{ uniformly on } \SSS(x^0) \mbox{ as } \delta \downarrow 0.
\end{equation}

Now set $g = g_{\delta}$ in (\ref{equ:onehundredfortysix}) and use (\ref{equ:onehundredfifty}), (\ref{equ:onehundredfiftyone}) and the obvious uniform convergence of $g_{\delta}$ to $R_{\lambda} f$ to see that
\[ \lambda \tilde{R}_{\lambda} (R_{\lambda} f) = R_{\lambda} f + \tilde{R}_{\lambda} ( \SB R_{\lambda} f ) + \tilde{R}_{\lambda} ( \lambda R_{\lambda} f - f ). \]
Rearranging, we get the required result.
\end{proof}
%
% ............................................................
%
Continuing with the proof of Theorem \ref{equ:thm_36_37}(b), note that the H\"older continuity of $\tilde{\gamma}_k$ and $\tilde{b}_k$ and the fact that they are constant outside a compact set imply $\tilde{\Gamma} < \infty$. Therefore we may choose $\lambda_1$ as in part (a) of Theorem \ref{equ:thm_36_37} so that for $\lambda \geq \lambda_1$, $\SB R_{\lambda}: \SC_w^{\alpha} \rightarrow \SC_w^{\alpha}$ with norm at most $1/2$. If $f \in \SC_w^{\alpha}$, we may iterate Lemma \ref{equ:lem_38} to see that
\[ \tilde{R}_{\lambda} f (x) = \sum \limits_{n=0}^{\infty} R_{\lambda} (( \SB R_{\lambda} )^n f )(x), \]
where the series converges uniformly and the error term approaches zero by the bound
\[ \norm{ (\SB R_{\lambda})^n f } \leq \normw{ (\SB R_{\lambda})^n f } \leq 2^{-n} \normw{f}. \]
This shows that for all $f \in \SC_w^{\alpha}, \tilde{R}_{\lambda} f (x)$ is unique for $\lambda \geq \lambda_1$ and hence so is $\tilde{P}_t f (x) = \tilde{E}_x (f (x_t) )$ for all $t \geq 0$. As $\SC_w^{\alpha}$ is measure determining, uniqueness of $\tilde{P}_t(x,dy)$ and hence $\tilde{P}^x$ follows. 
%
% ............................................................
%
% additional part ends
%
\end{proof}
%
% ------------------------------------------------------------
%

{\it Continuation of the proof of Theorem \ref{equ:theorem_4}.} Recall ``Step 1: Reduction of the problem'', in Subsection \ref{equ:subsection_1_3}. The remainder of the proof of uniqueness of $MP(\SA,\delta_{x^0})$ works analogously to \cite{r6} (compare the proof of Theorem 4 on pp. 380-382 in \cite{r6}) except for minor changes, making again use of Lemma \ref{equ:lemma_5}. The main step consists in using a localization argument of \cite{r11} (see e.g. the argument in the proof of Theorem 1.2 of \cite{r4}), which basically states that it is enough if for each $x^0 \in \SSS$ the martingale problem $MP(\tilde{\SA},\delta_{x^0})$ has a unique solution, where $b_i = \tilde{b}_i$ and $\gamma_i = \tilde{\gamma}_i$ agree on some neighborhood of $x^0$. By comparing the definition of $\SA$ (see (\ref{equ:recall_def_SA})) and $\tilde{\SA}$ (see (\ref{equ:onehundredfortythree})) one chooses
\begin{align*}
  & \tilde{b}_k(x) = b_k(x) \mbox{  for all } k \in V, \\
  & \tilde{\gamma}_j(x) = x_j \gamma_j(x) \mbox{ for } j \in N_R, \\
  & \tilde{\gamma}_j(x) = \left( \sum \limits_{i \in C_j} x_i \right) \gamma_j(x) \mbox{ for } j \in R \backslash N_R \\
  & \tilde{\gamma}_j(x) = \gamma_j(x) \mbox{ for } j \notin R.
\end{align*}
By setting
\[ b_k^0 \equiv \tilde{b}_k(x^0) \mbox{  and  } \gamma_k^0 \equiv \tilde{\gamma}_k(x^0) \]
and choosing $\tilde{b}_k$ and $\tilde{\gamma}_k$ in appropriate ways, the assumptions of Theorem \ref{equ:thm_36_37}(a), (b) will be satisfied in case $b_k^0 \geq 0$ for all $k \in N_2$ (and hence by Hypothesis \ref{equ:hypothesis_2} for all $k \in N_{C2}$). In particular the boundedness and continuity of the coefficients of $\tilde{\SA}$ will allow us to choose $\tilde{\epsilon}_0$ arbitrarily small. In case there exists $k \in N_2$ such that $b_k^0 < 0$ a Girsanov argument as in the proof of Theorem 1.2 of \cite{r4} allows the reduction of the latter case to the former case amd thus finishes the proof.
%

% additional part starts
%
% ............................................................
%
Indeed, we can adapt the proof of Theorem 1.2 in \cite{r4} to our setup. We have
\[ \tilde{\SA} = \sum_{j \in V} \tilde{b}_j(x) \dx{j} + \sum_{j \in N_R} \tilde{\gamma}_j(x) \left( \sum_{i \in C_j} x_i \right) \ddx{j} + \sum_{i \in N_{C2}} \tilde{\gamma}_i(x) x_i \ddx{i} \]
with $\tilde{\gamma}_i(x^0) > 0$ and investigate the case $\tilde{b}_j(x^0) = b_j(x^0) < 0$ for some $j \in N_2$. Note that it is essential for the proof that $j \in N_2$.

By relabeling the axes if necessary, let us choose $0 \leq m \leq |N_2|$ such that $\{ 1,\ldots,m \} \subset N_2$ and assume $x_i^0=0$ for $i \leq m$. By Hypothesis \ref{equ:hypothesis_2},
\[ \tilde{b}_i(x^0) = b_i(x^0) \geq 0 \mbox{ for all } i \leq m. \]
If $m=|N_2|$, we are in {\it Case 1} of the proof of Theorem 4 in \cite{r6}, i.e. where $b_k^0 \geq 0$ for all $k \in N_2$ ; so we assume $m<|N_2|$.

Set $a_{ii}(x) \equiv \tilde{\gamma}_i(x) x_i \geq 0, i \leq m$ and let $r>0$ such that 
\[ a_{ii}(x) \equiv \tilde{\gamma}_i(x) x_i \geq \epsilon > 0, \ m < i \leq |N_2| \]
in $B(x^0,2r) \cap \RR_+^d$. Define $\mu_i(x) \equiv 0$ for $i \leq m$ and
\[ \mu_i(x) = \rho_r(|x-x^0|) a_{ii}(x)^{-1} \tilde{b}_i(x), \ m < i \leq |N_2|, \ x \in \RR_+^d, \]
where $\rho_r:[0,\infty) \rightarrow [0,1]$ is the function depicted in Figure \ref{equ:graph_rho}.
\picturefig{0.5}{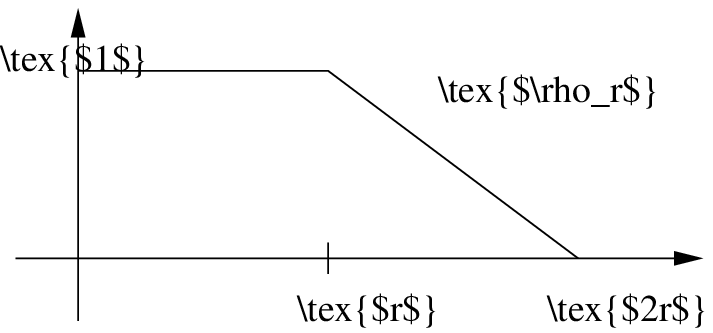}{Definition of $\rho_r$.}{graph_rho}

Then $\mu$ is bounded and continuous on the positive orthant. If
\[ \bar{b}_i(x) \equiv \left( \1 - \1(m<i\leq|N_2|) \right) \tilde{b}_i(x), \]
then $\bar{b}_i(x^0) \geq 0$ for all $i \in N_2$ and so $\bar{b}$ satisfies {\it Case 1} of the proof of Theorem 4 in \cite{r6} and thus a unique solution exists for $MP(\bar{\SL},\delta_z)$.

 Now set
\[ \hat{b}_i(x) \equiv \bar{b}_i(x) + a_{ii}(x) \mu_i(x) \]
and note that
\[ \mbox{ if } x_i^0=0, \mbox{ then } \hat{b}_i(x^0) = \bar{b}_i(x^0) \geq 0. \]
The existence of a solution to $MP(\hat{\SL},\delta_z)$ is again standard. Girsanov's theorem (see V.27 in [RW]) and the fact that there is a unique solution to $MP(\bar{\SL},\delta_z)$ shows that the solution to $MP(\hat{\SL},\delta_z)$ is unique. Here note that although $a_{ii}(x)$ is unbounded, one can still apply Girsanov's theorem to show that the law of $P(X(\cdot \wedge T_R) \in \cdot)$ is unique where $T_R$ is the exit time from $[0,R]^d$ and this gives the result. Note also that Girsanov's theorem applies without change in our $\RR_+^d$-valued setting.

Finally, if $x \in B(x^0,r) \cap \RR_+^d$, then
\[ \hat{b}_i(x) = \tilde{b}_i(x). \]
This completes the proof of this part. 
%
% ............................................................
%
% additional part ends
%
\end{proof}
%
%
% ============================================================
% ============================================================
%
\section*{Acknowledgement}
I would like to thank my Ph.D. supervisor Ed Perkins for suggesting this problem to me and for providing helpful comments and explanations. Further thanks go to a referee for a careful reading of the paper and a number of helpful suggestions.
%
% ============================================================
% ============================================================

% ============================================================

\end{document}